\documentclass[12pt]{article}
\usepackage[T1]{fontenc}
\usepackage[utf8]{inputenc}
\usepackage[margin=1in]{geometry}
\usepackage[toc,page]{appendix}

\usepackage{amssymb}
\usepackage{tabularx}
\usepackage{tikz}
\usetikzlibrary{patterns,hobby}
\usepackage{amsmath}
\usepackage{algorithm}
\usepackage{algpseudocode}
\usepackage{mathtools}
\usepackage{enumerate}

\usepackage{url}
\usepackage{hyperref}
\hypersetup{colorlinks,allcolors=[rgb]{0,0.13672,0.95703}}
\usepackage{graphicx}
\usepackage{fullpage}
\usepackage{subfig}
\usepackage{wrapfig}
\usepackage{color,soul}
\usepackage{float}
\usepackage{setspace}
\usepackage{tikz,pgfplots}
\pgfplotsset{compat=newest}
\usepackage{etoolbox}
\usepackage{comment}
\ifx\assumption\undefined
\newtheorem{assumption}{Assumption}
\fi

\def\shift{0mm} 
\def\ww{3in} 
\def\toyshift{0mm} 
\newif\ifdocenter
\docentertrue

\title{Second-order Information Promotes Mini-Batch Robustness in Variance-Reduced Gradients}
\author{Sachin Garg\thanks{Department of Electrical Engineering \& Computer Science, University of Michigan, \texttt{sachg@umich.edu}} \and Albert S. Berahas\thanks{Department of Industrial \& Operations Engineering, University of Michigan, \texttt{albertberahas@gmail.com}} \and  Micha{\l} Derezi{\'n}ski\thanks{Department of Electrical Engineering \& Computer Science, University of Michigan, \texttt{derezin@umich.edu}} }
\date{}

\def\qed{\hfill$\blacksquare$\medskip}

\def\xib{\boldsymbol\xi}

\def\Sigmab{\mathbf{\Sigma}}

\def\g {\mathbf{g}}

\def\Y{\mathbf Y}
\def\R{\mathbf R}

\ifx\BlackBox\undefined
\newcommand{\BlackBox}{\rule{1.5ex}{1.5ex}}  
\fi
\DeclareMathOperator*{\argmin}{\mathop{\mathrm{argmin}}}

\def\x{\mathbf x}

\def\y{\mathbf y}

\def\z{\mathbf z}

\def\X{\mathbf X}

\def\A{\mathbf A}

\def\U{\mathbf U}

\def\V{\mathbf V}

\def\M{\mathbf M}

\def\I{\mathbf I}

\def\A{\mathbf A}

\def\E{\mathbb E}

\def\R{\mathbb R} 

\let\origtop\top
\renewcommand\top{{\scriptscriptstyle{\origtop}}} 

\definecolor{silver}{cmyk}{0,0,0,0.3}
\definecolor{yellow}{cmyk}{0,0,0.9,0.0}
\definecolor{reddishyellow}{cmyk}{0,0.22,1.0,0.0}
\definecolor{black}{cmyk}{0,0,0.0,1.0}
\definecolor{darkYellow}{cmyk}{0.2,0.4,1.0,0}
\definecolor{darkSilver}{cmyk}{0,0,0,0.1}
\definecolor{grey}{cmyk}{0,0,0,0.5}
\definecolor{darkgreen}{cmyk}{0.6,0,0.8,0}
\newcommand{\Red}[1]{{\color{red}  {#1}}}

\newcommand{\Green}[1]{{\color{darkgreen}  {#1}}}
\newcommand{\Blue}[1]{\color{blue}{#1}\color{black}}
\newcommand{\Brown}[1]{{\color{brown}{#1}\color{black}}}

\ifx\proof\undefined
\newenvironment{proof}{\par\noindent{\bf Proof\ }}{\hfill\BlackBox\\[2mm]}
\fi

\ifx\theorem\undefined
\newtheorem{theorem}{Theorem}
\fi

\ifx\example\undefined
\newtheorem{example}{Example}
\fi

\ifx\property\undefined
\newtheorem{property}{Property}
\fi

\ifx\lemma\undefined
\newtheorem{lemma}{Lemma}
\fi

\ifx\proposition\undefined
\newtheorem{proposition}{Proposition}
\fi

\ifx\remark\undefined
\newtheorem{remark}{Remark}
\fi

\ifx\corollary\undefined
\newtheorem{corollary}{Corollary}
\fi

\ifx\definition\undefined
\newtheorem{definition}{Definition}
\fi

\ifx\conjecture\undefined
\newtheorem{conjecture}[theorem]{Conjecture}
\fi

\ifx\axiom\undefined
\newtheorem{axiom}[theorem]{Axiom}
\fi

\ifx\claim\undefined
\newtheorem{claim}[theorem]{Claim}
\fi

\ifx\assumption\undefined
\newtheorem{assumption}{Assumption}
\fi

\newcommand{\norm}[1]{\left\lVert#1\right\rVert}

\newcommand{\normH}[1]{\left\lVert#1\right\rVert_{\mathbf{H}}}

\newcommand{\normHsq}[1]{\left\lVert#1\right\rVert_{\mathbf{H}}^2}

\newcommand{\nsq}[1]{\|#1\|^{2}}
\newcommand{\T}[1]{\Tilde{#1}}

\newcommand{\vi}{\mathbf v}
\newcommand{\ai}{\mathbf a}
\newcommand{\Hi}{\mathbf H}

\newcommand\ind{\boldsymbol{1}}
\newcommand{\mycirc}[1][black]{\textcolor{#1}{\ensuremath\bullet}}

\begin{document}

\maketitle
\begin{abstract}
We show that, for  
finite-sum minimization problems, incorporating partial second-order information of the objective function can dramatically improve the robustness to mini-batch size  
of variance-reduced stochastic gradient methods, 
making them more scalable while retaining their benefits over traditional Newton-type approaches. We demonstrate this phenomenon on a prototypical stochastic second-order algorithm, called Mini-Batch Stochastic Variance-Reduced Newton (\texttt{Mb-SVRN}), which combines variance-reduced gradient estimates with access to an approximate Hessian oracle. In particular, we show that when the data size $n$ is sufficiently large, i.e., $n\gg \alpha^2\kappa$, where $\kappa$ is the condition number and $\alpha$ is the Hessian approximation factor, then \texttt{Mb-SVRN} achieves a fast linear convergence rate that is independent of the gradient mini-batch size $b$, as long $b$ is in the range between $1$ and $b_{\max}=O(n/(\alpha \log n))$. Only after increasing the mini-batch size past this critical point $b_{\max}$, the method begins to transition into a standard Newton-type algorithm which is much more sensitive to the Hessian approximation quality. We demonstrate this phenomenon empirically on benchmark optimization tasks showing that, after tuning the step size, the convergence rate of \texttt{Mb-SVRN} remains fast for a wide range of mini-batch sizes, and the dependence of the phase transition point $b_{\max}$ on the Hessian approximation factor $\alpha$ aligns with our theoretical predictions.

\end{abstract}

\section{Introduction}

Consider the following finite-sum convex minimization problem:
\begin{align}
    f(\x) &= \frac{1}{n}\sum_{i=1}^{n}\psi_i(\x),  
    \qquad\text{with}  \qquad f(\x^*) = \argmin f(\x).  \label{main}
\end{align}
The finite-sum formulation given in \eqref{main} can be used to express many empirical minimization tasks, where for a choice of the underlying parameter vector $\x \in \R^d$, each $\psi_i : \R^d \rightarrow \R$ models the loss on a particular observation \cite{sharpe1989mean,rigollet2011neyman,xiao2014proximal}. In modern-day settings, it is common to encounter problems with a very large number of observations $n$, making deterministic optimization methods prohibitively expensive \cite{bottou2018optimization}. In particular, we are interested in and investigate problems where $n$ is much larger than the condition number of the problem, $\kappa$. In this regime ($n \gg \kappa$), our aim is to develop a stochastic algorithm with guarantees of returning an $\epsilon$-approximate solution i.e., $\T{\x}$ such that $f(\T{\x}) - f(\x^*) \leq \epsilon$ with high probability. As is commonly assumed in the convex optimization literature, we consider the problem setting where $f$ is $\mu$-strongly convex and each $\psi_i$ is $\lambda$-smooth.  

A prototypical optimization method for \eqref{main} is stochastic gradient descent (SGD) \cite{robbins1951stochastic}, which relies on computing a gradient estimate based on randomly sampling a single component $\psi_i$, or more generally, a subset of components (i.e., a mini-batch).
Unfortunately, SGD with constant step size does not converge to the optimal solution $\x^*$, but only to a neighborhood around $\x^*$ that depends on the step size and variance in the stochastic gradient approximations. 
A slower sub-linear convergence rate to the optimal solution can be guaranteed with diminishing step sizes. To mitigate this deficiency of SGD, several works have explored variance reduction techniques for first-order stochastic methods, e.g., \texttt{SDCA}~\cite{shalev2013stochastic}, \texttt{SAG}~\cite{roux2012stochastic}, \texttt{SVRG}~\cite{johnson2013accelerating}, \texttt{Katyusha}~\cite{allen2017katyusha} and \texttt{S2GD}~\cite{konevcny2013semi}, and have derived faster linear convergence rates to the optimal solution. 

One of the popular first-order stochastic optimization methods that employs variance reduction is Stochastic Variance Reduced Gradient (\texttt{SVRG})~\cite{johnson2013accelerating}. The method has two stages and operates with inner and outer loops. At every inner iteration, the method randomly selects a component gradient, $\nabla{\psi}_i(\x)$, and couples this with an estimate of the true gradient computed at every outer iteration to compute a step. This combination, and the periodic computation of the true gradient at outer iterations, leads to a reduction in the variance of the stochastic gradients employed, and allows for global linear convergence guarantees. 
The baseline \texttt{SVRG} method (and its associated analysis) assumes only one observation is sampled at every inner iteration, and as a result, requires $\mathcal{O}(n)$ inner iterations in order to achieve the best convergence rate per data pass. This makes baseline \texttt{SVRG} inherently and highly sequential, and unable to take advantage of modern-day massively parallel computing environments. A natural remedy is to use larger mini-batches (gradient samples sizes; denoted as $b$) at every inner iteration. Unfortunately, this natural idea does not have the desired advantages as increasing the mini-batch size $b$ in \texttt{SVRG} leads to deterioration in the convergence rate per data pass. 
In particular, one can show that regardless of the chosen mini-batch size $b$, \texttt{SVRG} requires as many as $\mathcal{O}(\kappa)$ inner iterations to ensure fast convergence rate,

again rendering the method unable to take advantage of parallelization. 

On the other hand, one can use stochastic second-order methods for solving \eqref{main}, including Subsampled Newton ~\cite{roosta2019sub,bollapragada2019exact,erdogdu2015convergence}, Newton Sketch ~\cite{pilanci2017newton,berahas2020investigation}, and others \cite{moritz2016linearly,mokhtari2018iqn,derezinski2018batch,berahas2016multi}. These methods use either full gradients or require extremely large gradient mini-batches $b \gg \kappa$ at every iteration \cite{derezinski2022stochastic}, and as a result, their convergence rate (per data pass) is sensitive to the quality of the Hessian estimate and to the mini-batch size $b$. Several works \cite{gonen2016solving,gower2016stochastic,liu2019acceleration,moritz2016linearly} have also explored the benefits of including second-order information to accelerate first-order stochastic variance-reduced methods. However, these are primarily first-order stochastic methods (and analyzed as such) resulting again in highly sequential algorithms that are unable to exploit modern parallel computing frameworks. 

The shortcomings of the aforementioned stochastic methods  raise a natural question, which is central to our work:
\begin{center}
    \emph{Can second-order information plus variance-reduction achieve an accelerated and robust convergence rate for a wide range of gradient mini-batch sizes?}
\end{center}
\smallskip

In this work, we provide an affirmative answer to the above question. To this end, following recent prior work \cite{derezinski2022stochastic}, we analyze the convergence behavior of a prototypical stochastic optimization algorithm called Mini-batch Stochastic Variance-Reduced Newton (\texttt{Mb-SVRN}, Algorithm \ref{alg}), which combines variance reduction in the stochastic gradient, via a scheme based on SVRG \cite{johnson2013accelerating}, with second-order information from a Hessian oracle. Our Hessian oracle is general and returns an estimate that satisfies a relatively mild $\alpha$-approximation condition. This condition can be, for instance, satisfied by computing the Hessian of the subsampled objective (as in Subsampled Newton \cite{roosta2019sub}, which is what we use in our experiments), but can also be satisfied by other Hessian approximation techniques such as sketching. 

\begin{figure}
    \includegraphics[width=0.95\textwidth,clip=true]{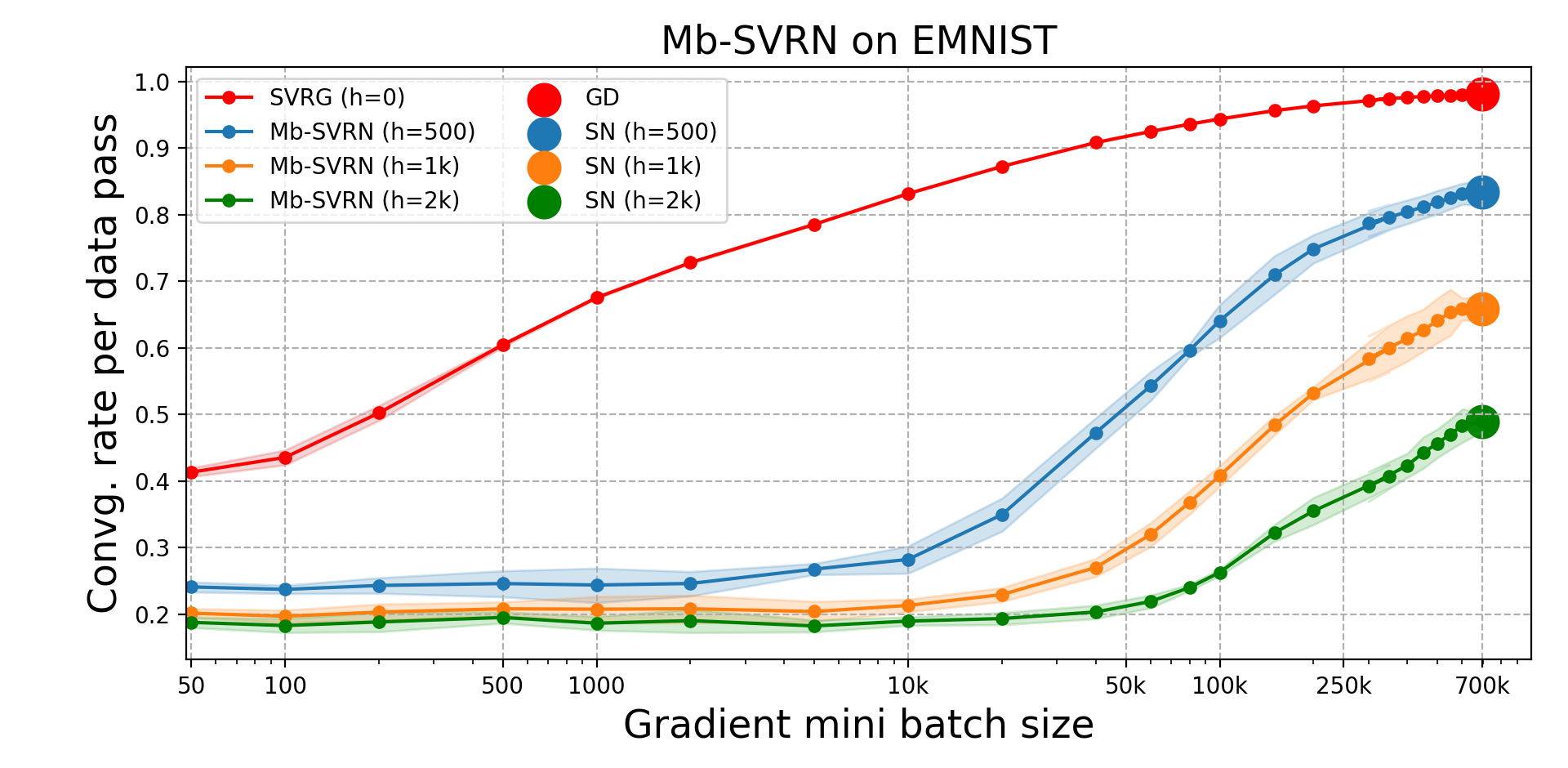}
    \caption{Convergence rate of \texttt{Mb-SVRN} (smaller is better, see Section \ref{Experiments}), as we vary gradient mini-batch size $b$ and Hessian sample size $h$, including the extreme cases of \texttt{SN} ($b\!=\!n$) and \texttt{SVRG} ($h\!=\!0$). The plot shows that after adding some second-order information (increasing~$h$), the convergence rate of \texttt{Mb-SVRN} quickly becomes robust to gradient mini-batch size. On the other hand, the performance of \texttt{SVRG} rapidly degrades as we increase the gradient mini-batch size $b$, which ultimately turns it into simple gradient descent (\texttt{GD}).
    }
 	\label{fig_1}
\end{figure}

In our main result (Theorem \ref{main_result}) we show that even a relatively weak Hessian estimate leads to a dramatic improvement in robustness to the gradient mini-batch size for \texttt{Mb-SVRN}: the method is endowed with \emph{the same} fast local linear convergence guarantee \emph{for any} gradient mini-batch size up to a critical point $b_{\max}$, which we characterize in terms of the number of function components $n$ and the Hessian oracle approximation factor $\alpha$. 
Remarkably, unlike results for most stochastic gradient methods, our result holds with high probability rather than merely in expectation, both for large and small mini-batch sizes, and is an improvement over comparable results \cite{derezinski2022stochastic} both in convergence rate and robustness to mini-batch sizes (see Section \ref{related_work}). As $b$ increases beyond $b_{\max}$, the method has to inevitably transition into a standard Newton-type algorithm with a weaker convergence rate that depends much more heavily on the Hessian approximation factor $\alpha$. We demonstrate this phenomenon empirically on the Logistic Regression task on the \texttt{EMNIST} and \texttt{CIFAR10} datasets (see Figure \ref{fig_1} and Section \ref{Experiments}); we show 
that the convergence rate of \texttt{Mb-SVRN} indeed exhibits robustness to 
mini-batch size $b$, while at the same time having a substantially better convergence rate per data pass than the Subsampled Newton method (\texttt{SN}) that uses full gradients. Furthermore, \texttt{Mb-SVRN} proves to be remarkably robust to the Hessian approximation quality in our experiments, showing that even low-quality Hessian estimates can significantly improve the scalability of the method.
\paragraph{Outline.}
In Section \ref{contribution}, we provide an informal version of our main result along with a comparison with the popular \texttt{SVRG} and \texttt{SN} methods.  
In Section \ref{related_work}, we give a detailed overview of the related work in stochastic variance reduction and second-order methods. 
In  Section \ref{technical_analysis} we present our technical analysis for the local convergence of \texttt{Mb-SVRN}, whereas in Theorem \ref{global} we provide a global convergence guarantee. We provide experimental evidence in agreement with our theoretical results in Section \ref{Experiments}.

\section{Main Convergence Result} \label{contribution}
In this section, we provide an informal version of our main result, Theorem \ref{final_result}. We start by formalizing the assumptions and algorithmic setup. 
\begin{assumption}{\label{assumption}}
Consider a twice continuously differentiable function $f: \R^d \rightarrow \R$ defined as in \eqref{main}, where $\psi_i: \R^d \rightarrow \R$ and $i \in \{1,2,...,n\}$. We make the following standard $\lambda$-smoothness and $\mu$-strong convexity assumptions, where $\kappa=\lambda/\mu$ denotes the condition number.
\begin{enumerate} 
    \item For each $i \in \{1,2,...,n\}$, $\psi_i$ is a $\lambda$-smooth convex function:
    \begin{align*}
        \psi_i(\y) \leq \psi_i(\x) + \nabla{\psi}_i(\x)^\top(\y-\x) + \frac{\lambda}{2}\nsq{\y-\x}, \ \  \ \forall \ \x,\y \in \mathbb{R}^d,
    \end{align*}
    and $f$ is a $\mu$-strongly convex function:
    \begin{align*}
         f(\y) \geq f(\x) + \nabla{f}(\x)^\top(\y-\x) + \frac{\mu }{2}\nsq{\y-\x}, \ \  \ \forall \ \x,\y \in \mathbb{R}^d.    
         \end{align*}
\end{enumerate}
We also assume that the Hessians of $f$ are Lipschitz continuous.
\begin{enumerate} 
    \item[2.] The Hessians $\nabla{f}^2(\x)$ of the objective function are $L$-Lipschitz continuous:
    \begin{align*}
        \norm{\nabla{f}^2(\x) - \nabla{f}^2(\y)} \leq L\norm{\x-\y},  \ \  \ \forall \ \x,\y \in \mathbb{R}^d.
    \end{align*}
\end{enumerate}
\end{assumption}

In our computational setup, we assume that the algorithm has access to a stochastic gradient and Hessian at any point via the following gradient and Hessian oracles. 
\begin{definition}[Gradient oracle]
    Given $1 \leq b \leq n$ and indices $i_1,...,i_b\in\{1,...,n\}$, the gradient oracle $\mathcal G$ returns $\hat\g\sim \mathcal G(\{i_1,...,i_b\})$ such that:
\begin{align*}
    \hat\g(\x) = \frac1b\sum_{j=1}^b\nabla\psi_{i_{j}}(\x),\qquad\text{for any $\x\in\R^d.$}
\end{align*}
\end{definition}
We formalize this notion of a gradient oracle to highlight that the algorithm only accesses component gradients in batches, enabling hardware acceleration for the gradient computations, which is one of the main motivations of this work. We also use the gradient oracle to measure the per-data-pass convergence rate of the algorithms. Here, a data pass refers to a sequence of gradient oracle queries that access $n$ component gradients.
\begin{definition}[Hessian oracle]\label{d:hessian-oracle}
For an $\alpha\geq 1$, and any $\x\in\R^d$, the Hessian oracle $\mathcal{H}=\mathcal{H}_{\alpha}$ returns 
$\hat\Hi \sim\mathcal{H}(\x)$ such that:
\begin{align*}
    \frac{1}{\sqrt{\alpha}} \nabla^2{f}(\x) \preceq \hat{\Hi} \preceq \sqrt{\alpha} 
    \nabla^2{f}(\x).
\end{align*}
\end{definition}

 We refer to $\alpha$ as the Hessian approximation factor and denote the approximation guarantee as $\hat{\Hi} \approx_{\sqrt{\alpha}} \nabla^2{f}(\x)$. Here, we mention that an approximate Hessian, $\hat{\Hi}$ satisfying the above guarantee can be constructed with probability $1-\delta$ for any $\delta>0$. We deliberately do not specify the method for 
 obtaining the Hessian estimate, although 
 in the numerical experiments we will use a sub-sampled Hessian estimate of the form $\hat\Hi = \frac1h\sum_{j}^h\nabla^2\psi_{i_j}(\x)$, where the sample size $1 \leq h \leq n$ controls the quality of the Hessian approximation (e.g., as in Figure \ref{fig_1}). We treat the cost of Hessian approximation separate from the gradient data passes, as the complexity of Hessian estimation is largely problem and method dependent. Naturally, from a computational stand-point, larger $\alpha$ generally means smaller Hessian estimation cost.
 \subsection{Main algorithm and result}
In this section, we present an 
algorithm that follows the above computational model (Algorithm \ref{alg}, \texttt{Mb-SVRN}), and provide our main technical result, the local convergence analysis of this algorithm across different values of the Hessian approximation factor $\alpha$ and gradient mini-batch sizes $b$. We note that the algorithm was deliberately chosen as a natural extension of both a standard stochastic variance-reduced first-order method (\texttt{SVRG}) and a standard Stochastic Newton-type method (\texttt{SN}), so that we can explore the effect of variance reduction in conjunction with 
second-order information on the convergence rate and robustness.

We now informally state our main result, which is a local convergence guarantee for Algorithm \ref{alg} (for completeness, we also provide a global convergence guarantee in Theorem \ref{global}). In this result, we show a high-probability  convergence bound, that holds for gradient mini-batch sizes $b$ as small as $1$ and as large as $\mathcal{O}(\frac{n}{\alpha\log n})$, where the fast linear rate of convergence is independent of the mini-batch size.

\begin{algorithm}
\caption{Mini-batch Stochastic Variance-Reduced Newton (\texttt{Mb-SVRN})}\label{alg}
\begin{algorithmic}
\Require $\Tilde{\x}_0$, gradient mini-batch size $b$, gradient/Hessian oracles $\mathcal G, \mathcal H$, inner iterations $t_{\max}$
\For{$s=0,1,2,\cdots,$}
\State Compute the Hessian estimate $\hat{\Hi}_s \sim \mathcal{H}(\tilde\x_s)$ 
\State Compute the full gradient $\g_s = \g(\tilde\x_s)$ \ for \ $\g \sim \mathcal G(\{1,...,n\})$
\State Set $\x_{0,s} = \Tilde{\x}_s$
\For{$t=0,1,2,\cdots,t_{\max}-1$}
\State Compute $\hat\g\!\sim\! \mathcal G(\{i_1,...,i_b\})$ for $i_1,...,i_b\!\sim \!\{1,...,n\}$ uniformly random
\State Compute $\hat{\g}_{t,s}= \hat\g(\x_{t,s})$ and $\hat\g_{0,s}=\hat\g(\tilde\x_s)$ 
\State Compute variance-reduced gradient $\bar{\g}_{t,s} = \hat{\g}_{t,s} - \hat{\g}_{0,s} + \g_s$
\State Update $\x_{t+1,s} = \x_{t,s} - \eta\hat{\Hi}_s^{-1}\bar{\g}_{t,s} $
\EndFor 
\State $\Tilde{\x}_{s+1} = \x_{t_{max},s}$
\EndFor
\end{algorithmic}
\end{algorithm}

\begin{theorem}[Main result; informal Theorem \ref{final_result}]{\label{main_result}}
Suppose that assumption \ref{assumption} holds and $n\gtrsim \alpha^2\kappa$. Then, in a local neighborhood around $\x^*$, Algorithm \ref{alg} using Hessian $\alpha$-approximations and any gradient mini-batch size $b\lesssim \frac n{\alpha\log n}$, with $t_{\max}=n/b$ and optimally chosen $\eta$, satisfies the following high-probability convergence bound:
\begin{align*}
    f(\T\x_s) - f(\x^*) \leq \rho^s\cdot\left(f(\T\x_0) - f(\x^*)\right),\quad\text{where}\quad \rho\lesssim \frac{\alpha^2\kappa}{n}.
\end{align*}
\end{theorem}
\begin{remark}
    Under our assumption that $n\gtrsim \alpha^2\kappa$, the convergence rate satisfies $\rho\ll 1/2$, i.e., it is a fast condition-number-free 
    linear rate of convergence that gets better for larger data sizes $n$. Since we used $t_{\max}=n/b$, one outer iteration of the algorithm corresponds to roughly two passes over the data (as measured by the gradient oracle calls).
\end{remark}

\subsection{Discussion}\label{s:discussion}
Our convergence analysis in Theorem \ref{main_result} shows that incorporating second-order information into a stochastic variance-reduced gradient method makes it robust to increasing the gradient mini-batch size $b$ up to the point where $b=\mathcal{O}\big(\frac n{\alpha\log n}\big)$. This is illustrated in Figure \ref{fig_toy} (compare to the empirical Figure \ref{fig_1}), where the robust regime corresponds to the convergence rate $\rho$ staying flat as we vary $b$. Improving the Hessian oracle quality (smaller $\alpha$) expands the robust regime, allowing for even larger mini-batch sizes with a flat convergence rate profile. 

\begin{figure}
    \includegraphics[width=.95\textwidth,clip=true]{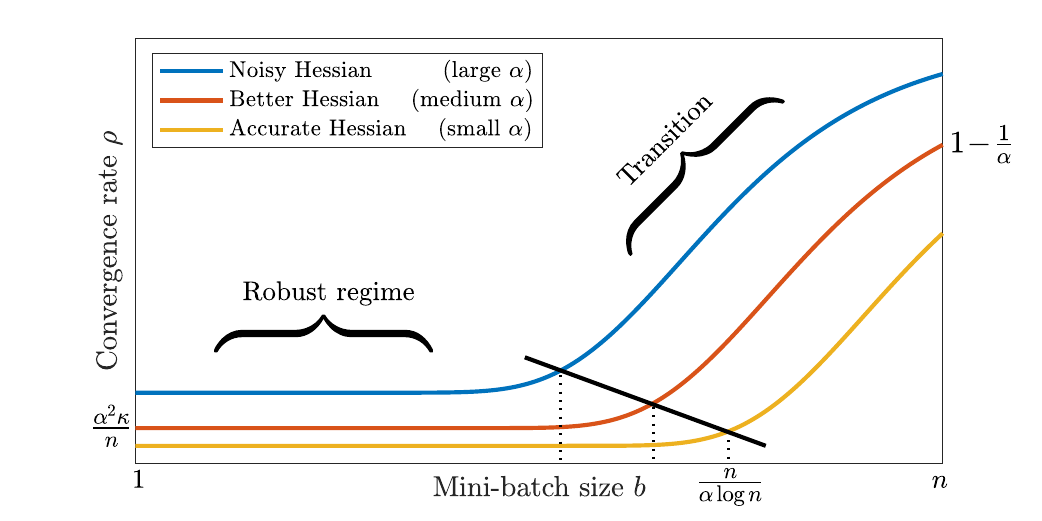}
    \caption{Illustration of the \texttt{Mb-SVRN} convergence analysis from Theorem \ref{main_result} (for $n\gtrsim \alpha^2\kappa$), showing how the regime of robustness to gradient mini-batch size $b$ depends on the quality of the Hessian oracle (smaller $\alpha$ means better Hessian estimate). As we increase $b$ past $\frac n{\alpha\log n}$, the algorithm gradually transitions to a full-gradient Newton-type method.}
 	\label{fig_toy}
\end{figure}

Note that, unless $\alpha$ is very close to $1$ (extremely accurate Hessian oracle), the convergence per data pass must eventually degrade, reaching $\rho \approx 1 - \frac1\alpha$, which is the rate achieved by the corresponding Newton-type method with full gradients (e.g., see Lemma \ref{lma2}). The transition point of $b\approx \frac n{\alpha\log n}$ arises naturally in this setting, because the convergence rate after one outer iteration of \texttt{Mb-SVRN} could not be better than $\rho \approx \left(1-\frac1\alpha\right)^{t_{\max}}$ where $t_{\max}=n/b$ (obtained by treating the stochastic gradient noise as negligible).

\paragraph{Implications for stochastic gradient methods.} 
In the case where the Hessian oracle always returns the identity, \texttt{Mb-SVRN} reduces to a first-order stochastic gradient method which is a variant of \texttt{SVRG} \cite{johnson2013accelerating}. So, it is natural to ask what our convergence analysis implies about the effect of second-order preconitioning on an SVRG-type algorithm. In our setting, \texttt{SVRG} with mini-batch size $b$ achieves an expected convergence rate of $\rho =\mathcal{O}\left(\sqrt{\frac{b\kappa}n}\right)$, compared to a high-probability convergence rate of $\rho = \mathcal{O}\left(\frac{\kappa}{n}\right)$ for \texttt{Mb-SVRN}. This means that, while for small mini-batches $b=\mathcal{O}(1)$ the advantage of preconditioning is not significant, the robust regime of \texttt{Mb-SVRN} (which is not present in \texttt{SVRG}) leads to a gap in convergence rates between the two methods that becomes larger as we increase $b$ (see Figure \ref{fig_1}). This suggests the following general rule-of-thumb:
\begin{center}
    \emph{To improve the robustness of an SVRG-type method to the gradient mini-batch size, one can precondition the method with second-order information based on a Hessian estimate.}
\end{center}

\paragraph{Implications for Newton-type methods.}
A different perspective on our results arises if our starting point is a full-gradient Newton-type method, such as Subsampled Newton \cite{roosta2019sub} or Newton Sketch \cite{pilanci2017newton}, which also arises as a simple corner case of \texttt{Mb-SVRN} by choosing mini-batch size $b=n$ and inner iterations $t_{\max}=1$. In this setting, as mentioned above, the convergence rate of the method behaves as $\rho\approx 1-\frac1\alpha$, which means that it is highly dependent on the quality of the Hessian oracle via the factor $\alpha$. As we decrease $b$ in \texttt{Mb-SVRN}, the method gradually transitions into an SVRG-type method that is more robust to Hessian quality. We note that while our theoretical convergence result in Theorem \ref{main_result} still exhibits a dependence on $\alpha$, empirical results suggest that for large datasets this dependence is much less significant in the robust regime of \texttt{Mb-SVRN} than it is in the full-gradient regime. This suggests the following general prescription:
\begin{center}
    \emph{To improve the robustness of a Newton-type method to the Hessian estimation quality, one can replace the full gradients with variance-reduced sub-sampled gradient estimates.}
\end{center}

\section{Related Work}
\label{related_work}

Among first-order methods, \cite{shalev2013stochastic} propose \texttt{SDCA}, a stochastic variant of dual coordinate ascent.
In \cite{schmidt2017minimizing}, the \texttt{SAG} algorithm is proposed, which
maintains in memory the most recent gradient value computed
for each training example, updates the gradient value by selecting a random sample 
at every iteration, and computes the average gradient. Both \texttt{SDCA} and \texttt{SAG} achieve convergence guarantees comparable to those of \texttt{SVRG} (see Section \ref{s:discussion}). 
In \cite{xiao2014proximal}, the authors propose 
a variant of \texttt{SVRG}, called \texttt{Prox-SVRG} for minimizing the sum of two convex functions, $f$ and $g$, where $f$ takes the form as in \eqref{main} and $g$ is a regularizer. They leverage the composite structure of the problem and incorporate a weighted sampling strategy to improve the complexity so that it depends on the so-called average condition number, $\hat{\kappa}$, instead of $\kappa$. We note that a similar weighted sampling strategy can be incorporated into \texttt{Mb-SVRN} with little additional effort. On the other hand, methods such as \texttt{Katyusha} \cite{allen2017katyusha,allen2018katyusha} and \texttt{Catalyst} \cite{lin2015universal}
use Nesterov-type techniques to accelerate \texttt{SVRG} and achieve
improved convergence guarantees where the condition number $\kappa$ is replaced by its square root $\sqrt\kappa$. In particular, \texttt{Catalyst} \cite{lin2015universal} adds a strongly convex penalty term to the objective function at every iteration, 
and solves a sub-problem followed by introducing momentum. That said, similar to \texttt{SVRG}, all methods discussed in this paragraph are first-order methods and the convergence rate (per data pass) of these methods is not robust to the
gradient mini-batch size. 

To this end, methods that incorporate stochastic second-order information have been proposed. In \cite{gonen2016solving}, using similar approaches to those in \cite{xiao2014proximal}, a sketched preconditioned \texttt{SVRG} method is proposed, for solving ridge regression problems, that reduces the average condition number of the problem. The works \cite{gower2016stochastic,lucchi2015variance,moritz2016linearly}, employ L-BFGS-type updates to approximate the Hessian inverse which is then used as a preconditioner to
variance-reduced gradients. In \cite{liu2019acceleration}, the authors proposed an inexact preconditioned second-order method. All these methods are inherently first-order methods that use preconditioning to improve the dependence of \texttt{SVRG}-type methods on the condition number, and for practical considerations. In particular, their existing convergence analyses do not show scalability and robustness to gradient mini-batch sizes. 
 
 Popular stochastic second-order methods, e.g., Subsampled Newton ~\cite{roosta2019sub,bollapragada2019exact,erdogdu2015convergence} and Newton Sketch ~\cite{pilanci2017newton,berahas2020investigation}, either use the full gradient or require extremely large gradient mini-batch sizes $b \gg \kappa$ at every iteration \cite{derezinski2022stochastic}. As a result, the convergence rate of these methods depends heavily on the quality of Hessian approximation, and the methods are unsuitable for implementations with low computational budget.
 Several other works have used strategies to reduce the variance of the stochastic second-order methods. One such work is Incremental Quasi-Newton \cite{mokhtari2018iqn} which relies on memory and uses aggregated gradient and Hessian information coupled with solving a second-order Taylor approximation of $f$ in a local neighborhood around $\x^*$ in order to reduce the variance of stochastic estimates employed. In \cite{bollapragada2018progressive} the authors propose methods that combine adaptive sampling strategies to reduce the variance and quasi-Newton updating to robustify the method. Finally, in \cite{na2023hessian} the authors show that averaging certain types of stochastic Hessian estimates (such as those based on subsampling) results in variance reduction and 
 improved local convergence guarantees when used with full gradients. Unfortunately, the convergence rates of these methods are not robust to the sample size used for computing the stochastic gradients and/or the quality of Hessian approximation. Also, we note that these methods update the Hessian approximations at every iteration and are not directly comparable to our setting in which
 we update the Hessian approximation using a multi-stage approach.

Recently, \cite{derezinski2022stochastic} investigated the effect of variance-reduced gradient estimates in conjunction with a Newton-type method, using an algorithm 
comparable to \texttt{Mb-SVRN}. They show that using gradient mini-batch size $b$ of the order $\mathcal{O}\left(\tfrac{n}{\alpha\log n}\right)$, one can achieve a local linear convergence rate of $\rho=\Tilde{\mathcal{O}}\left(\tfrac{\alpha^3\kappa}{n}\right)$. However, deviating from this prescribed mini-batch size (in either direction) causes the local convergence guarantee to rapidly degrade, and in particular, the results are entirely vacuous unless $b \gg \alpha^2\kappa$. Our work can be viewed as a direct improvement over this prior work: first, through a better convergence rate dependence on $\alpha$ (with $\alpha^2$ versus $\alpha^3$); and second, in showing that this faster convergence can be achieved for any mini-batch size $b\lesssim \frac{n}{\alpha\log n}$. We achieve this by using a fundamentally different analysis via a chain of martingale concentration arguments combined with using much smaller step sizes in the inner iterations, which allows us to compensate for the additional stochasticity in the gradients in the small and moderate mini-batch size regimes.

\section{Convergence Analysis} \label{technical_analysis}

In this section, we present our technical analysis that concludes with the main convergence result, Theorem \ref{final_result}, informally stated earlier as Theorem \ref{main_result}. We start by introducing some auxiliary lemmas in Section \ref{intlemma}. Then, in Theorem~\ref{Exp1} (Section~\ref{sec.4.2}), we establish a one inner iteration convergence result in expectation, forming the key building block of our analysis. Then, in Section~\ref{High_probability_analysis}, we present our main technical contribution; we construct a submartingale framework with random start and stopping times, which is needed to establish fast high-probability linear convergence of outer iterates in Theorem \ref{final_result}, our main convergence result (Section~\ref{convg_martingale}). At the end of this section, we supplement our main local convergence results with a global convergence guarantee for \texttt{Mb-SVRN} in Theorem \ref{global_convergence} (Section~\ref{global}).

\paragraph{Notation.} Let $\Tilde{\x}_0$ denote the starting outer iterate and $\Tilde{\x}_s$ denote the outer iterate after $s$ outer iterations. Within the $(s+1)^{st}$ outer iteration, the inner iterates are indexed as $\x_{t,s}$. Since our results require working only within one outer iteration, we drop the subscript $s$ throughout the analysis, and denote $\Tilde{\x}_0$ as $\Tilde{\x}$ and corresponding $t^{th}$ inner iterate as $\x_t$. Similarly let $\Tilde\Hi$ and $\Hi_t$ denote the Hessian at $\T\x$ and $\x_t$, respectively. For gradient mini-batch size $b$, let $\hat{\T\g}$ denote $\frac{1}{b}\sum_{j=1}^{b}{\nabla{\psi}_{i_j}(\T\x)}$, let $\hat{\g}_t$ denote $\frac{1}{b}\sum_{j=1}^{b}{\nabla{\psi}_{i_j}(\x_{t})}$ and let $\bar{\g}_t= \hat{\g}_t -\hat{\T\g} + \T\g$ denote the variance-reduced gradient at $\x_t$, where $\T\g$ represents exact gradient at $\T\x$. Moreover, $\g_t$ represents exact gradient at $\x_t$, $\Hi$ represents Hessian at $\x^*$, and let $\T\Delta = \T\x-\x^*$, $\Delta_t = \x_t-\x^*$. Also, we use $\eta$ to denote the fixed step size. Finally, we refer to $\E_t$ as the conditional expectation given the event that the algorithm has reached iterate $\x_t$ starting with the outer iterate $\T\x$.

Next, we define the local convergence neighborhood, by using the notion of a Mahalanobis norm: $\|\x\|_{\M} = \sqrt{\x^\top\M\x}$, for any positive semi-definite matrix $\M$ and vector $\x$. 
\begin{definition}\label{d:neighborhood}
For a given $\epsilon_0 > 0$, we define a local neighborhood $\mathcal{U}_f(\epsilon_0)$ around $\x^*$, as follows: 
$$\mathcal{U}_f(\epsilon_0):= \Big\{\x: \normHsq{\x - \x^*} < \frac{\mu^{3/2}}{L}\epsilon_0\Big\},$$
where $L>0$ is the Lipschitz constant for Hessians of $f$, and $\mu>0$ is the strong convexity parameter for $f$.
\end{definition}

\subsection{Auxiliary lemmas}{\label{intlemma}}

In this section we present auxiliary lemmas that will be used in the analysis. The first result upper bounds the variance of the \texttt{SVRG}-type stochastic gradient in terms of the smoothness parameter $\lambda$ and gradient mini-batch size $b$. Similar variance bounds have also been derived in prior works \cite{johnson2013accelerating,derezinski2022stochastic,berahas2023accelerating}. 

\begin{lemma}[\textbf{Upper bound on variance of stochastic gradient}]\label{lma1}
Let $\T\x \in \mathcal{U}_f(\epsilon_0\eta)$ and $\x_t \in \mathcal{U}_f(c\epsilon_0\eta)$ for some $c\geq 1$. Then: 
\begin{align*}
\E_t [\nsq{\bar{\g}_{t}-\g_{t}}] \leq (1+c\epsilon_0\eta)\frac{\lambda}{b} \nsq{\Tilde \x-\x_{t}}_{\Hi}.
\end{align*}
 
\end{lemma}
\begin{proof}
   Refer to Appendix \ref{plma1}. 
\end{proof}
In the following result, we use standard approximate Newton analysis to establish that, using an $\alpha$-approximate Hessian at $\T\x$ (Definition \ref{d:hessian-oracle}, denoted as $\hat{\Hi}\approx_{\sqrt{\alpha}}\Tilde\Hi$), a Newton step with a sufficiently small step size $\eta$ reduces the distance to $\x^*$ (in $\Hi$ norm), at least by a factor of $\left(1-\frac{\eta}{8\sqrt\alpha}\right)$. We call this an approximate Newton step since the gradient is exact and only the second-order information is stochastic.

\begin{lemma}[Guaranteed error reduction for approximate Newton]{\label{lma2}} Consider an approximate Newton step where $\x_{\textnormal{AN}} =\x_{t}-\eta\hat{\Hi}^{-1}\g_{t}$, $\hat{\Hi}\approx_{\sqrt{\alpha}}\Tilde\Hi$, $\T\x \in \mathcal{U}_f(\epsilon_0\eta)$, and $ \x_{t} \in \mathcal{U}_f(c\epsilon_0\eta)$, for some $c \geq 1$, $\epsilon_0 < \frac{1}{8c\sqrt{\alpha}}$ and $\eta < \frac{1}{4\sqrt{\alpha}}$. Then:
 \begin{align*}
     \normH{\x_{\textnormal{AN}}-\x^*} \leq \left(1-\frac{\eta}{8\sqrt{\alpha}}\right)\normH{\x_t-\x^*},
 \end{align*} 
 where $\Hi$ denotes the Hessian matrix at $\x^*$.
\end{lemma}
\begin{proof}
Refer to Appendix \ref{plma2}.
\end{proof}
Note that the approximate Newton step in Lemma \ref{lma2} uses an exact gradient at every iteration, requiring the use of all $n$ data samples. In our work, in addition to approximate Hessian, we incorporate variance-reduced gradients as well, calculated using $b$ samples out of $n$. For a high probability analysis, we require an upper bound on the noise introduced due to the stochasticity in the gradients. The next result provides an upper bound on the noise (with high probability) in a single iteration introduced due to noise in the stochastic gradient.

\begin{lemma}[High-probability bound on stochastic gradient noise]{\label{lma3}}
Let $\epsilon_0 < \frac{1}{8c\sqrt{\alpha}}$, $\eta < \frac{1}{4\sqrt{\alpha}}$, $ \T\x \in \mathcal{U}_f(\epsilon_0\eta)$, $\x_t \in \mathcal{U}_f(c\epsilon_0\eta)$. Then, for any $\delta>0$, with probability at least $1-\delta\frac{b^2}{n^2}$ (depending on the gradient mini-batch size $b$):
\begin{align*}
    \norm{\g_t-\bar{\g}_t} \leq
    \begin{cases}
(16\sqrt{\kappa\lambda}/3b)
\ln(n/b\delta)\normH{\x_t-\T\x} &\text{for $b< \frac{8}{9}\kappa$,}\\[2mm]
4\sqrt{\lambda/b}\,\ln(n/b\delta)\normH{\x_t-\T\x} &\text{for $b \geq \frac{8}{9}\kappa$.}
\end{cases}
\end{align*}
\end{lemma}
\begin{proof}
   Refer to Appendix \ref{plma3}. 
\end{proof}
The next result provides an upper bound on the gradient $\g_t=\g(\x_t)$ in terms of the distance of $\x_t$ to the minimizer $\x^*$. The result is useful in establishing a 
lower bound on $\normH{\Delta_{t+1}}$ in terms of $\normH{\Delta_t}$, essential to control the randomness in our submartingale framework, described in later sections.
\begin{lemma}[Mean Value Theorem]{\label{MVT}}
Let $\T\x \in \mathcal{U}_f(\epsilon_0\eta)$, $\x_t \in \mathcal{U}_f(c\epsilon_0\eta)$ for $\epsilon_0 < \frac{1}{8c\sqrt{\alpha}}$. If the Hessian estimate satisfies $\hat\Hi\approx_{\sqrt\alpha}\tilde\Hi$, then
\begin{align*}
   \normH{\hat{\Hi}^{-1}\g_t}\leq 2\sqrt{\alpha}\normH{\Delta_t}.
\end{align*}
\end{lemma}

\begin{proof}
    Refer to Appendix \ref{proof_MVT}.
\end{proof}

\subsection{One step expectation result}\label{sec.4.2}

In the following theorem, we use the auxiliary lemmas (Lemma~\ref{lma1} and Lemma~\ref{lma2}) to show that for a sufficiently small step size $\eta$, in expectation, \texttt{Mb-SVRN} is similar to the approximate Newton step (Lemma \ref{lma2}), with an additional small error term that can be controlled by the step size $\eta$. We note that while this result illustrates the convergence behavior of \texttt{Mb-SVRN}, it is by itself not sufficient to guarantee  
overall convergence  
(either in expectation or with high probability), because it does not ensure that the iterates will remain  
in the local neighborhood $\mathcal{U}_f$ throughout the algorithm (this is addressed with our submartingale framework, Section~\ref{High_probability_analysis}).
\begin{theorem}[\textbf{One inner iteration conditional expectation result}]{\label{Exp1}}
Let $\epsilon_0 < \frac{1}{8c\sqrt{\alpha}}$, $\T\x \in \mathcal{U}_f(\epsilon_0\eta)$, $ \x_t \in \mathcal{U}_f(c\epsilon_0\eta)$.
Consider \texttt{Mb-SVRN} with step size $0<\eta<\min\{\frac{1}{4\sqrt{\alpha}},\frac{b}{48\alpha^{3/2}\kappa}\}$ and gradient mini-batch size $b$. Then:
\begin{align*}
   \E_{t}[\normHsq{\Delta_{t+1}}
] \leq \left(1-\frac{\eta}{8\sqrt{\alpha}}\right)\normHsq{\Delta_{t}} + 3\eta^2\frac{\alpha\kappa}{b}\normHsq{\T\Delta}.
\end{align*}

\end{theorem}
\begin{proof}
     Taking the conditional expectation $\E_t$ of 
     $\normHsq{\Delta_{t+1}}$,
\begin{align*}
\E_t[\normHsq{\Delta_{t+1}}] &= \E_t\left[\normHsq{\x_t - \eta\hat{\Hi}^{-1}\bar{\g}_t-\x^*}\right]\\
    &\leq\E_t\left[\normHsq{\Delta_t-\eta\hat\Hi^{-1}\g_t} \right] + \eta^2\cdot \E_t\left[\normHsq{\hat\Hi^{-1}(\g_t-\bar{\g}_t)}\right].
\end{align*}
This is because $\E_t[\bar{\g}_t-\g_t]=0$ and therefore the cross term vanishes.
Since  we know that $\x_t \in \mathcal{U}_f(c\epsilon_0\eta)$ and $\epsilon_0 < \frac{1}{8c\sqrt{\alpha}}$, by Lemma \ref{lma2} we have $\normH{\Delta_t- \eta\hat\Hi^{-1}\g_t} \leq (1-\frac{\eta}{8\sqrt{\alpha}})\normH{\Delta_t}$. Substituting this in the previous inequality, it follows that
\begin{align*}
    \E_t[\normHsq{\Delta_{t+1}}] &\leq \left(1-\frac{\eta}{8\sqrt{\alpha}}\right)^2\E_t\left[\normHsq{\Delta_t}\right] +  \eta^2\cdot \E_t\left[\normHsq{\hat\Hi^{-1}(\g_t-\bar{\g}_t)}\right]\\
     & \leq \left(1-\frac{\eta}{8\sqrt{\alpha}}\right)^2\normHsq{\Delta_t} + \eta^2 \nsq{\Hi^{1/2}\hat{\Hi}^{-1/2}}\cdot\E_t\left[\nsq{\hat{\Hi}^{-1/2}(\g_t-\bar{\g}_t)}\right].
\end{align*}
 Using that $\hat{\Hi} \approx_{\sqrt{\alpha}} \T\Hi$ and $\T\Hi\approx_{ (1+\epsilon_0\eta)}\Hi$, we have $\hat{\Hi} \approx_{\sqrt{\alpha}(1+\epsilon_0\eta)}\Hi$. Therefore, $\nsq{\Hi^{1/2}\hat{\Hi}^{-1/2}} = \norm{\Hi^{1/2}\hat{\Hi}^{-1}\Hi^{1/2}} \leq \sqrt{\alpha}(1+\epsilon_0\eta)$. So we get,
\begin{align*}
   \E_t[\normHsq{\Delta_{t+1}}] &\leq \left(1-\frac{\eta}{8\sqrt{\alpha}}\right)^2\normHsq{\Delta_t} + \eta^2 \sqrt{\alpha}(1+\epsilon_0\eta) \cdot \E_t\left[\nsq{\hat{\Hi}^{-1/2}(\g_t-\bar{\g}_t)}\right].
\end{align*}
Upper bounding $\nsq{\hat{\Hi}^{-1/2}} = \norm{\hat{\Hi}^{-1}} \leq \frac{\sqrt\alpha}{\mu}$,
\begin{align*}
    \E_t[\normHsq{\Delta_{t+1}}] &\leq \left(1-\frac{\eta}{8\sqrt{\alpha}}\right)^2\normHsq{\Delta_t} + \eta^2 \frac{\alpha}{\mu}(1+\epsilon_0\eta) \cdot \E_t[\nsq{(\g_t-\bar{\g}_t)}].
\end{align*}
By 
Lemma \ref{lma1}, we bound 
the last term of the previous inequality,
\begin{align*}
    \E_t\left[\nsq{\g_t-\bar{\g}_t}\right] &\leq \frac{(1+c\epsilon_0\eta)\lambda}{b}\normHsq{\Delta_t-\T\Delta}\\
    &\leq \frac{2(1+c\epsilon_0\eta)\lambda}{b}\left(\normHsq{\Delta_t}+\normHsq{\T\Delta}\right).
\end{align*}
Putting it all together, 
\begin{align}
   \E_t[\normHsq{\Delta_{t+1}}] &\leq \left[\left(1-\frac{\eta}{8\sqrt{\alpha}}\right)^2 + 3\eta^2\alpha\frac{\kappa}{b}\right]\normHsq{\Delta_t} + 3\eta^2 \frac{\alpha \kappa}{b} \cdot \normHsq{\T\Delta} \nonumber \\
   & = \left[1- \frac{\eta}{4\sqrt{\alpha}} + \frac{\eta^2}{64\alpha} + 3\eta^2 \frac{\alpha \kappa}{b} \right]\normHsq{\Delta_t} + 3\eta^2 \frac{\alpha \kappa}{b} \cdot \normHsq{\T\Delta} \nonumber\\
   & < \left(1- \frac{\eta}{8\sqrt{\alpha}}\right)\normHsq{\Delta_t} + 3\eta^2 \frac{\alpha \kappa}{b} \cdot \normHsq{\T\Delta},\label{exp_1}
\end{align}
where we used that $(1+\epsilon_0\eta) <1+\frac{1}{8}$ and $\eta < \frac{b}{48\alpha^{3/2}\kappa}$ imply  
$3\eta^2\frac{\alpha\kappa}{b} < \frac{\eta}{16\sqrt{\alpha}}$.
\ifdocenter
\else
\qed
\fi
\end{proof}
In the remainder of the  
analysis, we set $c=e^2$ and consider $\eta=\frac{b\sqrt{\alpha}\beta}{n}$ for some $\beta >1$. These specific values for $c$ and $\eta$ are set in hindsight based on the optimal step size and the neighborhood scaling factor $(c)$ derived later as artifacts of our analysis. On the other hand, one can do the analysis with general $c$ and $\eta$ and later derive these assignments for $\eta$ and $c$. The value for $\beta$ also gets specified as the analysis proceeds.

\subsection{Building blocks for submartingale framework}{\label{High_probability_analysis}}

We begin by building an intuitive understanding of our submartingale framework, where we define a random process $Y_t$ to describe the convergence behavior of the algorithm. Informally, one  
can think of $Y_t$ as representing the error $\normHsq{\Delta_t}$, and aiming to establish the submartingale property $\E_t[Y_{t+1}] < Y_t$. We next highlight two problematic scenarios with achieving the submartingale property, which are illustrated in Figure \ref{fig:submartingale}. 

First, the assumption $\x_t \in \mathcal{U}_f(c\epsilon_0\eta)$ in Theorem \ref{Exp1} ensures that the approximate Newton step reduces the error at least by a factor of $(1-\frac{\eta}{8\sqrt\alpha})$. However, if  $\x_t \not \in \mathcal{U}_f(c\epsilon_0\eta)$, this claim is no longer valid, leading to the first scenario that disrupts the submartingale behavior. This implies the necessity of the property $\x_t \in \mathcal{U}_f(c\epsilon_0\eta)$, for some constant $c$, to establish  
$\E_t[Y_{t+1}] < Y_t$. In the course of our martingale concentration analysis, we show the existence of an absolute constant $c>0$, such that $\x_t \in \mathcal{U}_f(c\epsilon_0\eta)$, for all $t$, with high~probability.

We now explain the second scenario that disrupts the submartingale property. Note that, $Y_{t+1}$ involves noise arising due to the stochasticity of the variance-reduced gradient $\bar{\g}_t$. This stochasticity leads to the term $3\eta^2\frac{\alpha\kappa}{b}\normHsq{\T\Delta}$ in the statement of Theorem \ref{Exp1}. Now, in the event that this noise dominates the error reduction due to the approximate Newton step, one cannot establish that $\E_t[Y_{t+1}] \leq Y_t$. This problematic scenario occurs only when $\normHsq{\Delta_t} \lesssim \frac{\kappa\alpha^2}{n}\normHsq{\T\Delta}$, i.e., when the error becomes small enough to recover our main convergence guarantee. As the algorithm cannot automatically detect this scenario and restart a new outer iteration, it still performs the remaining inner iterations. This results in subsequent inner iterates being affected by high stochastic gradient noise, disrupting the submartingale~property. 
 
To address these scenarios, we introduce random stopping times $T_i$ and random resume times $L_i$, where stopping times capture the disruptive property scenarios, and resume times capture property restoration after encountering stopping times. Refer to Figure \ref{fig:submartingale} for a visual representation, where the dashed curve denotes the error $\normHsq{\Delta_t}$, the $x$-axis denotes the iteration index $t$, \mycirc[red] denotes the stopping time, and \mycirc[teal] denotes the resume time. The submartingale property holds as long as $\normHsq{\Delta_t}$ stays between the brown and blue lines (i.e., in the white strip between the grey
). The brown line at the top denotes the scenario where  
$\x_t \not \in \mathcal{U}_f(c\epsilon_0\eta)$, and the blue line at the bottom denotes the scenario where  
$\normHsq{\Delta_t} < \frac{\kappa\alpha^2}{n}\normHsq{\T\Delta}$. The green line denotes the convergence guarantee provided  
in our work, which remarkably is a constant multiple of the blue line. Our martingale concentration argument shows that after $n/b$ iterations, the error will, with high probability, fall below the green line.

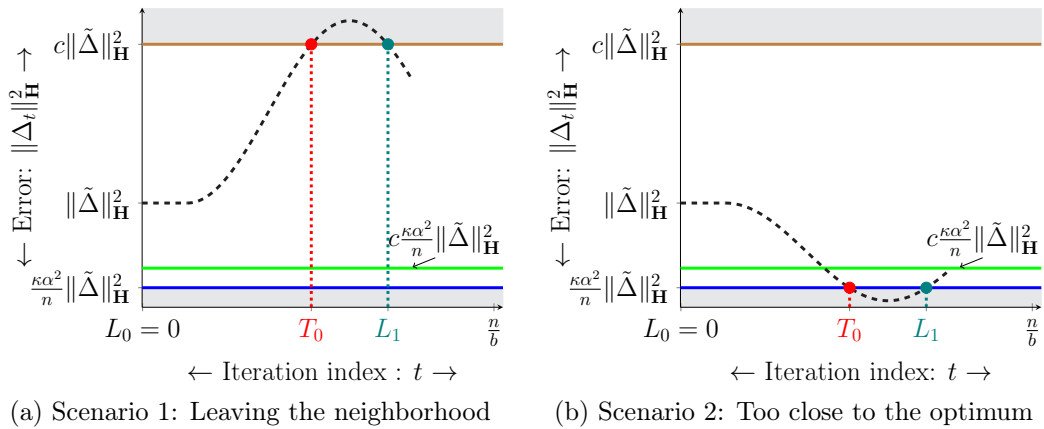
\begin{figure}
\ifdocenter
\centering
\fi
\begin{tabular}{cc}
\hspace{-7mm}\subfloat[Scenario 1: Leaving the neighborhood] { \hspace{\toyshift}\begin{tikzpicture}[scale=0.7,font=\normalsize]
    \draw [black!10,fill=black!10] (0,5) rectangle (6.85,5.675);
    \draw [black!10,fill=black!10] (0,0) rectangle (6.85,0.35);
        \begin{axis}[
            xmin=0,xmax=80,
            xlabel={z},
            ymin=0,ymax=2300,
        xtick={0.1,37.5,54.5,78},
        xticklabels={\large $L_0=0$, \large \Red{$T_0$}, \large \textcolor{teal}{$L_1$}, \large $\frac{n}{b}$},
        ytick={2019,800,150},
        yticklabels={\large $c\|\tilde\Delta\|_{\Hi}^2$, \large $\|\tilde\Delta\|_{\Hi}^2$, \large $\hspace{-5mm}\frac{\kappa\alpha^2}{n}\|\tilde\Delta\|_{\Hi}^2$},
        xlabel={\large $\leftarrow$ Iteration index : $t$ $\rightarrow$},  
        ylabel={\large $\leftarrow$ Error: $\|\Delta_t\|_{\Hi}^2$ $\rightarrow$},
        axis lines=left] 
        \addplot[ultra thick, domain =0 : 80, samples =100, color=brown]{2019};
        \addplot[ultra thick, domain =0 : 80, samples =100, color=green]{300};
        \addplot[ultra thick, domain =0 : 80, samples =100, color=blue]{150};
        \addplot[ultra thick, dashed, domain =0 : 10, samples =100, color=black]{800};
        \addplot[ultra thick, dashed, domain =10 : 60, samples =100,
        color=black]{1500 + 700*sin(5*(x-10)-90)};
        \draw [dotted,ultra thick,color= red] (37.5,0) -- (37.5,2019);
         \node at (67,530) {\large $c\frac{\kappa\alpha^2}{n}\|\T\Delta\|_{\Hi}^2$};
        \draw[->](65,380)   -- (60,310);
        \draw [dotted,ultra thick,color= teal] (54.5,0) -- (54.5,2019);
        \addplot[red,mark=*,mark size=3pt] coordinates {(37.5,2019)} node{};
        \addplot[teal,mark=*,mark size=3pt] coordinates {(54.5,2019)} node{};
    \end{axis}
    \end{tikzpicture}}&
\subfloat[Scenario 2: Too close to the optimum]{\begin{tikzpicture}[scale=0.7,font=\normalsize]
    \draw [black!10,fill=black!10] (0,5) rectangle (6.85,5.675);
    \draw [black!10,fill=black!10] (0,0) rectangle (6.85,0.35);
        \begin{axis}[
            xmin=0,xmax=80,
            xlabel={z},
            ymin=0,ymax=2300,
        xtick={0.1,37.5,54.5,78},
        xticklabels={\large $L_0=0$, \large \Red{$T_0$}, \large \textcolor{teal}{$L_1$}, \large $\frac{n}{b}$},
        ytick={2019,800,150},
        yticklabels={\large $c\|\tilde\Delta\|_{\Hi}^2$, \large $\|\tilde\Delta\|_{\Hi}^2$, \large $\hspace{-5mm}\frac{\kappa\alpha^2}{n}\|\tilde\Delta\|_{\Hi}^2$},
        xlabel={\large $\leftarrow$ Iteration index: $t$ $\rightarrow$},  
        ylabel={\large $\leftarrow$ Error: $\|\Delta_t\|_{\Hi}^2$ $\rightarrow$},
        axis lines=left] 
        \addplot[ultra thick, domain =0 : 80, samples =100, color=brown]{2019};
        \addplot[ultra thick, domain =0 : 80, samples =100, color=green]{300};
        \addplot[ultra thick, domain =0 : 80, samples =100, color=blue]{150};
        \node at (67,530) {\large $c\frac{\kappa\alpha^2}{n}\|\T\Delta\|_{\Hi}^2$};
        \draw[->](67,380)   -- (62,310);
        \addplot[ultra thick, dashed, domain =0 : 10, samples =100, color=black]{800};
        \addplot[ultra thick, dashed, domain =10 : 59, samples =100,
        color=black]{425 + 375*sin(5*(x-10)+90)};
        \draw [dotted,ultra thick,color= red] (37.5,0) -- (37.5,150);
        \draw [dotted,ultra thick,color= teal] (54.5,0) -- (54.5,150);
        \addplot[red,mark=*,mark size=3pt] coordinates {(37.5,150)} node{};
        \addplot[teal,mark=*,mark size=3pt] coordinates {(54.5,150)} node{};
    \end{axis}
    \end{tikzpicture}}
\end{tabular}
\caption{Visual illustration of two problematic scenarios disrupting the submartingale behavior of $\normHsq{\Delta_t}$. The green line denotes the convergence guarantee we prove in our work (Theorem \ref{final_result}). In the left plot, the red dot (\mycirc[red]) represents the first stopping time obtained due to failure of the local neighborhood condition, plotted via the brown line near the top. In the right plot, the red dot (\mycirc[red]) represents the first stopping time obtained due the iterate $\x_t$ lying too close to $\x^*$, plotted via the blue line near the bottom. 
The teal colored dots (\mycirc[teal]) denote the resume time representing the restoration of submartingale~property.}\label{fig:submartingale}
\end{figure}

We proceed to formally define the random stopping and resume times. Consider a fixed $\gamma>0$. Define random stopping times $T_i$ and random resume times $L_i$, with $L_0=0$ and for $i\geq 0$, as
\begin{align}
    T_{i} &= \min\left(\left\{t>L_i \ | \ \normHsq{\Delta_t} > e^2\normHsq{\T\Delta} \ \text{or} \ \normHsq{\Delta_t} < \frac{\kappa\alpha^2\gamma}{n}\normHsq{\T\Delta}\right\}\cup \left\{\frac{n}{b}\right\}\right), \nonumber\\
      L_{i+1} &=\min\left(\left\{t>T_{i} \ | \ \normHsq{\Delta_t} \leq e^2\normHsq{\T\Delta}\ \text{and} \ \normHsq{\Delta_t} \geq \frac{\kappa\alpha^2\gamma}{n}\normHsq{\T\Delta}\right\}\cup \left\{\frac{n}{b}\right\}\right) \label{random_times}.
\end{align}
The random stopping time $T_i$ denotes the $(i+1)^{th}$ iteration index when the random sequence $(\x_t)$ leaves the local neighborhood $\mathcal{U}_f(e^2\epsilon_0\eta)$ or satisfies $\normHsq{\Delta_t} < \frac{\kappa\alpha^2\gamma}{n}\normHsq{\T\Delta}$. The random resume time $L_i$ denotes the $(i+1)^{th}$ iteration index when the random sequence $(\x_t)$ returns back to the local neighborhood and also satisfies $\normHsq{\Delta_t} \geq \frac{\kappa\alpha^2\gamma}{n}\normHsq{\T\Delta}$. The condition $\normHsq{\Delta_t} \leq e^2\normHsq{\T\Delta}$, ensures that the previous results (Lemma \ref{lma1}, \ref{lma2}, \ref{lma3}) hold with $c=e^2$, whereas the condition $\normHsq{\Delta_t} \geq \frac{\kappa\alpha^2\gamma}{n}\normHsq{\T\Delta}$, captures the fact that the optimal convergence rate has not been achieved at $\x_t$.  These two different conditions are required to design a submartingale from random time $L_i$ to $T_i$. Note that, in hindsight, we have defined stopping time in a way similar to our convergence guarantee, $\normHsq{\Delta_\frac{n}{b}} \lesssim \frac{\kappa\alpha^2}{n}\normHsq{\T\Delta}$. We add a factor $\gamma$ to capture constants and log factors that can appear in the convergence guarantee. One can show that for $T_i \leq t<L_{i+1}$, the behavior of $\normHsq{\Delta_t}$ is dictated by the variance in the stochastic gradient and not the  
progress made by the approximate Newton~step. 
\begin{remark}
The stopping and resume times satisfy the following properties:
\begin{enumerate}[(a)]
\item  For $L_0 \leq t <T_0$, we have $\normHsq{\Delta_t} \leq e^2 \normHsq{\T\Delta}$. This is due to the definitions of $L_0$ and $T_0$ in \eqref{random_times}. Equivalently, we can write,
\begin{align}
\Pr\left(\normHsq{\Delta_t} \leq e^2 \normHsq{\T\Delta} \ | \ L_0 \leq t <T_0 \right) =1. \label{first_stop_time}
\end{align}

    \item For $T_i \leq t < L_{i+1}$ and given that $\normHsq{\Delta_{T_i}} < \frac{\kappa\alpha^2\gamma}{n} \normHsq{\T\Delta} $, we have $\normHsq{\Delta_t} < \frac{\kappa\alpha^2\gamma}{n} \normHsq{\T\Delta}$. This is because $L_{i+1}$ is the first instance after $T_i$ such that $\frac{\kappa\alpha^2\gamma}{n}\normHsq{\T\Delta} \leq \normHsq{\Delta_{L_i}} \leq e^2\normHsq{\T\Delta}$. Equivalently, we can write,
    \begin{align}
        \Pr\left(\normHsq{\Delta_t} < \frac{\kappa\alpha^2\gamma}{n} \normHsq{\T\Delta} \ | \ T_i \leq t < L_{i+1}, \normHsq{\Delta_{T_i}} < \frac{\kappa\alpha^2\gamma}{n} \normHsq{\T\Delta}  \right) =1 \label{from_T_to_L}.
    \end{align}

     \item For any gradient mini-batch size $1\leq b\leq n$, $\texttt{Mb-SVRN}$ performs $n/b$ inner iterations, and therefore by construction random times $T_i$ and $L_i$ would not be realized more than $n/b$ times.
\end{enumerate}
\end{remark}
Note that due to \eqref{from_T_to_L}, $\texttt{Mb-SVRN}$ provides a very strong guarantee on the error of the iterates $\x_t$ when $T_i \leq t < L_{i+1}$ and $\normHsq{\Delta_{T_i}} < \frac{\kappa\alpha^2\gamma}{n} \normHsq{\T\Delta}$. However, the stopping time $T_i$ can be realized due to $\normHsq{\Delta_{T_i}} > e^2\normHsq{\T\Delta}$ or $T_i =n/b$, and in that case the conditioning event in \eqref{from_T_to_L} may  
never hold. 

Now we proceed to construct a submartingale framework from the random resume time $L_i$ to the random stopping time $T_i$. Our aim is to apply Freedman's inequality on a carefully constructed martingale and prove strong concentration guarantees on $\normH{\Delta_{T_i}}$, resulting in our main result. We state a version of Freedman's inequality for submartingales (which is a minor modification of standard Freedman's for martingales \cite{tropp2011freedman}, see Appendix \ref{Freedman_proof}).

\begin{theorem}[Freedman's inequality for submartingales]{\label{Freedman}}
For a random process $Y_t$ satisfying $\E_t\left[Y_{t+1}\right] \leq Y_t$ and $|Y_{t+1}-Y_t| \leq R$, it follows that:

\begin{align*}
    \Pr\left(\exists t \ | \ Y_t > Y_0 + \lambda \ and \ \sum_{j=0}^{t-1}\E_j\left(Y_{j+1}-Y_j\right)^2 \leq \sigma^2 \ \right) \leq \exp\left(-\frac{1}{4}\min\left\{\frac{\lambda^2}{\sigma^2}, \frac{\lambda}{R}\right\}\right).
\end{align*}
\end{theorem}

As required for the application of Freedman's inequality on any random process, we need to establish the following three properties:
\begin{enumerate}[(a)]
\item the submartingale property, i.e., $\E_t[Y_{t+1}]\leq Y_t$;
\item the predictable quadratic variation bound, i.e., a bound on $\E_t\,(Y_{t+1}-Y_t)^2$;
\item and the almost sure upper bound, i.e., a bound on $|Y_{t+1}-Y_t|$.
\end{enumerate}

The next lemma shows that if $\x_t$ lies in the local neighborhood $ \mathcal{U}_f(e^2\epsilon_0\eta)$, and does not satisfy $\normHsq{\Delta_t} < \frac{\kappa\alpha^2\gamma}{n}\normHsq{\T\Delta}$, then we have a submartingale property ensuring that in expectation $\normH{\Delta_{t+1}}$ is smaller than $\normH{\Delta_t}$. 
\begin{lemma}[Submartingale property till stopping time]{\label{submartingale_prop}}
Let the gradient mini-batch size be $1 \leq b \leq n$ and $\epsilon_0 < \frac{1}{8e^2\sqrt{\alpha}}$. 
Let $\T\x \in \mathcal{U}_f(\epsilon_0\eta)$, $\x_t \in \mathcal{U}_f(e^2\epsilon_0\eta)$
, $\eta = \frac{b\sqrt{\alpha}\beta}{n} < \min\{\frac{1}{4\sqrt{\alpha}},\frac{b}{48\alpha^{3/2}\kappa}\}$, $\normHsq{\T\Delta} \leq \frac{n}{\kappa\alpha^2\gamma}\normHsq{\Delta_t}$, and $\frac{3\beta}{\gamma} < \frac{1}{16}$. Then:
    \begin{align*}
        \E_t\normH{\Delta_{t+1}} \leq \left(1-\frac{\eta}{32\sqrt{\alpha}}\right)\normH{\Delta_t}.
    \end{align*}
\end{lemma}
\begin{proof} By Theorem \ref{Exp1} and using $\frac{3\beta}{\gamma} < \frac{1}{16}$, it follows that,
\begin{align*}
        \E_t\normHsq{\Delta_{t+1}} &\leq \left(1-\frac{\eta}{8\sqrt{\alpha}}\right)\normHsq{\Delta_t} + 3\eta^2\frac{\alpha\kappa}{b}\normHsq{\T\Delta}\\
        &\leq  \left(1-\frac{\eta}{8\sqrt{\alpha}}\right)\normHsq{\Delta_t} + 3\eta\cdot\frac{b\sqrt{\alpha}\beta}{n}\cdot\frac{\alpha\kappa}{b}\cdot\frac{n}{\kappa\alpha^2\gamma}\normHsq{\Delta_t}\\
        & = \left(1-\frac{\eta}{8\sqrt{\alpha}} + \frac{3\eta}{\sqrt{\alpha}}\cdot\frac{\beta}{\gamma}\right)\normHsq{\Delta_t}\\
        &\leq \left(1-\frac{\eta}{16\sqrt{\alpha}}\right)\normHsq{\Delta_t},
    \end{align*}
    implying that 
\begin{align*}
        \E_t\normH{\Delta_{t+1}} \leq \left(1-\frac{\eta}{32\sqrt{\alpha}}\right)\normH{\Delta_t},
    \end{align*}
    which concludes the proof.
    \ifdocenter
\else
\qed
\fi
    
\end{proof}

\begin{remark} We make a few remarks about the step size $\eta = \frac{b\sqrt{\alpha}\beta}{n}$ which comes out of our analysis as the optimal choice. Lemma \ref{submartingale_prop} requires that $\eta  < \min\left\{\frac{1}{4\sqrt{\alpha}}, \frac{b}{48\alpha^{3/2}\kappa}\right\}$, where the first term in the bound ensures convergence of the approximate Newton step, whereas the second term guarantees control over the gradient noise.
In the regime where $n\gg \alpha^2\kappa\beta$ (which we assume in our main result), it is always true that our chosen step size satisfies $\eta < \frac{b}{48\alpha^{3/2}\kappa}$. However, to ensure that $\eta < \frac{1}{4\sqrt{\alpha}}$ we must restrict the mini-batch size to $b<\frac{n}{4\alpha\beta}$. In our main result, we use $\beta = \mathcal{O}(\log(n/\alpha^2\kappa))$. This suggests that, in the regime of $b \gtrsim \frac{n}{\alpha\log(n)}$, the choice of the step size is primarily restricted by the Hessian approximation factor $\alpha$. Ultimately, this leads to deterioration of the convergence rate as $b$ increases beyond $\frac{n}{\alpha\log(n)}$.
\end{remark}    
   
Throughout our analysis,  
we assume $\frac{b\sqrt{\alpha}\beta}{n} < \frac{1}{4\sqrt{\alpha}}$, which corresponds to the condition $b\lesssim \frac{n}{\alpha\log n}$ stated informally in Theorem \ref{main_result}. This is the regime where the convergence rate of \texttt{Mb-SVRN} does not deteriorate with $b$. Having established the submartingale property, we now prove
the predictable quadratic variation bound property. For proving a strong upper bound on the quadratic variation, we need a result upper bounding $\normH{\T\Delta}$ by $\E_t\normH{\Delta_{t+1}}$, as long as the stopping time criteria are not satisfied at $\x_t$.

\begin{lemma}[\textbf{Bounded variation}]{\label{BV}}
Let $\T\x \in \mathcal{U}_f(\epsilon_0\eta)$ with $\epsilon_0 < \frac{1}{8e^2\sqrt{\alpha}}$ and $\eta \leq \frac{1}{4\sqrt{\alpha}}$. Also, let $\normHsq{\Delta_t} \leq e^2\normHsq{\T\Delta}$ and $\normHsq{\Delta_t} \geq \frac{\kappa\alpha^2\gamma}{n}\normHsq{\T\Delta}$ for some $\gamma >0$. Then, with $\omega = \sqrt{\frac{n}{\kappa\alpha^2\gamma}}$:
\begin{align*}
    \normH{\T\Delta} \leq 2\omega\E_t\normH{\Delta_{t+1}}.
\end{align*}
\end{lemma}
\begin{proof}
We apply Jensen's inequality to $\E_t\normH{\Delta_{t+1}}$ to get,
    \begin{align*}
        \E_t\normH{\Delta_{t+1}} &\geq \normH{\E_t\Delta_{t+1}} = \normH{\Delta_t-\eta\hat{\Hi}^{-1}\g_t}\\
        &\geq \normH{\Delta_t} - \eta\normH{\hat{\Hi}^{-1}\g_t}.
    \end{align*}
    Since 
    $\normHsq{\Delta_t} \leq e^2\normHsq{\T\Delta}$, we have $\x_t \in \mathcal{U}_f(e^2\epsilon_0\eta)$. By using Lemma \ref{MVT} on the term $\normH{\hat{\Hi}^{-1}\g_t}$,
    \begin{align*}
         \E_t\normH{\Delta_{t+1}} &\geq \normH{\Delta_t} -2\eta\sqrt{\alpha}\normH{\Delta_t} \geq \frac{1}{2}\normH{\Delta_t}.
    \end{align*}
    The last inequality holds because $\eta \leq \frac{1}{4\sqrt{\alpha}}$. Finally, we use the condition that $\normHsq{\Delta_t} \geq \frac{\kappa\alpha^2\gamma}{n}\normHsq{\T\Delta}$, concluding the proof.
    \ifdocenter\else\qed\fi
\end{proof}

We proceed with the predictable quadratic variation bound for $\E_t\normH{\Delta_{t+1}}$, assuming that $\x_t$ does not satisfy the stopping time criteria.
\begin{lemma}[\textbf{Predictable quadratic variation bound}]\label{Quadratic_bound}
    Let $\T\x \in \mathcal{U}_f(\epsilon_0\eta)$ with $\epsilon_0 < \frac{1}{8e^2\sqrt{\alpha}}$ and $\eta = \frac{b\sqrt{\alpha}\beta}{n} \leq \min\left\{\frac{1}{4\sqrt{\alpha}},\frac{b}{48\alpha^{3/2}\kappa}\right\}$ for some $\beta >0$. Also, let $\normHsq{\Delta_t} \leq e^2\normHsq{\T\Delta}$ and $\normHsq{\Delta_t} \geq \frac{\kappa\alpha^2\gamma}{n}\normHsq{\T\Delta}$ for some $\gamma>0$. Then, with $\omega = \sqrt{\frac{n}{\kappa\alpha^2\gamma}}$:
    \begin{align*}
     \E_t\left(\normH{\Delta_{t+1}} - \E_t\normH{\Delta_{t+1}}\right)^2 < 80\eta\frac{\beta}{\gamma\sqrt{\alpha}}\left(\E_t\normH{\Delta_{t+1}}\right)^2.
\end{align*}
\end{lemma}
\begin{proof}
    Consider $\E_t\left(\normH{\Delta_{t+1}} - \E_t\normH{\Delta_{t+1}}\right)^2$. We have,
    \begin{align*}
        \E_t\left(\normH{\Delta_{t+1}} - \E_t\normH{\Delta_{t+1}}\right)^2 &= \E_t\normHsq{\Delta_{t+1}} - (\E_t\normH{\Delta_{t+1}})^2\\
        & = \E_t \normHsq{\Delta_t -\eta\hat{\Hi}^{-1}\bar{\g}_t} - (\E_t\normH{\Delta_{t+1}})^2\\
        & = \normHsq{\Delta_t -\eta\hat{\Hi}^{-1}\g_t} + \eta^2\E_t\normHsq{\hat{\Hi}^{-1}(\bar{\g}_t - \g_t)} \\
        &+ 2\eta\E_t\left(\Delta_t -\eta\hat{\Hi}^{-1}\g_t\right)^\top\left(\hat{\Hi}^{-1}(\bar{\g}_t - \g_t)\right) - (\E_t\normH{\Delta_{t+1}})^2.
    \end{align*}
    Using $\E_t\left(\hat{\Hi}^{-1}(\bar{\g}_t - \g_t)\right) = \hat{\Hi}^{-1}\E_t\left(\bar{\g}_t - \g_t\right)=0$ and $\Delta_t- \eta\hat{\Hi}^{-1}\g_t = \E_t\Delta_{t+1}$, we get,
    \begin{align*}
        \E_t\left(\normH{\Delta_{t+1}} - \E_t\normH{\Delta_{t+1}}\right)^2 = \normHsq{\E_t\Delta_{t+1}} + \eta^2\E_t\normHsq{\hat{\Hi}^{-1}(\bar{\g}_t - \g_t)} - (\E_t\normH{\Delta_{t+1}})^2.
    \end{align*}
    Using Jensen's inequality on the first term, we have $\normH{\E_t\Delta_{t+1}} \leq \E_t\normH{\Delta_{t+1}}$, from which it follows that,
    \begin{align}
          \E_t\left(\normH{\Delta_{t+1}} - \E_t\normH{\Delta_{t+1}}\right)^2 &\leq \left(\E_t\normH{\Delta_{t+1}}\right)^2 + \eta^2\E_t\normHsq{\hat{\Hi}^{-1}(\bar{\g}_t - \g_t)} - (\E_t\normH{\Delta_{t+1}})^2 \nonumber \\
          &= \eta^2\E_t\normHsq{\hat{\Hi}^{-1}(\bar{\g}_t - \g_t)}\nonumber \\
          & = \eta^2 \E_t \nsq{\Hi^{1/2}\hat{\Hi}^{-1/2}\hat{\Hi}^{-1/2}(\bar{\g}_t - \g_t)} \nonumber \\
          & \leq \eta^2\cdot\nsq{\Hi^{1/2}\hat{\Hi}^{-1/2}}\cdot\E_t\nsq{\hat{\Hi}^{-1/2}(\bar{\g}_t - \g_t)}. \label{PQV_1}
    \end{align}
    Using the fact that $\hat{\Hi} \approx_{\sqrt{\alpha}} \T\Hi$ and $ \T\Hi \approx_{ (1+\epsilon_0\eta)}\Hi$, we have $\hat{\Hi} \approx_{\sqrt{\alpha}(1+\epsilon_0\eta)}\Hi$. Therefore,  $\nsq{\Hi^{1/2}\hat{\Hi}^{-1/2}} = \norm{\Hi^{1/2}\hat{\Hi}^{-1}\Hi^{1/2}} \leq \sqrt{\alpha}(1+\epsilon_0\eta)$. Combining this with the inequality \eqref{PQV_1},
    \begin{align*}
    \E_t\left(\normH{\Delta_{t+1}} - \E_t\normH{\Delta_{t+1}}\right)^2 &\leq  \eta^2 \sqrt{\alpha}(1+\epsilon_0\eta) \cdot \E_t\nsq{\hat{\Hi}^{-1/2}(\g_t-\bar{\g}_t)}.
\end{align*}
Upper bounding $\nsq{\hat{\Hi}^{-1/2}} = \norm{\hat{\Hi}^{-1}} \leq \frac{\sqrt\alpha}{\mu}$, we get,
\begin{align}
   \E_t\left(\normH{\Delta_{t+1}} - \E_t\normH{\Delta_{t+1}}\right)^2 &\leq  \eta^2 \frac{\alpha}{\mu}(1+\epsilon_0\eta) \cdot \E_t[\nsq{(\g_t-\bar{\g}_t)}]. \label{PQV_2}
\end{align}
Since $\normHsq{\Delta_t} < c^2\normHsq{\T\Delta}$, we have $\x_t \in \mathcal{U}_f(e^2\epsilon_0\eta)$. By using Lemma \ref{lma1} on the last term of the inequality \eqref{PQV_2}, it follows that,
\begin{align}
    \E_t\left[\nsq{\g_t-\bar{\g}_t}\right] &\leq \frac{(1+c\epsilon_0\eta)\lambda}{b}\normHsq{\Delta_t-\T\Delta}\leq \frac{2(1+c\epsilon_0\eta)\lambda}{b}\left(\normHsq{\Delta_t}+\normHsq{\T\Delta}\right). \label{PQV_3}
\end{align}
Substituting \eqref{PQV_3} in the inequality \eqref{PQV_2}, we get,
\begin{align*}
    \E_t\left(\normH{\Delta_{t+1}} - \E_t\normH{\Delta_{t+1}}\right)^2 &\leq 2(1+\epsilon_0\eta)(1+c\epsilon_0\eta)\eta^2\frac{\kappa\alpha}{b}\left(\normHsq{\Delta_t}+\normHsq{\T\Delta}\right).
\end{align*}
Again, using $\normHsq{\Delta_t} < c^2\normHsq{\T\Delta}$,
\begin{align*}
    \E_t\left(\normH{\Delta_{t+1}} - \E_t\normH{\Delta_{t+1}}\right)^2 <20\eta^2\frac{\kappa\alpha}{b}\normHsq{\T\Delta}.
\end{align*}
Since 
$\eta \leq \frac{1}{4\sqrt{\alpha}}$, we use Lemma \ref{BV} on $\normHsq{\T\Delta}$ to obtain,
\begin{align*}
    \E_t\left(\normH{\Delta_{t+1}} - \E_t\normH{\Delta_{t+1}}\right)^2 <80\eta^2\frac{\kappa\alpha}{b}\omega^2\left(\E_t\normH{\Delta_{t+1}}\right)^2.
\end{align*}
Substituting one of the $\eta$ factors as $\frac{b\sqrt{\alpha}\beta}{n}$, and $\omega^2 = \frac{n}{\kappa\alpha^2\gamma}$, concludes the proof.
\ifdocenter
\else
\qed
\fi

\end{proof}

Finally, as the last building block of the submartingale framework, we establish a high probability upper and lower bound on $\normH{\Delta_{t+1}}$ in terms of $\normH{\Delta_t}$. Here, due to
the noise in the stochastic gradient, we cannot prove almost sure bounds. However, by Lemmas \ref{lma3} and \ref{BV}, we can get the upper and lower bounds, holding with probability at least $1-\delta\frac{b^2}{n^2}$, for any $\delta>0$. For notational convenience, we use $M$ to denote the high probability upper bound on $\norm{\bar{\g}_t-\g_t}$, guaranteed in Lemma \ref{lma3}, 
$$M =\begin{cases}
    \frac{\kappa}{b} \  \ \  \ \ \text{if} \ b <\frac{8}{9}\kappa \\
    \sqrt{\frac{\kappa}{b}} \ \ \text{if} \ b \geq \frac{8}{9}\kappa.
\end{cases}$$
\begin{lemma}[\textbf{One step high probability upper bound}]\label{high_prob}
 Let $n > \frac{(96)^2\beta^2\ln(n/b\delta)^2}{\gamma}\kappa$, $\T\x \in \mathcal{U}_f(\epsilon_0\eta)$ with $\epsilon_0 < \frac{1}{8e^2\sqrt{\alpha}}$, $ \eta = \frac{b\sqrt{\alpha}\beta}{n} \leq \min\left\{\frac{1}{4\sqrt{\alpha}},\frac{b}{48\alpha^{3/2}\kappa}\right\}$ and $b < \min\{\frac{n}{4\alpha\beta}, \frac{\gamma n}{(96)^2\beta^2\ln(n/b\delta)^2}\}$, for some $\beta >0$, $\gamma>0$, and any $\delta>0$. Also, let $\normHsq{\Delta_t} \leq e^2\normHsq{\T\Delta}$ and $\normHsq{\Delta_t} \geq \frac{\kappa\alpha^2\gamma}{n}\normHsq{\T\Delta}$. Then, with $\omega = \sqrt{\frac{n}{\kappa\alpha^2\gamma}}$, $s=96\eta\omega\sqrt{\alpha}M\ln(n/\delta b)$ and probability at least $1-\delta\frac{b^2}{n^2}$:
\begin{align*}
    \left(1+s\right)^{-1}\E_t\normH{\Delta_{t+1}} \leq\normH{\Delta_{t+1}} \leq \left(1+s\right)\E_t\normH{\Delta_{t+1}}.
\end{align*}

 \end{lemma}

\begin{proof}
    We first prove the right-hand side inequality (upper bound),
    \begin{align*}
        \normH{\Delta_{t+1}} &=\normH{\Delta_t -\eta\hat{\Hi}^{-1}\bar{\g}_t}\\
        & \leq \normH{\Delta_t -\eta\hat{\Hi}^{-1}\g_t} + \eta\normH{\hat{\Hi}^{-1}(\bar{\g}_t-\g_t)}\\
        &\leq  \normH{\Delta_t -\eta\hat{\Hi}^{-1}\g_t} +\eta\norm{\Hi^{1/2}\hat{\Hi}^{-1}\Hi^{1/2}}\cdot\norm{\Hi^{-1/2}(\bar{\g}_t-\g_t)}.
    \end{align*}
    Since, 
    $\hat{\Hi} \approx_{\sqrt{\alpha}} \T\Hi$ and $\T\x \in \mathcal{U}_f(\epsilon_0\eta)$, we have $\hat{\Hi} \approx_{\sqrt{\alpha}(1+\epsilon_0\eta)} \Hi$, and therefore, $\norm{\Hi^{1/2}\hat{\Hi}^{-1}\Hi^{1/2}} \leq \sqrt{\alpha}(1+\epsilon_0\eta)$. Using this we get,
     \begin{align*}
         \normH{\Delta_{t+1}} &\leq \normH{\Delta_t -\eta\hat{\Hi}^{-1}\g_t} + \eta\frac{\sqrt{\alpha}(1+\epsilon_0\eta)}{\sqrt{\mu}}\norm{(\bar{\g}_t-\g_t)}.
     \end{align*}
         Note that the first term is $\normH{\E_t\Delta_{t+1}}$. Also as $\normHsq{\Delta_t} \leq e^2\normHsq{\T\Delta}$, we have $\x_t \in \mathcal{U}_f(e^2\epsilon_0\eta)$, and therefore, we can apply Lemma \ref{lma3} on the second term. As $M$ captures the effect of whether $b <\frac{8}{9}\kappa$ or $b \geq \frac{8}{9}\kappa$, we have with probability at least $1-\delta\frac{b^2}{n^2}$,
    \begin{align*}
        \normH{\Delta_{t+1}} &\leq \E_t\normH{\Delta_{t+1}} + 6\eta\sqrt{\alpha}\cdot M\ln(n/b\delta)\left(\normH{\Delta_t} + \normH{\T\Delta}\right).
    \end{align*}
    By
    $\normH{\Delta_t} \leq e\cdot\normH{\T\Delta}$ it follows that,
    \begin{align*}
        \normH{\Delta_{t+1}} &\leq \E_t\normH{\Delta_{t+1}} + 20\eta\sqrt{\alpha}\cdot M \ln(n/b\delta)\normH{\T\Delta},
    \end{align*}
    and using Lemma \ref{BV}, we upper bound $\normH{\T\Delta}$ by $2\omega\cdot\E_t\normH{\Delta_{t+1}}$ and get,
    \begin{align}
        \normH{\Delta_{t+1}} &\leq \left(1+ 40\omega\eta\sqrt{\alpha}\cdot M \ln(n/b\delta)\right)\E_t\normH{\Delta_{t+1}} \label{one_step_eq_1}.
    \end{align}
    
    Next, we prove the left-hand side inequality (lower bound), observing that,
    \begin{align*}
        \E_t\normH{\Delta_{t+1}} &= \E_t \normH{\Delta_t -\eta\hat{\Hi}^{-1}\bar{\g}_t}\\
        & \leq \normH{\Delta_t -\eta\hat{\Hi}^{-1}\g_t} + \eta\E_t\normH{\hat{\Hi}^{-1}(\bar{\g}_t-\g_t)}\\
        & \leq \normH{\Delta_t -\eta\hat{\Hi}^{-1}\bar{\g}_t} + \eta\normH{\hat{\Hi}^{-1}(\bar{\g}_t-\g_t)} +\eta\E_t\normH{\hat{\Hi}^{-1}(\bar{\g}_t-\g_t)}.
    \end{align*}
    Note that,
    \begin{align*}\E_t\normH{\hat{\Hi}^{-1}(\bar{\g}_t-\g_t)} &\leq \norm{\Hi^{1/2}\hat{\Hi}^{-1}\Hi^{1/2}}\cdot\norm{\Hi^{-1/2}}\cdot\E_t\norm{(\bar{\g}_t-\g_t)} \\ &\leq (1+\epsilon_0\eta)\frac{\sqrt{\alpha}}{\sqrt{\mu}}\cdot\E_t\norm{(\bar{\g}_t-\g_t)}.
    \end{align*}
    Using Lemma \ref{lma1} to upper bound $\E_t\norm{(\bar{\g}_t-\g_t)}$ and substituting in the previous inequality for $\E_t\normH{\Delta_{t+1}}$, we get,
    \begin{align*}
        \E_t\normH{\Delta_{t+1}} \leq &\normH{\Delta_{t+1}} + \eta\normH{\hat{\Hi}^{-1}(\bar{\g}_t-\g_t)} \\ 
&\qquad +\eta(1+\epsilon_0\eta)\sqrt{1+c\epsilon_0\eta}\cdot\frac{\sqrt{\kappa\alpha}}{\sqrt{b}}\left(\normH{\Delta_t} + \normH{\T\Delta}\right).
    \end{align*} 
    Now in the last term, we use $\normH{\Delta_t} < e\cdot\normH{\T\Delta}$ and upper bound the second term using Lemma \ref{lma3}. Following the same steps as we did for the right-hand side inequality, we get
    \begin{align*}
         \E_t\normH{\Delta_{t+1}} &\leq \normH{\Delta_{t+1}} + 20\eta\sqrt{\alpha}\cdot M\ln(n/b\delta)  
         \normH{\T\Delta} + 4\eta\sqrt{\alpha}\frac{\sqrt{\kappa}}{\sqrt{b}}\normH{\T\Delta}.
    \end{align*}
    Using Lemma \ref{BV} on $\normH{\T\Delta}$,
    \begin{align*}
        \E_t\normH{\Delta_{t+1}} &\leq \normH{\Delta_{t+1}} + 48\eta\sqrt{\alpha}\cdot M \ln(n/b\delta)\E_t\normH{\Delta_{t+1}},
    \end{align*}
implying,
    \begin{align*}
        \E_t\normH{\Delta_{t+1}} &\leq \left(1-48\omega\eta\sqrt{\alpha}\cdot M\ln(n/b\delta)\right)^{-1}\normH{\Delta_{t+1}}.
    \end{align*}
    By the condition on $n$, namely 
    $n > \frac{(96)^2\beta^2\ln(n/b\delta)^2}{\gamma}\max\{\kappa,b\}$, it follows that,
    \begin{align}
        \E_t\normH{\Delta_{t+1}} &\leq \left(1+96\omega\eta\sqrt{\alpha}\cdot M\ln(n/b\delta)\right)\normH{\Delta_{t+1}} \label{one_step_eq_2}.
    \end{align}
    Combining (\ref{one_step_eq_1}) and (\ref{one_step_eq_2}) we conclude the proof.
    \ifdocenter
\else
\qed
\fi
\end{proof}

\begin{remark} 
We make the following remark about the mini-batch size prescribed in Lemma \ref{high_prob}, i.e., $b < \min\{\frac{n}{4\alpha\beta}, \frac{\gamma\cdot n}{(96)^2\beta^2\ln(n/b\delta)^2}\}$. The first term, $b < \frac{n}{4\alpha\beta}$, ensures that optimal step size $\eta$ can be picked for $b \lesssim \frac{n}{\alpha\log(n)}$. The second condition on $b$ eventually reduces to $b \lesssim \frac{n}{\log(n)}$, as we set $\gamma = 1280\beta^2\ln(n/b\delta)$ later in the analysis (see Theorem \ref{freedman_concentration}).
\end{remark}  

\subsection{Martingale setup}{\label{martingale_setup}
In what follows, we analyze the behavior of a carefully defined random process from random times $L_i$ to $T_i$, for $i\geq 0$. In particular, we show that if $T_i = \frac{n}{b}$ then $\normHsq{\Delta_{T_i}} < c_1\frac{\kappa\alpha^2\gamma}{n}\normHsq{\T\Delta}$ with very high probability for some absolute constant $c_1$, and if $T_i < \frac{n}{b}$ then $\normHsq{\Delta_{T_i}} <\frac{\kappa\alpha^2\gamma}{n}\normHsq{\T\Delta}$ with very high probability. For any $t\geq0$, such that $L_i + t <T_i$, we consider an event $\mathcal{A}^i_{t}$,
\begin{align*}
    \mathcal{A}^i_{t}&=\bigcap_{j=L_i}^{j=L_i+t}\left\{\norm{\g_{j} - \bar{\g}_j}_{\Hi} \leq 6M\ln(n/b\delta)\normH{\x_j-\T\x}\right\}.
\end{align*}
The event $\mathcal{A}^i_{t}$ captures the occurrence of the high probability event mentioned in Lemma \ref{lma3}, for iterates ranging from random time $L_i$ to $L_i+t$. Using Lemma \ref{lma3}, it follows that $\Pr\left(\mathcal{A}^i_{t+1} | \mathcal{A}^i_{t}\right) \geq 1-\delta\frac{b^2}{n^2}$, and by the union bound $\Pr\left(\mathcal{A}^i_{t+1} | \mathcal{A}^i_{0}\right) \geq 1-t\delta\frac{b^2}{n^2} > 1-\delta\frac{b}{n}$.
Consider a random process $Y_t^i$ defined as:
\begin{align*}
    Y_0^i &= \ln(\normH{\Delta_{L_i}}),
\end{align*}
and for $ L_i + t < T_i$,
\begin{align*}
    Y_{t+1}^{i} &= \left(\ln(\normH{\Delta_{L_i+t+1}}) + \sum_{j=0}^{t}{\ln\left(\frac{\normH{\Delta_{L_i+j}}}{\E_{L_i+j}\normH{\Delta_{L_i+j+1}}}\right)}\right)\cdot\ind_{\mathcal{A}^i_t} + Y^i_t\cdot\ind_{\neg\mathcal{A}^i_t},
\end{align*}
where $\E_{L_i+t}$ means expectation conditioned on the past till iterate $\x_{L_i+t}$. At iteration $t+1$, the random process $Y^i_{t+1}$ checks for the two halting conditions mentioned in the definition of $T_i$. If any of the halting condition is met then we get $L_i+t=T_i$, the random process halts, otherwise, the random process proceeds to iteration $t+2$. We analyze the random process $Y_{t+1}^{i}$ till $L_i+t= T_i$. The first observation is that $Y_{t+1}^{i}$ is a sub-martingale, proven as follows. If $\ind_{\mathcal{A}^i_t}=0$ or $\ind_{\mathcal{A}^i_{t-1}}=0$ then $Y^i_{t+1} = Y^i_{t}$ and trivially $\E_tY^{t+1}_i = Y^i_t$. So we consider $\ind_{\mathcal{A}^i_t} = \ind_{\mathcal{A}^i_{t-1}}=1$ and get,
\begin{align}
    \E_t[Y_{t+1}^{i}] -Y_t^i &= \E_{L_i+1} \ln(\normH{\Delta_{L_i+t+1}}) +\ln\left(\frac{\normH{\Delta_{L_i+t}}}{\E_{L_i+t}\normH{\Delta_{L_i+t+1}}}\right) - \ln(\normH{\Delta_{L_i+t}}) \nonumber\\
    &\leq \ln(\E_{L_i+t}\normH{\Delta_{L_i+t+1}}) + \ln\left(\frac{1}{\E_{L_i+t}\normH{\Delta_{L_i+t+1}}}\right) =0 .\label{submartingale_property}
\end{align}
Due to the quadratic variation bound Lemma \ref{Quadratic_bound}, for $L_i+t < T_i$ we have,
\begin{align}
\E_{L_i+t}\left(\normH{\Delta_{L_i+t+1}}-\E_{L_i+t}\normH{\Delta_{L_i+t+1}}\right)^2 \leq 80\eta\frac{\beta}{\gamma\sqrt{\alpha}} \E_{L_i+t}\normHsq{\Delta_{L_i+t+1}}. \label{QB_2}
\end{align}
We use (\ref{QB_2}) to upper bound $\E_t\left(Y_{t+1}^{i}-Y_t^i\right)^2$, where $\E_t$ denotes the 
expectation conditioned on the past and assuming $Y_t^i$ is known. First note that if $\ind_{\mathcal{A}^i_t}=0$ or $\ind_{\mathcal{A}^i_{t-1}}=0$, then $Y^i_{t+1} = Y^i_{t}$, and therefore $\E_t\left(Y_{t+1}^{i}-Y_t^i\right)^2 =0$. Hence, it remains to consider the case when $\ind_{\mathcal{A}^i_t} = \ind_{\mathcal{A}^i_{t-1}}=1$,
\begin{align*}
\E_t\left(Y_{t+1}^{i}-Y_t^i\right)^2 &= \E_{L_i+t}\left(\ln\norm{\Delta_{L_i+t+1}}-\ln\norm{\Delta_{L_i+t}} + \ln \left(\frac{\normH{\Delta_{L_i+t}}}{\E_{L_i+t}\norm{\Delta_{L_i+t+1}}}\right) \right)^2\\
& = \E_{L_i+t}\left(\ln \left(\frac{\normH{\Delta_{L_i+t+1}}}{\E_{L_i+t}\norm{\Delta_{L_i+t+1}}}\right)\right)^2.
\end{align*}
Note that in the event of $\ind_{\mathcal{A}^i_t}=1$, by Lemma \ref{high_prob} we know that $\frac{1}{2} \leq \frac{\normH{\Delta_{L_i+t+1}}}{\E_{L_i+t}\norm{\Delta_{L_i+t+1}}} \leq 2 $, assuming that $n$ and $b$ satisfy the assumptions of Lemma \ref{high_prob}. Consider the following inequality for $p,q>0$ such that  $\frac{1}{2} \leq \frac{p}{q} \leq 2$,
\begin{align*}
    \left(\ln\left(\frac{p}{q}\right)\right)^2 \leq \max\left\{\ln\left(1+ \left(\frac{p}{q}-1\right)^2 \right),\ln\left(1+ \left(\frac{q}{p}-1\right)^2 \right)\right\}.
\end{align*}
\noindent With $p = \normH{\Delta_{L_i+t+1}}$ and $q =  \E_{L_i+t}\normH{\Delta_{L_i+t+1}}$, we use the above inequality to get,
\begin{align*}
   \E_t\left(Y_{t+1}^{i}-Y_t^i\right)^2 &\leq \E_{L_i+t}\left[\ln \left (1+4 \frac{\left(\normH{\Delta_{L_i+t+1}} -\E_{L_i+t}\normH{\Delta_{L_i+t+1}}\right)^2}{\left(\E_{L_i+t}\normH{\Delta_{L_i+t+1}}\right)^2} \right)\right]\\
   &\leq \ln \left( 1+ 320\eta\frac{\beta}{\gamma\sqrt{\alpha}}\right).
\end{align*}
where the last inequality is due to \eqref{QB_2}. Thus, we get the following predictable quadratic variation bound for random process $Y^i_t$,
\begin{align}    
    \E_t\left(Y_{t+1}^{i}-Y_t^i\right)^2 \leq \ln \left(1 + 320\eta\frac{\beta}{\gamma\sqrt{\alpha}}\right) \label{QB_3}.
\end{align}
Now we aim to upper bound $|Y_{t+1}^{i}-Y_{t}^i|$ for $t$ such that $L_i+t < T_i$. Again, in the events $\ind_{\mathcal{A}^i_t}=0$ or $\ind_{\mathcal{A}^i_{t-1}}=0$, we have $Y^i_{t} = Y^i_{t-1}$ and we are done. Considering the case $\ind_{\mathcal{A}^i_{t}}=\ind_{\mathcal{A}^i_{t-1}}=1$,
\begin{align*}
    Y_{t+1}^{i}- Y_{t}^i = \ln\left(\frac{\normH{\Delta_{L_i+t+1}}}{\E_{L_i+t}\normH{\Delta_{L_i+t+1}}}\right).
\end{align*}
Again due to having $\ind_{\mathcal{A}^i_{t}}=1$, we invoke Lemma \ref{high_prob} to get,
\begin{align*}
    \normH{\Delta_{L_i+t+1}} \leq \left(1+96\eta\omega\sqrt{\alpha}\cdot M\ln(n/b\delta)\right)\E_{L_i+t}\normH{\Delta_{L_i+t+1}},
\end{align*}
where $\omega^2 = \frac{n}{\kappa\alpha^2\gamma}$, and
\begin{align*}
    Y_{t+1}^{i}- Y_{t}^i \leq 2\ln\left(1+96\eta\omega\sqrt{\alpha}\cdot M\ln(n/b\delta)\right).
\end{align*}
Similarly, we get,
\begin{align*}
    Y_{t}^i- Y_{t+1}^{i} \leq 2\ln\left(1+96\eta\omega\sqrt{\alpha}\cdot M\ln(n/b\delta)\right).
\end{align*}
Combining the above upper bound property with (\ref{submartingale_property}) and (\ref{QB_3}), we get the following submartingale framework.
\begin{lemma}[Submartingale setup]{\label{submartingale_setup}}
 Let $n > \frac{(96)^2\beta^2\ln(n/b\delta)^2}{\gamma}\kappa$, $\epsilon_0 < \frac{1}{8e^2\sqrt{\alpha}}$, $ \eta = \frac{b\sqrt{\alpha}\beta}{n} \leq \min\left\{\frac{1}{4\sqrt{\alpha}},\frac{b}{48\alpha^{3/2}\kappa}\right\}$ and $b < \min\{\frac{n}{4\alpha\beta}, \frac{\gamma\cdot n}{(96)^2\beta^2\ln(n/b\delta)^2}\}$ for some $\beta >0$ and $\gamma>0$. Consider the random process defined as $Y_0^i = \ln(\normH{\Delta_{L_i}})$, and for $L_i+t <T_i$,
\begin{align*}
    Y_{t+1}^{i} &= \left(\ln(\normH{\Delta_{L_i+t+1}}) + \sum_{j=0}^{t}{\ln\left(\frac{\normH{\Delta_{L_i+j}}}{\E_{L_i+t}\normH{\Delta_{L_i+j+1}}}\right)}\right)\cdot\ind_{\mathcal{A}^i_t} + Y^i_t\cdot\ind_{\neg\mathcal{A}^i_t}.
\end{align*}
Then, letting $\omega^2 = \frac{n}{\kappa\alpha^2\gamma}$: 
\begin{align*}
\E_t[Y_{t+1}^{i}] &\leq Y_t^i, \\
|Y_{t+1}^{i}- Y_{t}^i| &\leq 2\ln\left(1+96\eta\omega\sqrt{\alpha}\cdot M\ln(n/b\delta)\right),\\
\E_t\left(Y_{t+1}^{i}-Y_t^i\right)^2 &\leq \ln \left(1 + 320\eta\frac{\beta}{\gamma\sqrt{\alpha}}\right).
\end{align*}
\end{lemma}

\subsection{High probability convergence via martingale framework}{\label{convg_martingale}}
In the previous section, we constructed a submartingale framework satisfying a quadratic variation bound and high probability upper bound at every step. In this section, we invoke a well-known measure concentration result for martingales, Freedman's inequality stated in Theorem \ref{Freedman}, on our framework. We apply Theorem \ref{Freedman} on the random process $Y_t^{i}$. For brevity, we provide the analysis only for the case $b < \frac{8}{9}\kappa$. This means replacing $M$ by $\frac{\kappa}{b}$ in Lemma \ref{submartingale_setup}. The proof for the case $b \geq \frac{8}{9}\kappa$ follows along the similar line, by replacing $M$ with $\sqrt{\frac{\kappa}{b}}$.

\noindent We consider $R = 2\ln\left(1+\frac{96\eta\omega\sqrt{\alpha}\kappa\ln(n/b\delta)}{b}\right) $, $\sigma^2 = t\cdot\ln\left(1+320\eta\frac{\beta}{\gamma\sqrt{\alpha}}\right) $, fix $\lambda=1$, and apply Freedman's inequality on the random process $Y_t^i$ until $\normHsq{\Delta_t}$ does not satisfy any of stopping time criteria mentioned in the definition of $T_i$,  \eqref{random_times}. Consider two cases here,

\noindent\textbf{Case 1: $\min\left\{\frac{\lambda^2}{\sigma^2}, \frac{\lambda}{R}\right\} = \frac{\lambda^2}{\sigma^2}$}. Note that
\begin{align*}
    \sigma^2 &\leq t\cdot\ln\left(1+320\eta\frac{\beta}{\gamma\sqrt{\alpha}}\right)\\
    & \leq \frac{n}{b}\cdot\ln\left(1+320\eta\frac{\beta}{\gamma\sqrt{\alpha}}\right) \leq \frac{n}{b}\cdot320\eta\frac{\beta}{\gamma\sqrt{\alpha}},
\end{align*}
where in the last inequality we use $t\leq \frac{n}{b}$ and $\ln(1+x) <x$. We get,
\begin{align*}
    \exp\left(-\frac{1}{4}\cdot\frac{\lambda}{\sigma^2}\right) \leq \exp\left(-\frac{1}{4}\cdot\frac{1}{320\frac{n}{b}\frac{b\sqrt{\alpha\beta}}{n}\frac{\beta}{\gamma\sqrt{\alpha}}}\right)
    & = \exp\left(-\frac{1}{4}\cdot\frac{\gamma}{320\beta^2}\right) < \delta\frac{b}{n},
\end{align*}
where last inequality holds if $\gamma > 1280\beta^2\ln(n/b\delta)$.

\noindent \textbf{Case 2: $\min\left\{\frac{\lambda^2}{\sigma^2}, \frac{\lambda}{R}\right\} = \frac{\lambda}{R}$}. Note that
\begin{align*}
    R &= 2\ln\left(1+\frac{96\eta\omega\sqrt{\alpha}\kappa\ln(n/b\delta)}{b}\right)\\
    &<\frac{200\eta\omega\sqrt{\alpha}\kappa\ln(1/\delta)}{b},
\end{align*}
where in the last inequality we use $\ln(1+x)<x$. We get,
\begin{align*}
    \exp\left(-\frac{1}{200}\cdot\frac{\lambda}{R}\right)
    & \leq \exp\left(-\frac{1}{4}\cdot \frac{b}{200\eta\omega\sqrt{\alpha}\kappa\ln(n/b\delta)}\right).
\end{align*}
Substitute $\eta = \frac{b\sqrt{\alpha}}{n}\cdot\beta$ and $\omega = \sqrt{\frac{n}{\kappa\alpha^2\gamma}}$, we get,
\begin{align*}
    \exp\left(-\frac{1}{4}\cdot\frac{\lambda}{R}\right) &\leq \exp\left(-\frac{1}{4}\cdot \frac{b}{200\frac{b\sqrt{\alpha}\beta}{n}\sqrt{\frac{n}{\kappa\alpha^2\gamma}}\sqrt{\alpha}\kappa\ln(n/b\delta)}\right)\\
    &=\exp \left(-\frac{1}{4}\cdot \frac{\sqrt{n}\sqrt{\gamma}}{200\beta\sqrt{\kappa}\ln(n/b\delta)}\right) \leq \delta\frac{b}{n}.
\end{align*}
Letting $\gamma > 1280\beta^2\ln(n/b\delta),$ $n >400 \kappa\alpha^2\gamma\ln(n/b\delta)^2$, we get the failure probability less than $\delta\cdot\frac{b}{n}$. Combining both cases, we get the following powerful concentration guarantee.
\begin{theorem}[Freedman's concentration]{\label{freedman_concentration}}
    Let $n >400\kappa\alpha^2\gamma(\ln(n/b\delta))^2$, $\gamma > 1280\beta^2\ln(n/b\delta)$, $\epsilon_0 < \frac{1}{8e^2\sqrt{\alpha}}$, $\eta = \frac{b\sqrt{\alpha}\beta}{n} \leq \min\{\frac{1}{4\sqrt{\alpha}}, \frac{b}{48\alpha^{3/2}\kappa}\}$ and $b < \min\{\frac{n}{4\alpha\beta}, \frac{\gamma\cdot n}{(96)^2\beta^2\ln(n/b\delta)^2}\}$.  Then, for any $t\geq 0$ satisfying $L_i+t \leq T_i$:
\begin{align*}
        \ln\left(\frac{\normH{\Delta_{L_i+t}}}{\normH{\Delta_{L_i}}}\right) \leq -t\ln\left(\frac{1}{1-\rho}\right) + 1,
    \end{align*}
    where $\rho=\frac{\eta}{32\sqrt{\alpha}}$, 
    with probability at least $1-\delta\frac{b}{n}$.
\end{theorem}
\begin{proof}
        By Theorem \ref{Freedman}, we have $Y_t \leq Y_0 + 1$ with probability at least $1-\delta\frac{b}{n}$. This implies,
        \begin{align}
        \left(\ln(\normH{\Delta_{L_i+t+1}}) + \sum_{j=0}^{t}{\ln\left(\frac{\normH{\Delta_{L_i+t}}}{\E_{L_i+t}\normH{\Delta_{L_i+t+1}}}\right)}\right)\cdot\ind_{\mathcal{A}^i_t} + Y^i_t\cdot\ind_{\neg\mathcal{A}^i_t} \leq \ln\left(\normH{\Delta_{L_i}}\right) + 1. \label{free_con_1}
    \end{align}
        Since $L_i + t \leq T_i$, 
        we have $\x_{L_i+t-1} \in \mathcal{U}_f(e^2\epsilon_0\eta)$ and also $\x_{L_i} \in  \mathcal{U}_f(e^2\epsilon_0\eta) $. By Lemma \ref{lma3}, and applying the union bound for $t$ inner iterations starting from $\x_{L_i}$ we have $\Pr\left(\mathcal{A}^i_t\right) \geq 1-t\delta\frac{b^2}{n^2} \geq 1- \delta\frac{b}{n}$. Combining this with (\ref{free_con_1}) and rescaling $\delta$ by a factor of $2$, we get with probability at least $1-\delta\frac{b}{n}$,
         \begin{align*}
        \left(\ln(\normH{\Delta_{L_i+t+1}}) + \sum_{j=0}^{t}{\ln\left(\frac{\normH{\Delta_{L_i+j}}}{\E_{L_i+j}\normH{\Delta_{L_i+j+1}}}\right)}\right)  \leq \ln\left(\normH{\Delta_{L_i}}\right) + 1.
    \end{align*}
    By Lemma \ref{submartingale_prop}, with $\rho=\frac{\eta}{32\sqrt{\alpha}}$, we have $ \ln\left(\frac{1}{1-\rho}\right) < \ln\left(\frac{\normH{\Delta_{L_i+j}}}{\E_{L_i+j}\normH{\Delta_{L_i+j+1}}}\right)$, for any $j$ such that $L_i+j \leq T_i$. So, with probability at least $1-\delta\frac{b}{n}$,
    \begin{align*}
        \left(\ln(\normH{\Delta_{L_i+t+1}}) + (t+1)\cdot\ln \left(\frac{1}{1-\rho}\right)\right) \leq \ln\left(\normH{\Delta_{L_i}}\right) + 1,
    \end{align*}
    which concludes the proof.
    \ifdocenter
\else
\qed
\fi
\end{proof}

\begin{remark}
We make the following remarks about the Freedman's concentration results (Theorem~\ref{freedman_concentration}):
\begin{enumerate}[(a)]
    \item An upper bound on $\normHsq{\Delta_{T_i}}$,
    \begin{align}
        \Pr\left(\normHsq{\Delta_{T_i}} \leq e^2 \normHsq{\Delta_{L_i}}\right) \geq 1-\delta\frac{b}{n}. \label{Freedman_conclusion_1}
    \end{align}
    This holds since we run the submartingale $Y_t^i$ till $L_i+t=T_i$. Applying the Freedman concentration for $t$ such that $L_i+t=T_i$ yields (\ref{Freedman_conclusion_1});
    \item In particular for $i=0$, we have,
    \begin{align}
        \Pr\left(\normHsq{\Delta_{T_0}} \leq e^2 \normHsq{\T\Delta}\right) \geq 1-\delta\frac{b}{n}; \label{Freedman_conclusion_2}
    \end{align}

    \item Suppose $T_0 < \frac{n}{b}$, by combining the definition of $T_0$ \eqref{random_times} and  \eqref{Freedman_conclusion_2} we get,
    \begin{align}
        \Pr\left(\normHsq{\Delta_{T_0}} \leq \frac{\kappa\alpha^2\gamma}{n} \normHsq{\T\Delta} \ | T_0 < \frac{n}{b} \ \right) \geq 1-\delta\frac{b}{n}.\label{Freedman_conclusion_3}
    \end{align}
\end{enumerate}
\end{remark}

The next two results show that for any stopping time $T_i$, $\normHsq{\Delta_{T_i}} \leq 2e^2\frac{\kappa\alpha^2\gamma}{n}\normHsq{\T\Delta}$ with high probability, and, moreover, if $T_i < \frac{n}{b}$ then $\normHsq{\Delta_{T_i}} < \frac{\kappa\alpha^2\gamma}{n}\normHsq{\T\Delta}$. We start by proving the result for the first stopping time $T_0$.

\begin{lemma}[High probability result for first stopping time]{\label{first_stopping time}}
Let the conditions of Theorem \ref{freedman_concentration} hold. Then:
\begin{align*}
    \Pr\left(\normHsq{\Delta_{T_0}} \leq 2e^2\frac{\kappa\alpha^2\gamma}{n}\normHsq{\T\Delta}  \right) \geq 1-\delta\frac{b}{n},
\end{align*}    
where $\beta \geq 32\ln\left(\frac{n}{2\alpha^2\kappa}\right)$, and $\gamma \geq 1280\beta^2\ln\left(n/b\delta\right)$.
\end{lemma}
\begin{proof} By conditioning on the first stopping time, whether $T_0 < \frac{n}{b}$ or $T_0 = \frac{n}{b}$, and using law of total probability, it follows that,
  \begin{align*}
    \Pr\left(\normHsq{\Delta_{T_0}} \leq 2e^2\frac{\kappa\alpha^2\gamma}{n}\normHsq{\T\Delta}\right) &= \Pr\left(\normHsq{\Delta_{T_0}} \leq 2e^2\frac{\kappa\alpha^2\gamma}{n}\normHsq{\T\Delta} \ | T_0 = \frac{n}{b}\right)\Pr\left(T_0 = \frac{n}{b}\right) \\
    &+ \Pr\left(\normHsq{\Delta_{T_0}} \leq 2e^2\frac{\kappa\alpha^2\gamma}{n}\normHsq{\T\Delta} \ | T_0 < \frac{n}{b}\right)\Pr\left(T_0 < \frac{n}{b}\right).
\end{align*}
Note that we bound the second term above by \eqref{Freedman_conclusion_3}. We proceed to bound the first term. Consider the submartingale $Y_t^0$ running till stopping time $T_0$. If $ T_0=\frac {n}{b}$, then this implies that the submartingale $Y_t^0$ continued for the full outer iteration. We apply the result of Theorem \ref{freedman_concentration} on $Y_t^0$ to get, with probability at least $1-\delta\frac{b}{n}$,
\begin{align*}
    \ln\left(\frac{\normH{\Delta_{n/b}}}{\normH{\T\Delta}}\right) \leq -\frac{n}{b}\ln\left(\frac{1}{1-\rho}\right) + 1,
\end{align*}
which implies,
\begin{align*}
    \ln\left(\frac{\normHsq{\Delta_{n/b}}}{\normHsq{\T\Delta}}\right) &\leq -\frac{2n}{b}\ln\left(\frac{1}{1-\rho}\right) + 2 \quad \Rightarrow \quad \normHsq{\Delta_{n/b}} &\leq e^2\left(1-\rho\right)^{2n/b}\normHsq{\T\Delta}.
\end{align*}
For $\beta = 32\ln\left(\frac{n}{2\alpha^2\kappa}\right) $ and $\gamma = 1280\beta^2\ln\left(n/b\delta\right)$, it follows that $\left(1-\rho\right)^{2n/b} \leq 2\frac{\kappa\alpha^2\gamma}{n}$,
which concludes the proof.
\ifdocenter
\else
\qed
\fi
\end{proof}

Since there can be more than one stopping time, we prove a result similar to Lemma \ref{first_stopping time} for stopping times occurring after the first stopping time. Note that this requires conditioning on the previous stopping time. As we already have an unconditional high probability result for $T_0$ in Lemma \ref{first_stopping time}, having conditioning on previous stopping time allows us to establish a high probability result for any subsequent stopping time.
\begin{lemma}[High probability result for non-first stopping time]{\label{nonfirst_stopping_time}}
Let the conditions of Theorem \ref{freedman_concentration} hold. Then, for any $i\geq 0$:
     \begin{align*}
         \Pr\left(\normHsq{\Delta_{T_{i+1}}} \leq 2e^2 \frac{\kappa\alpha^2\gamma}{n}\normHsq{\T\Delta} \ | \normHsq{\Delta_{T_{i}}} < \frac{\kappa\alpha^2\gamma}{n}\normHsq{\T\Delta}\right) \geq 1-\delta\frac{b}{n},
    \end{align*}
 where $\beta \geq 32\ln\left(\frac{n}{2\alpha^2\kappa}\right) $ and $\gamma \geq 1280\beta^2\ln\left(n/b\delta\right)$.
\end{lemma}
\begin{proof}
      First, note that as $\normHsq{\Delta_{T_{i}}} < \frac{\kappa\alpha^2\gamma}{n}\normHsq{\T\Delta}$, we have  $\normHsq{\Delta_{L_{i+1}-1}} < \frac{\kappa\alpha^2\gamma}{n}\normHsq{\T\Delta} $. Now due to one step high probability bound in Lemma \ref{high_prob} we have $\normHsq{\Delta_{L_{i+1}}} \leq 2\frac{\kappa\alpha^2\gamma}{n}\normHsq{\T\Delta}$. Similar to Lemma \ref{first_stopping time} we show that,
 \begin{align}
      \Pr\left(\normHsq{\Delta_{T_{i+1}}} \leq 2e^2 \frac{\kappa\alpha^2\gamma}{n}\normHsq{\T\Delta} \ | \normHsq{\Delta_{T_{i}}} < \frac{\kappa\alpha^2\gamma}{n}\normHsq{\T\Delta}, T_{i+1} =n/b\right) \geq 1-\delta\frac{b}{n}, \label{non_first_1}
  \end{align}
and,
  \begin{align}
      \Pr\left(\normHsq{\Delta_{T_{i+1}}} \leq \frac{\kappa\alpha^2\gamma}{n}\normHsq{\T\Delta} \ | \normHsq{\Delta_{T_{i}}} < \frac{\kappa\alpha^2\gamma}{n}\normHsq{\T\Delta}, T_{i+1} <n/b\right) \geq 1-\delta\frac{b}{n}. \label{non_first_2}
  \end{align}
 Now we invoke the Freedman concentration result Theorem \ref{freedman_concentration} on the sub-martingale $Y_t^{i+1}$, to get with probability at least $1-\delta\frac{b}{n}$,
     \begin{align*}
         \normHsq{\Delta_{T_{i+1}}} \leq e^2\normHsq{\Delta_{L_{i+1}}},
     \end{align*}
proving that 
     \begin{align*}
         \normHsq{\Delta_{T_{i+1}}} \leq 2e^2\frac{\kappa\alpha^2\gamma}{n}\normHsq{\T\Delta}.
     \end{align*}
     Note that this condition does not depend on the value of $T_{i+1}$. In particular, for $T_{i+1} =  \frac {n}{b}$ we get,
     \begin{align*}
         \Pr\left(\normHsq{\Delta_{T_{i+1}}} \leq 2e^2\frac{\kappa\alpha^2\gamma}{n}\normHsq{\T\Delta} \ | \normHsq{\Delta_{T_{i}}} \leq \frac{\kappa\alpha^2\gamma}{n}\normHsq{\T\Delta}, T_{i+1} = \frac{n}{b}\right) \geq 1-\delta\frac{b}{n},
    \end{align*}
    establishing the inequality (\ref{non_first_1}).
    Next, to prove the inequality (\ref{non_first_2}), observe that $T_{i+1} < \frac{n}{b}$ implies that either $\normHsq{\Delta_{T_{i+1}}} > 2e^2\normHsq{\T\Delta}$ or  $\normHsq{\Delta_{T_{i+1}}} < \frac{\kappa\alpha^2\gamma}{n}\normHsq{\T\Delta}$. We have already shown that with probability at least $1-\delta\frac{b}{n}$,
    \begin{align*}
        \normHsq{\Delta_{T_{i+1}}} &\leq 2e^2\frac{\kappa\alpha^2\gamma}{n}\normHsq{\T\Delta}
        < 2e^2\normHsq{\T\Delta}.
    \end{align*}
    Combining this with $T_{i+1}<\frac{n}{b}$ we get, 
       \begin{align*}
        \normHsq{\Delta_{T_{i+1}}} < \frac{\kappa\alpha^2\gamma}{n}\normHsq{\T\Delta},
    \end{align*}
obtaining that
\begin{align*}
        \Pr\left(\normHsq{\Delta_{T_{i+1}}} < \frac{\kappa\alpha^2\gamma}{n}\normHsq{\T\Delta} \ | \normHsq{\Delta_{T_{i}}} < \frac{\kappa\alpha^2\gamma}{n}\normHsq{\T\Delta}, T_{i+1} < \frac{n}{b}\right) \geq 1-\delta\frac{b}{n}.
    \end{align*}  
which concludes the proof.
\ifdocenter
\else
\qed
\fi
\end{proof}

Note that for the $\beta$ and $\gamma$ values prescribed in Lemma \ref{first_stopping time} and the lower bound on $n$ from Theorem \ref{freedman_concentration}, a sufficient condition on $n$ can be stated as: $n> c_1\kappa\alpha^2\ln(2c_1\kappa\alpha^2/\delta)^7$ for $c_1 = 400\cdot1280\cdot(32)^2$.
We are now ready to prove our main convergence result.

\begin{theorem}[High probability convergence over outer iterates]{\label{final_result}}
Let $n> c_1\kappa\alpha^2\ln(2c_1\kappa\alpha^2/\delta)^7$, and $\epsilon_0 < \frac{1}{8e^2\sqrt{\alpha}}$. 
Let $\T\x_0 \in \mathcal{U}_f(\epsilon_0\eta)$, $\eta = \frac{b\sqrt{\alpha}\beta}{n} $ and $b < \frac{n}{c_2\alpha\ln(n/\kappa)}$. Then with probability at least $1-2\delta$:
\begin{align*}
    f(\T\x_{s}) -f(\x^*) \leq \left(\frac{c_3\kappa\alpha^2\ln(n/\delta)^2}{n}\right)^{s}\cdot \left(f(\T\x_0) - f(\x^*)\right),
\end{align*}
for absolute constants $c_1,c_2$ and $c_3$.
\end{theorem}
\begin{proof} Note that in our running notation, $\T\x_0$ is just $\T\x$. We prove one outer iteration result where $\Tilde\x_1$ means $\x_{n/b}$.
Consider $\beta = 32\ln\left(\frac{n}{2\alpha^2\kappa}\right) $ and $\gamma = 1280\beta^2\ln\left(n/b\delta\right)$
and note that due to our assumptions, the results of Theorem~\ref{freedman_concentration} and Lemmas \ref{first_stopping time} and \ref{nonfirst_stopping_time} hold. Also for the given values of $\beta$ and $\gamma$, the upper bound condition on $b$ can be replaced with $b < \frac{n}{c_2\alpha\ln(n/2\alpha^2\kappa)}$, where $c_2 = 32\cdot 4$.
Let $\mathcal{S}$ be the event denoting our convergence guarantee, i.e., 
\begin{align*}
    \mathcal{S} :=\left\{\normHsq{\Delta_{n/b}} \leq 2e^2\frac{\kappa\alpha^2\gamma}{n}\normHsq{\T\Delta}\right\}.
\end{align*}
Also we define events $\mathcal{E}_i$ for all $i\geq 0$ as following,
\begin{align*}
    \mathcal{E}_i := \left\{T_i < \frac{n}{b}\right\}\cap\left\{\normHsq{\Delta_{T_i}} < \frac{\kappa\alpha^2\gamma}{n}\normHsq{\T\Delta}\right\}.
\end{align*}
Due to Lemma \ref{nonfirst_stopping_time} we know that $\Pr\left(\mathcal{E}_{i+1} | \mathcal{E}_i, T_{i+1} < \frac{n}{b}\right) \geq  1-\delta\frac{b}{n}$. Moreover, by using the law of total probability in conjunction with conditioning on the first stopping time, we have,
\begin{align*}
    \Pr\left(\mathcal{S}\right) = \Pr\left(\mathcal{S} \ | T_0=\frac{n}{b}\right)\cdot \Pr\left(T_0=\frac{n}{b}\right) + \Pr\left(\mathcal{S} \ | T_0<\frac{n}{b}\right)\cdot \Pr\left(T_0<\frac{n}{b}\right).
\end{align*}
Using Lemma \ref{first_stopping time}, we have $ \Pr\left(\mathcal{S} \ | T_0=\frac{n}{b}\right) \geq 1-\delta\frac{b}{n}$, and
\begin{align}
    \Pr\left(\mathcal{S}\right) = \left(1-\delta\frac{b}{n}\right)\cdot \Pr\left(T_0=\frac{n}{b}\right) + \Pr\left(\mathcal{S} \ | T_0<\frac{n}{b}\right)\cdot \Pr\left(T_0<\frac{n}{b}\right). \label{second_last_0}
\end{align}
Next, we consider the term $\Pr\left(\mathcal{S} \ | T_0<\frac{n}{b}\right)$,
\begin{align*}
    \Pr\left(\mathcal{S} \ | T_0<\frac{n}{b}\right) \geq \Pr\left(\mathcal{S} \ | \mathcal{E}_0, T_0<\frac{n}{b}\right)\cdot \Pr\left(\mathcal{E}_0 \ | \ T_0 < \frac{n}{b}\right).
\end{align*}
By Lemma \ref{first_stopping time}, $\Pr\left(\mathcal{E}_0 \ | \ T_0 < \frac{n}{b}\right) \geq 1-\delta\frac{b}{n}$. Substituting this in (\ref{second_last_0}), 
\begin{align}
     \Pr\left(\mathcal{S}\right) \geq \left(1-\delta\frac{b}{n}\right)\left[ \Pr\left(T_0=\frac{n}{b}\right)+\Pr\left(\mathcal{S} \ | \mathcal{E}_0\right)\cdot\Pr\left(T_0<\frac{n}{b}\right)\right]. \label{second_last_1}
\end{align}
We now show that $\Pr\left(\mathcal{S} \ | \mathcal{E}_i  \ \right) \geq 1-\delta$ for any $i$. Consider the following,
\begin{align*}
    \Pr\left(\mathcal{S} \ | \mathcal{E}_i  \ \right) =  &\Pr\left(\mathcal{S} \ | T_{i+1} = \frac{n}{b}, \mathcal{E}_i  \ \right)\cdot \Pr\left(T_{i+1} = \frac{n}{b} \ | \ \mathcal{E}_i\right) \\
    &\quad +\Pr\left(\mathcal{S} \ | T_{i+1} < \frac{n}{b}, \mathcal{E}_i  \ \right)\cdot \Pr\left(T_{i+1} < \frac{n}{b} \ | \ \mathcal{E}_i\right).
\end{align*}
By Lemma \ref{nonfirst_stopping_time}, we have $\Pr\left(\mathcal{S} \ | T_{i+1} = \frac{n}{b}, \mathcal{E}_i  \ \right) \geq 1-\delta\frac{b}{n}$, 
\begin{align}
    \Pr\left(\mathcal{S} \ | \mathcal{E}_i  \ \right) \geq \left(1-\delta\frac{b}{n}\right)\cdot &\Pr\left(T_{i+1} = \frac{n}{b} \ | \ \mathcal{E}_i\right) \nonumber \\ +&\Pr\left(\mathcal{S} \ | T_{i+1} < \frac{n}{b}, \mathcal{E}_i  \ \right)\cdot \Pr\left(T_{i+1} < \frac{n}{b} \ | \ \mathcal{E}_i\right). \label{second_last_2}
\end{align}
We write $\Pr\left(\mathcal{S} \ | T_{i+1} < \frac{n}{b}, \mathcal{E}_i  \ \right)$ as follows,
\begin{align*}
    \Pr\left(\mathcal{S} \ | T_{i+1} < \frac{n}{b}, \mathcal{E}_i  \ \right) &\geq \Pr\left(\mathcal{S} \ | T_{i+1} < \frac{n}{b},\mathcal{E}_i,\mathcal{E}_{i+1}  \ \right)\cdot\Pr\left(\mathcal{E}_{i+1} | T_{i+1} < \frac{n}{b}, \mathcal{E}_i\right)\\
    &=\Pr\left(\mathcal{S} \ | \ \mathcal{E}_{i+1}  \ \right)\cdot\Pr\left(\mathcal{E}_{i+1} | T_{i+1} < \frac{n}{b}, \mathcal{E}_i\right),
\end{align*}
where in the last inequality we use $\Pr\left(\mathcal{S} \ | T_{i+1} < \frac{n}{b},\mathcal{E}_i,\mathcal{E}_{i+1}  \ \right) = \Pr\left(\mathcal{S} \ | \mathcal{E}_{i+1}  \ \right)$. We also use Lemma \ref{nonfirst_stopping_time} to write $\Pr\left(\mathcal{E}_{i+1} | T_{i+1} < \frac{n}{b}, \mathcal{E}_i\right) \geq 1-\delta\frac{b}{n}$, implying,
\begin{align}
    \Pr\left(\mathcal{S} \ | T_{i+1} < \frac{n}{b}, \mathcal{E}_i  \ \right) &\geq \left(1-\delta\frac{b}{n}\right) \Pr\left(\mathcal{S} \ | \mathcal{E}_{i+1}  \ \right). \label{second_last_3}
\end{align}
Substituting (\ref{second_last_3}) in (\ref{second_last_2}), we get,
\begin{align*}
    \Pr\left(\mathcal{S} \ | \ \mathcal{E}_i\right) \geq \left(1-\delta\frac{b}{n}\right)\left[ \Pr\left(T_{i+1} = \frac{n}{b} \ | \ \mathcal{E}_i\right)+ \Pr\left(\mathcal{S} \ | \mathcal{E}_{i+1}  \ \right)\Pr\left(T_{i+1} < \frac{n}{b} \ | \ \mathcal{E}_i\right)  \right].
\end{align*}
Since we perform a finite number of iterations, $\Pr\left(T_{i+1} < \frac{n}{b} \ | \ \mathcal{E}_i\right) =0$ for any $i \geq \frac{n}{b}$. This intuitively means that the number of stopping times is upper bounded by the number of inner iterations. In the worst case $i=\frac{n}{b}-1$, which means that $\Pr\left(T_{i+1} < \frac{n}{b} \ | \ \mathcal{E}_i\right) =0$ for $i = \frac{n}{b}-1$. implying that
$\Pr\left(\mathcal{S}|\mathcal{E}_{\frac{n}{b}-1}\right) \geq 1-\delta\frac{b}{n}$. Also note that if $\Pr\left(\mathcal{S}\ | \ \mathcal{E}_{i+1}\right) \geq \left(1-\delta\frac{b}{n}\right)^l$ for some integer $l \geq 0$, then $\Pr\left(\mathcal{S}\ | \ \mathcal{E}_{i}\right) \geq \left(1-\delta\frac{b}{n}\right)^{l+1}$. Combining this with the fact there cannot be more than $\frac{n}{b}$ stopping times, along with using (\ref{second_last_1}) we get,
\begin{align*}
    \Pr\left(\mathcal{S}\right) \geq 1-\delta.
\end{align*}
Incorporating the total failure probability of Lemma \ref{lma3} for any of the $n/b$ iterations, we get a failure probability of at most $2\delta$. 

As a final step, we prove our result in terms of function values. Since $f$ has continuous first- and second-order derivatives, by the 
quadratic Taylor expansion, for vectors $\ai$ and $\vi$, there exists a $\theta \in [0,1]$ such that,
\begin{align*}
 f(\ai+\vi) = f(\ai) + \langle \nabla{f}(\ai), \vi \rangle + \frac{1}{2}\vi^\top\nabla^2{f}(\ai+\theta \vi)\vi.
\end{align*}
Let $\ai=\x^*$, $\vi=\x_{n/b} - \x^*$,
\begin{align*}
f(\x_{n/b}) - f(\x^*) = \frac{1}{2}(\x_{n/b}-\x^*)^\top\nabla^2{f}(\x^* + \theta (\x_{n/b}-\x^*))(\x_{n/b}-\x^*).
\end{align*}
With high probability, we know that $\x_{n/b} \in \mathcal{U}_f(e^2\epsilon_0\eta)$ and have $\x^* + \theta (\x_{n/b}-\x^*) \in \mathcal{U}_f(e^2\epsilon_0\eta) $. Take $\z = \x^* + \theta (\x_{n/b}-\x^*)$, we have $\frac{1}{1+e^2\epsilon_0\eta}\cdot\Hi\preceq \nabla^2{f}(\z) \preceq (1+e^2\epsilon_0\eta)\cdot\Hi$. Then, 
\begin{align*}
\frac{1}{2}(\x_{n/b}-\x^*)^\top\nabla^2{f}(\x^* + \theta (\x_{n/b}-\x^*)) ^\top (\x_{n/b}-\x^*) &= \frac{1}{2}\nsq{\x_{n/b}-\x^*}_{\nabla^2{f}(\z)}\\
& \leq \frac{1}{2}(1+ e^2\epsilon_0\eta)\normHsq{\Delta_{n/b}}\\
& < \normHsq{\Delta_{n/b}},
\end{align*}
so it follows that,
\begin{align*}
f(\x_{n/b}) - f(\x^*) \leq \normHsq{\Delta_{n/b}}.
\end{align*}
Using the reverse inequality with $\T\x$ in place of $\x_{n/b}$, we have,
\begin{align*}
    f(\T\x) - f(\x^*) \geq 2(1+\epsilon_0)\normHsq{\T\Delta}.
\end{align*}
Combining the above two relations with the definition of $\mathcal{S}$, we get with probability at least $1-2\delta$,
\begin{align*}
    f(\x_{n/b}) -f(\x^*) \leq 6e^2\frac{\kappa\alpha^2\gamma}{n}\cdot \left(f(\T\x) - f(\x^*)\right).
\end{align*}
Setting $c_3 =6e^2\cdot 1280\cdot32$, and writing the result in terms of outer iterates by noticing $\T\x = \T\x_0$ and $\x_{n/b}=\T\x_1$, concludes the proof.
\ifdocenter
\else
\qed
\fi
\end{proof}

\subsection{Global convergence analysis}{\label{global}}

In this section, we prove the global convergence of $\texttt{Mb-SVRN}$. Unlike Theorem \ref{final_result} we have no local neighborhood condition and we do not put any condition on the gradient mini-batch size. Due to being a stochastic second-order method, the global rate of convergence of $\texttt{Mb-SVRN}$ is provably much slower than the local convergence guarantee.

\begin{theorem}[Global convergence of \texttt{Mb-SVRN}]{\label{global_convergence}}
For any gradient mini-batch size $b$ and $\eta = \min\left\{\frac{2}{\kappa\sqrt{\alpha}}, \frac{b}{8\kappa^3\alpha^{3/2}}\right\}$ there exists $\rho' <1$ such that,
    \begin{align*}
        \E f(\x_{n/b}) -f(\x^*) < \rho'\cdot \left(f(\T\x) - f(\x^*)\right).
    \end{align*}

\end{theorem}

\begin{proof}
Refer to Appendix \ref{p_global_convergence}.
\end{proof}

\section{Numerical Experiments}{\label{Experiments}}
We now present more extensive
empirical evidence to support our theoretical analysis. We considered a regularized logistic loss minimization task with two datasets, \texttt{EMNIST} dataset ($n \approx 700k$) and \texttt{CIFAR10} dataset ($n = 60k$) transformed using a random feature map obtaining $d=256$ features for both datasets. In addition to the above two datasets, we performed experiments on a logistic regression task with synthetic data (with $n = 500k$, $d=256$); see Appendix \ref{synthetic_experiments}.

\subsection{Experimental setup}{\label{experiment_setup}}
 The regularized logistic loss function can be expressed 
 as the following finite-sum objective function: 
\begin{align*}
    f(\x) = \frac{1}{n}\sum_{i=1}^{n}{\ln(1+e^{-b_i\ai_i^\top\x})} + \frac{\mu}{2}\nsq{\x}.
\end{align*}
where $(\ai_i,b_i)$ denote the training examples with $\ai_i \in \R^d$ and $b_i \in \{+1,-1\}$ for $i \in \{1,2,\dots,n\}$. The function $f(\x)$ is $\mu$-strongly convex. We set $\mu=10^{-6}$ in our experiments on \texttt{EMNIST} and \texttt{CIFAR10}. Before running \texttt{Mb-SVRN}, for all $3$ datasets, we find the respective $\x^*$ to very high precision by using 
Newton's Method. This is done in order to calculate the error $f(\x_t)-f(\x^*)$ at every iteration $t$.
For the Hessian approximation in \texttt{Mb-SVRN}, we use Hessian subsampling, e.g., as in \cite{roosta2019sub,bollapragada2019exact}, sampling $h$ component Hessians at the snapshot vector $\T\x$ and constructing $\hat{\Hi}_s$. The algorithmic implementation is discussed in Algorithm \ref{alg}.
For every combination of gradient mini-batch size ($b$), step size ($\eta$) and Hessian sample size ($h$) chosen for any particular dataset, we tune the number of inner iterations $t_{\max}$ to optimize the convergence rate per data pass and take the average over $10$ runs. We run the experiments for a wide range of step size values $\eta$  in the interval $[2^{-10},2]$.

We compute the convergence rate per data pass experimentally with an approximation scheme described in this section. First, note that one full data pass is done at the start of the outer iteration to compute $\nabla{f}(\T\x)$. Thereafter, for chosen gradient mini batch-size $b$, $2b$ component gradients are computed at every inner iteration ($b$ component gradients each for $\x_t$ and $\T\x$, respectively). Thus, after $t$ inner iterations the number of data passes is given as:
\begin{align*}
    \text{Number of data passes} \ (w) &= 1 + \frac{2bt}{n},\\
\end{align*}
Once we have the number of data passes $(w)$, we can compute the convergence rate per data pass at $t^{th}$ inner iterate as the following:
\begin{align*}
    \text{Convergence rate per data pass} \ (\rho) = \left(\frac{f(\x_t)-f(\x^*)}{f(\T\x)-f(\x^*)}\right)^{1/w}
\end{align*}
where $\x_t$ is the $t^{th}$ inner iterate and $\T\x$ is the outer iterate. Since it is infeasible to find $\rho$ at every inner iterate with an aim to find $t$ that minimizes $\rho$, we use the following an approximation scheme.

We calculate the convergence rate per data pass ($\rho$) once after every $200$ inner iterations. If our calculations yield that $\rho$ has increased twice continuously or the method starts diverging $(\rho >1)$, we halt and return the minimum value of $\rho$ recorded thus far. Furthermore, for very large values of $b$, we increase the frequency of computing $\rho$ (compute $\rho$ after every $\frac{n}{2b}$ inner iterations). In the following sections, we discuss the robustness of \texttt{Mb-SVRN} to $(1)$ gradient mini-batch size and $(2)$ step size.

\subsection{Robustness to gradient mini-batch size}{\label{gradient_experiments}}

\begin{figure}
\ifdocenter
\centering
\vspace{-1cm}
\fi
\hspace{\shift}\begin{tabular}{cccc}
\subfloat[Convergence rate per data pass of \texttt{Mb-SVRN} on \texttt{EMNIST} dataset.]{\includegraphics[width =5in]{Figs/EMNIST/EMNIST_all.png}}&\\
\subfloat[Convergence rate per data pass of \texttt{Mb-SVRN} on \texttt{CIFAR10} dataset.]{\includegraphics[width = 5in]{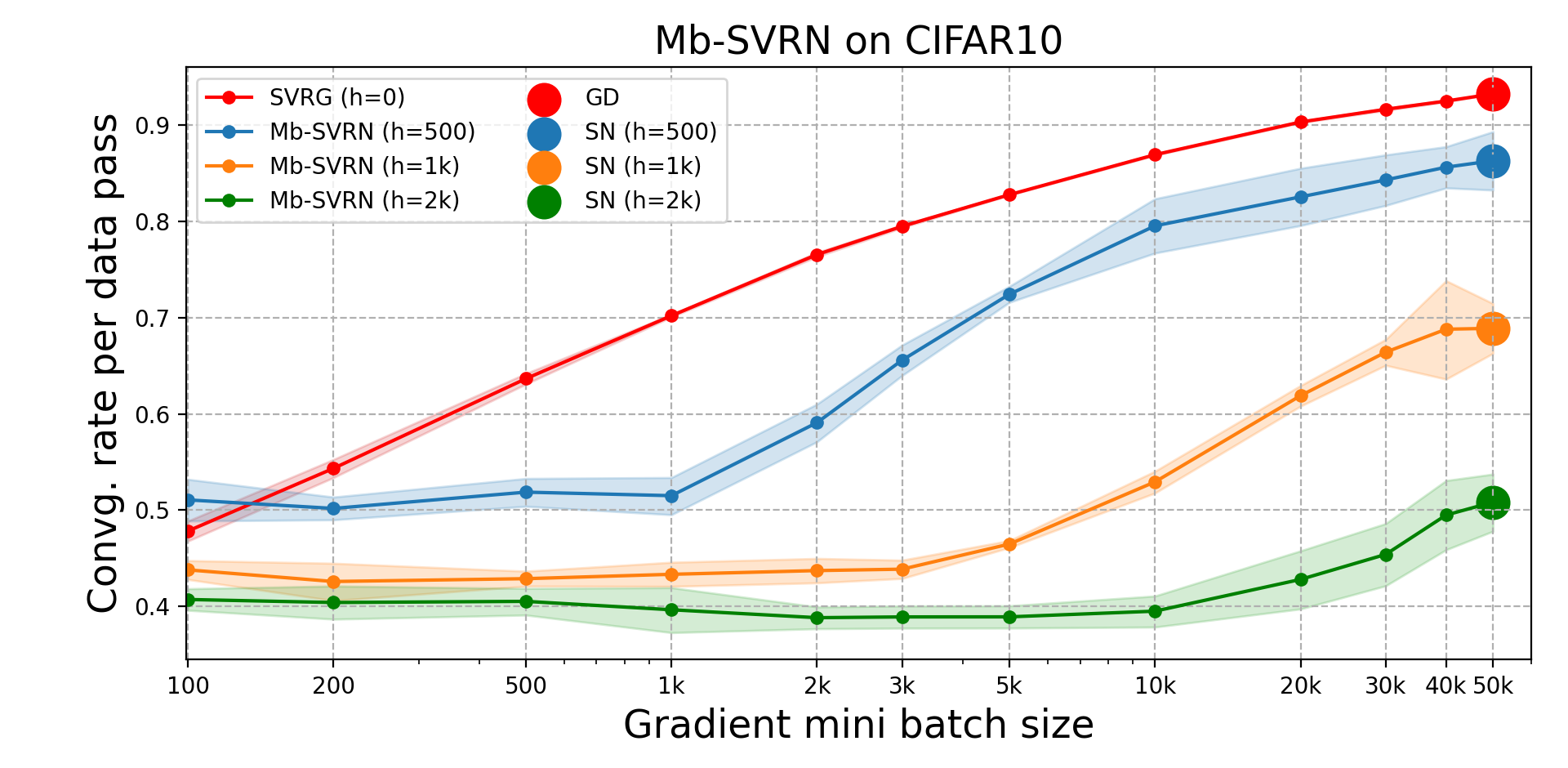}}&\\
\subfloat[Convergence rate per data pass of \texttt{Mb-SVRN} on the synthetic dataset.]{\includegraphics[width = 5in]{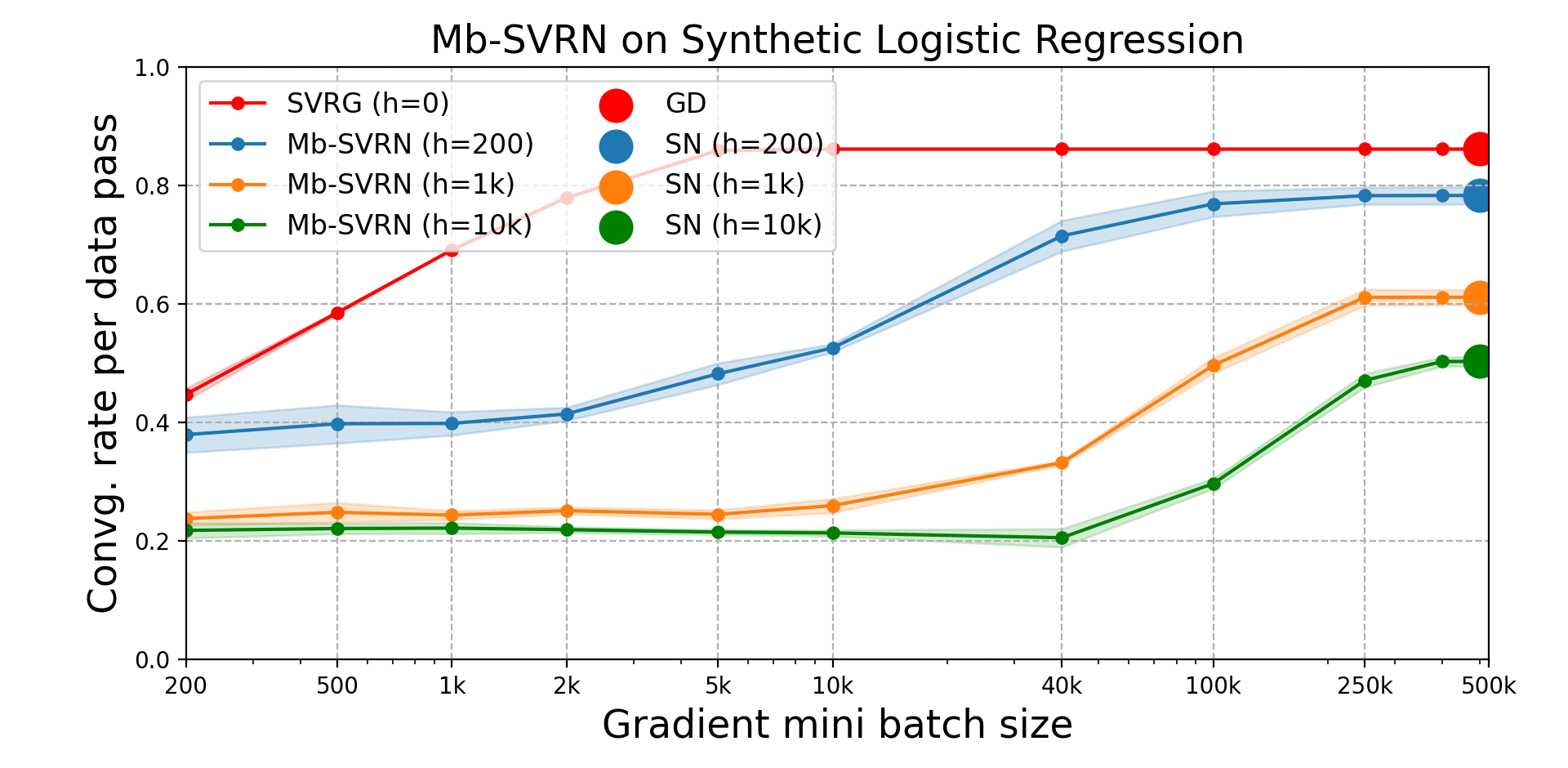}}
\end{tabular}
\caption{Experiments on \texttt{EMNIST}, \texttt{CIFAR10} and the synthetic dataset, as we vary gradient mini-batch size $b$ and Hessian sample size $h$, showing the robustness of \texttt{Mb-SVRN} to gradient mini-batch size and phase transition into standard Newton's method for large mini-batches.}
\label{robust_mini_batch}
\end{figure}
\def\ww{3in}
\begin{figure}
\ifdocenter
\centering
\fi
\hspace{-2cm}
\begin{tabular}{cc}
\subfloat
{\includegraphics[width = \ww]{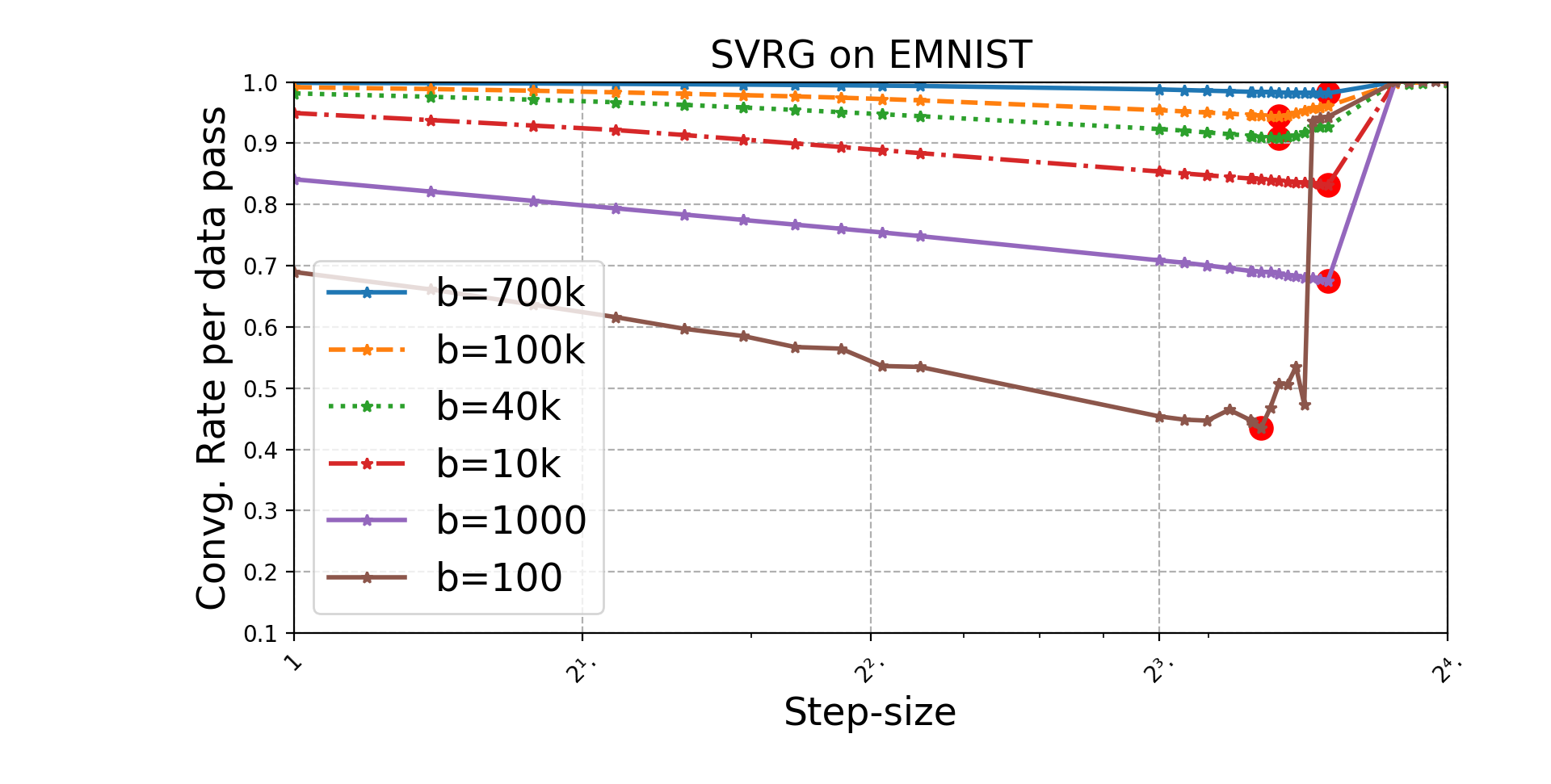}}&
\hspace{-0.75cm}\subfloat
{\includegraphics[width = \ww]{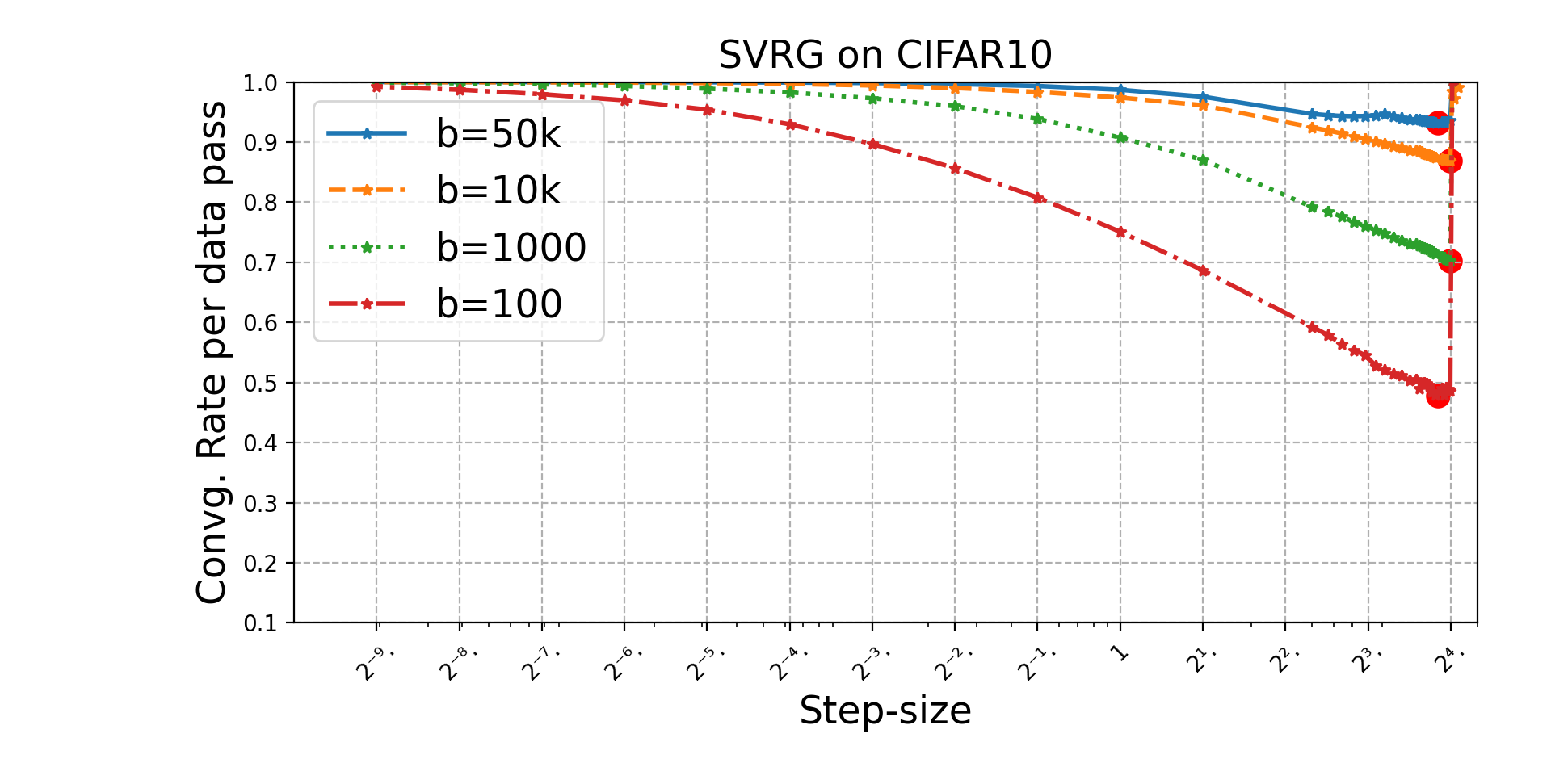}}\\
\subfloat
{\includegraphics[width = \ww]{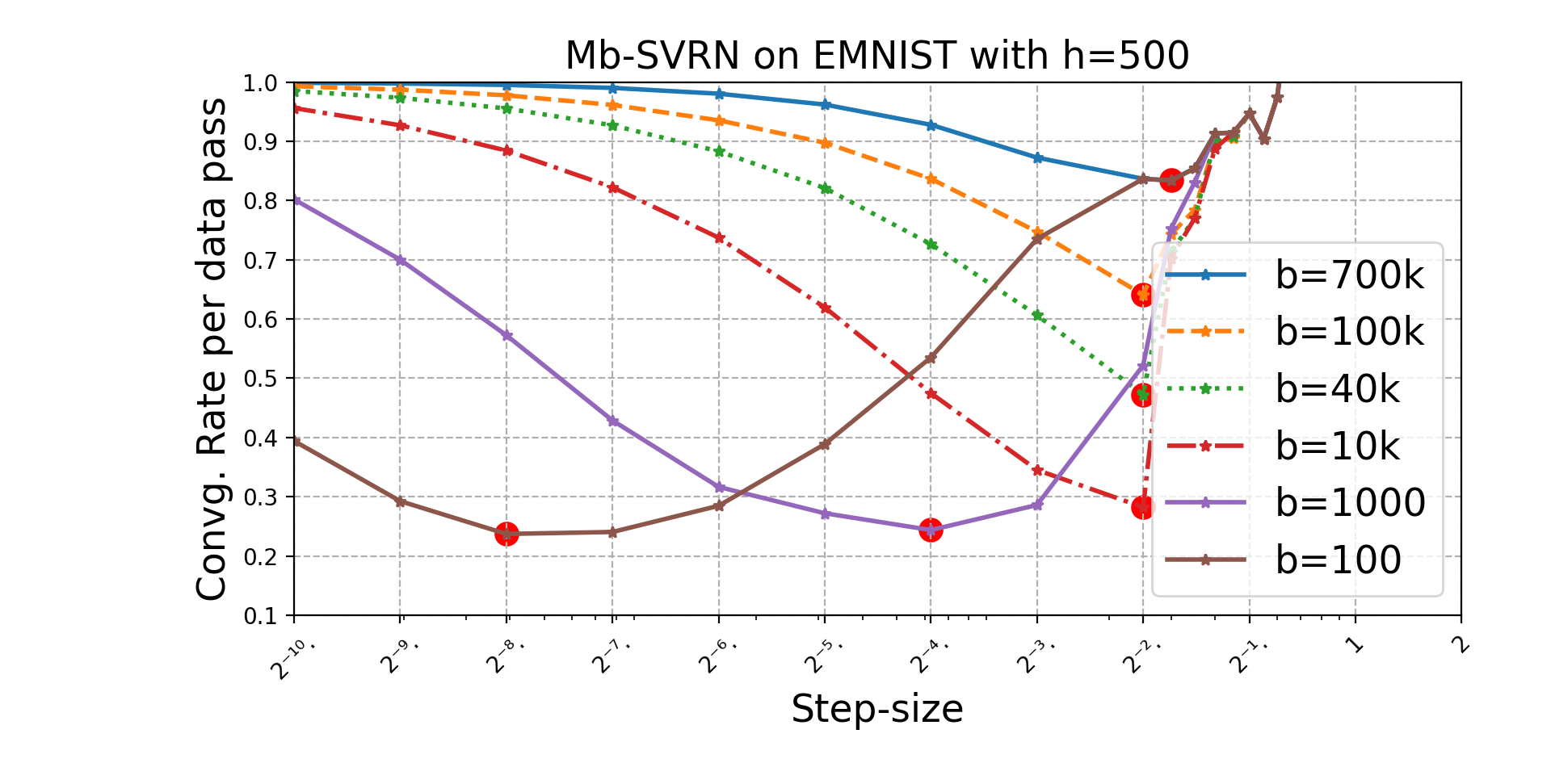}}&
\hspace{-0.75cm}\subfloat
{\includegraphics[width = \ww]{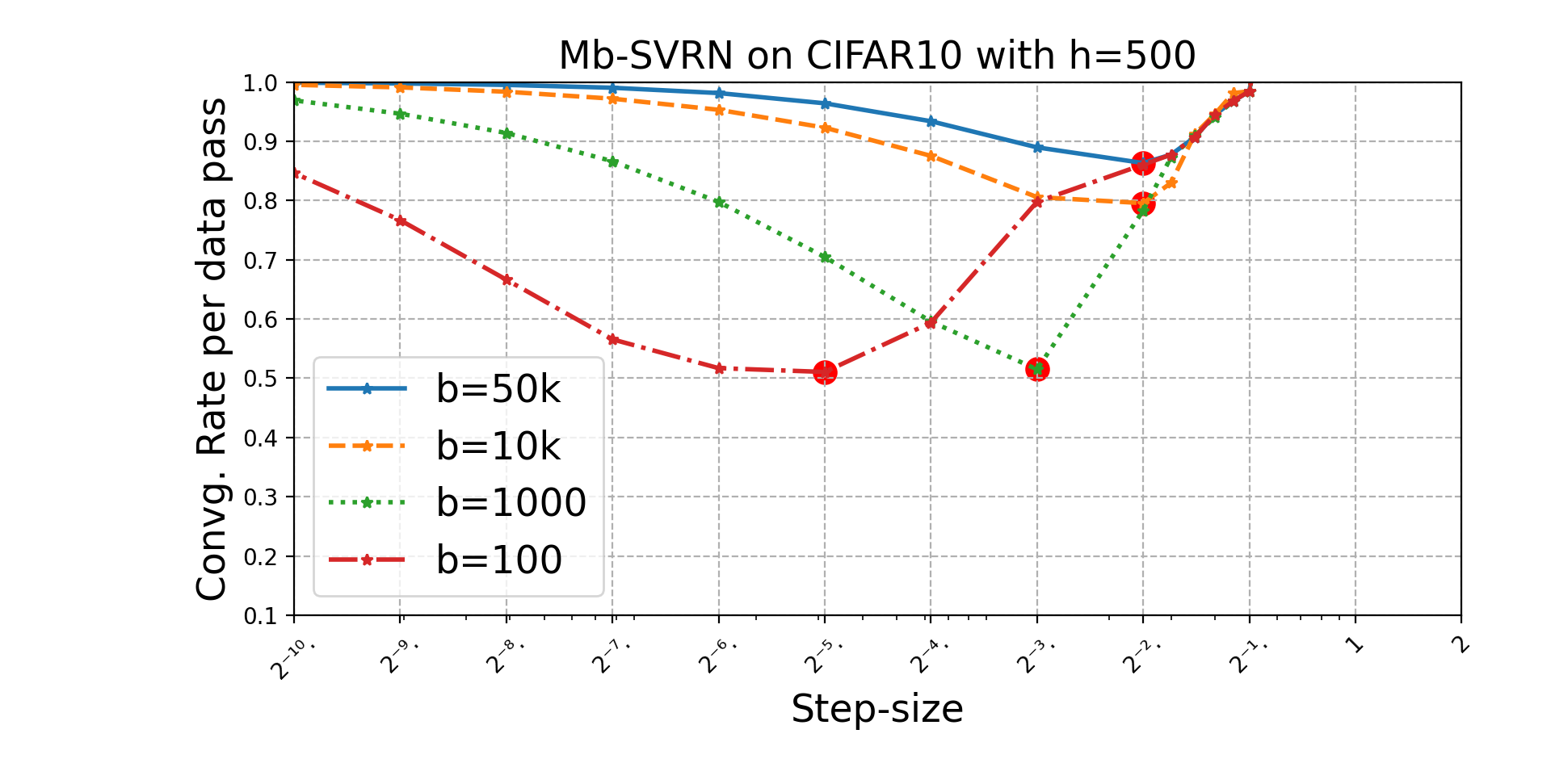}}\\
\subfloat
{\includegraphics[width = \ww]{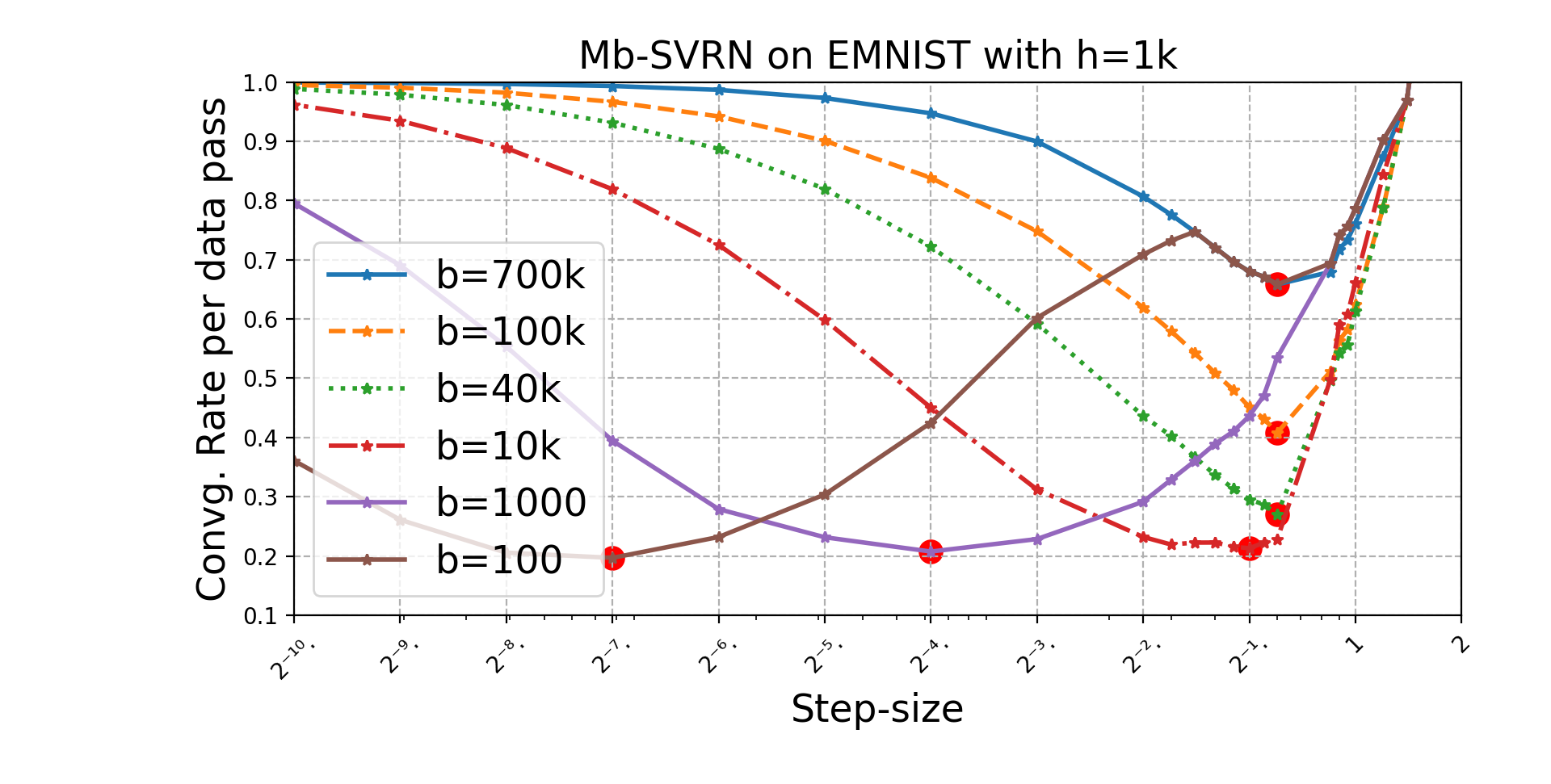}}&
\hspace{-0.75cm}\subfloat
{\includegraphics[width = \ww]{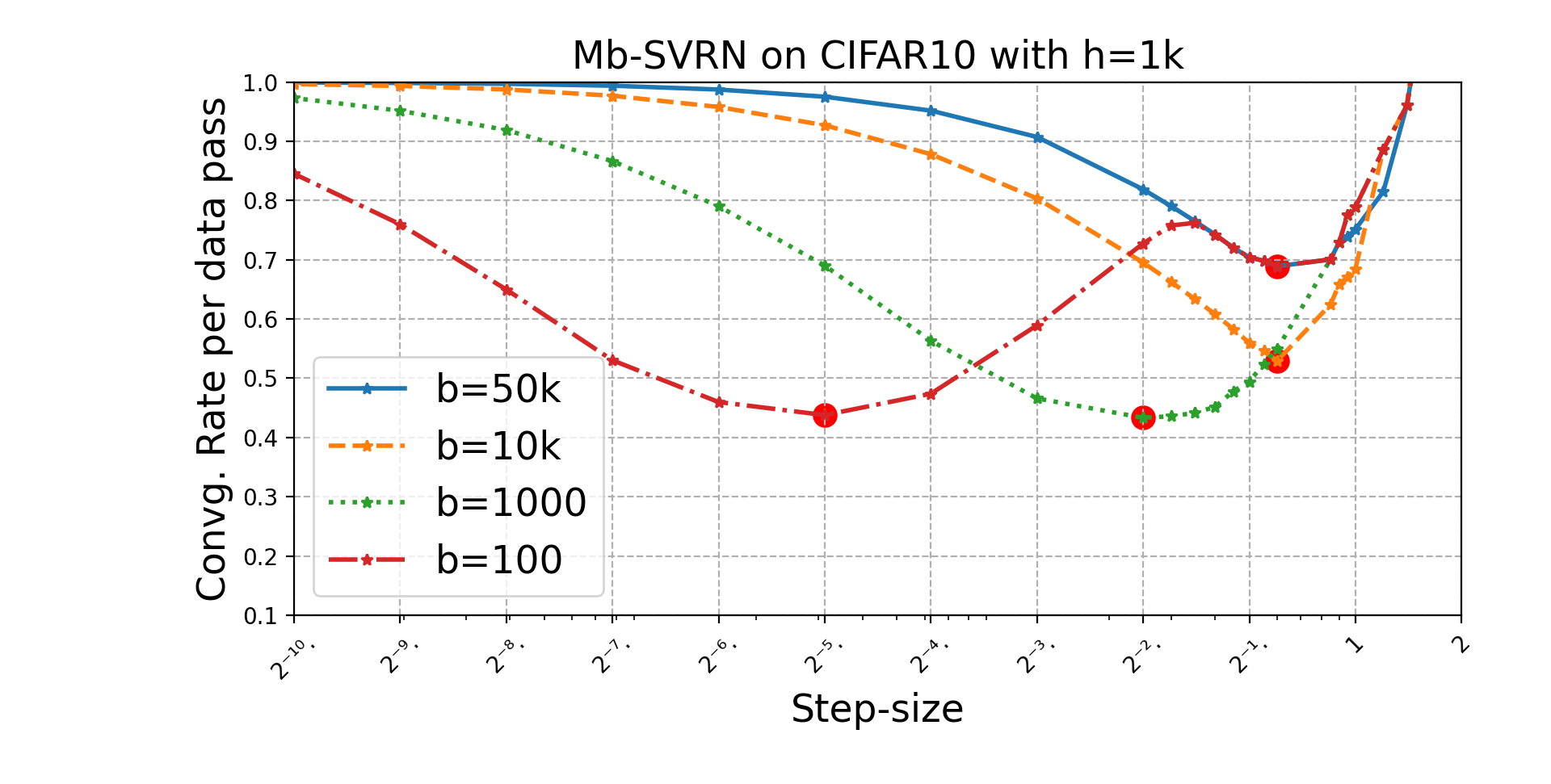}}\\
\subfloat
{\includegraphics[width = \ww]{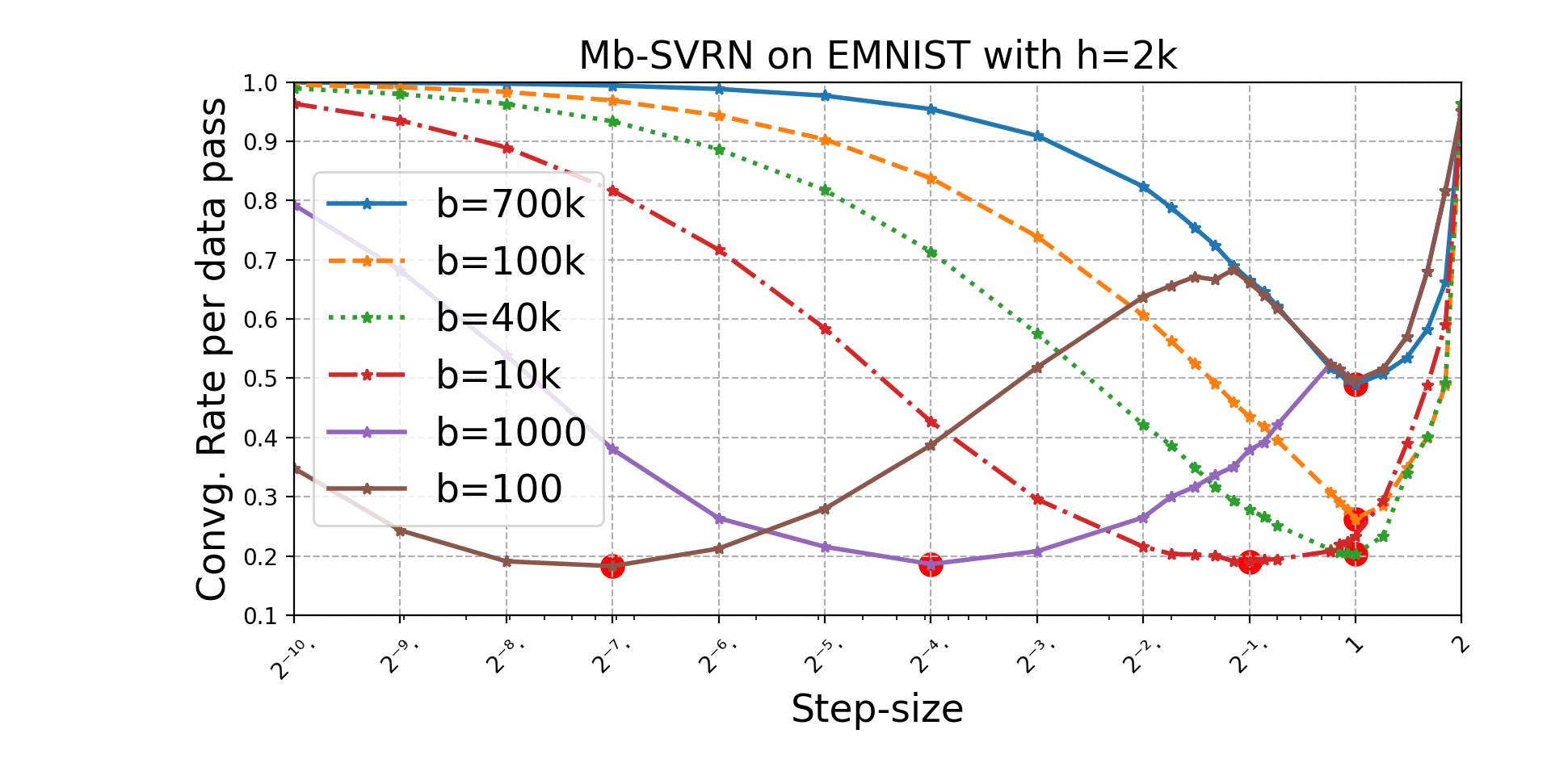}}&
\hspace{-0.75cm}\subfloat
{\includegraphics[width = \ww]{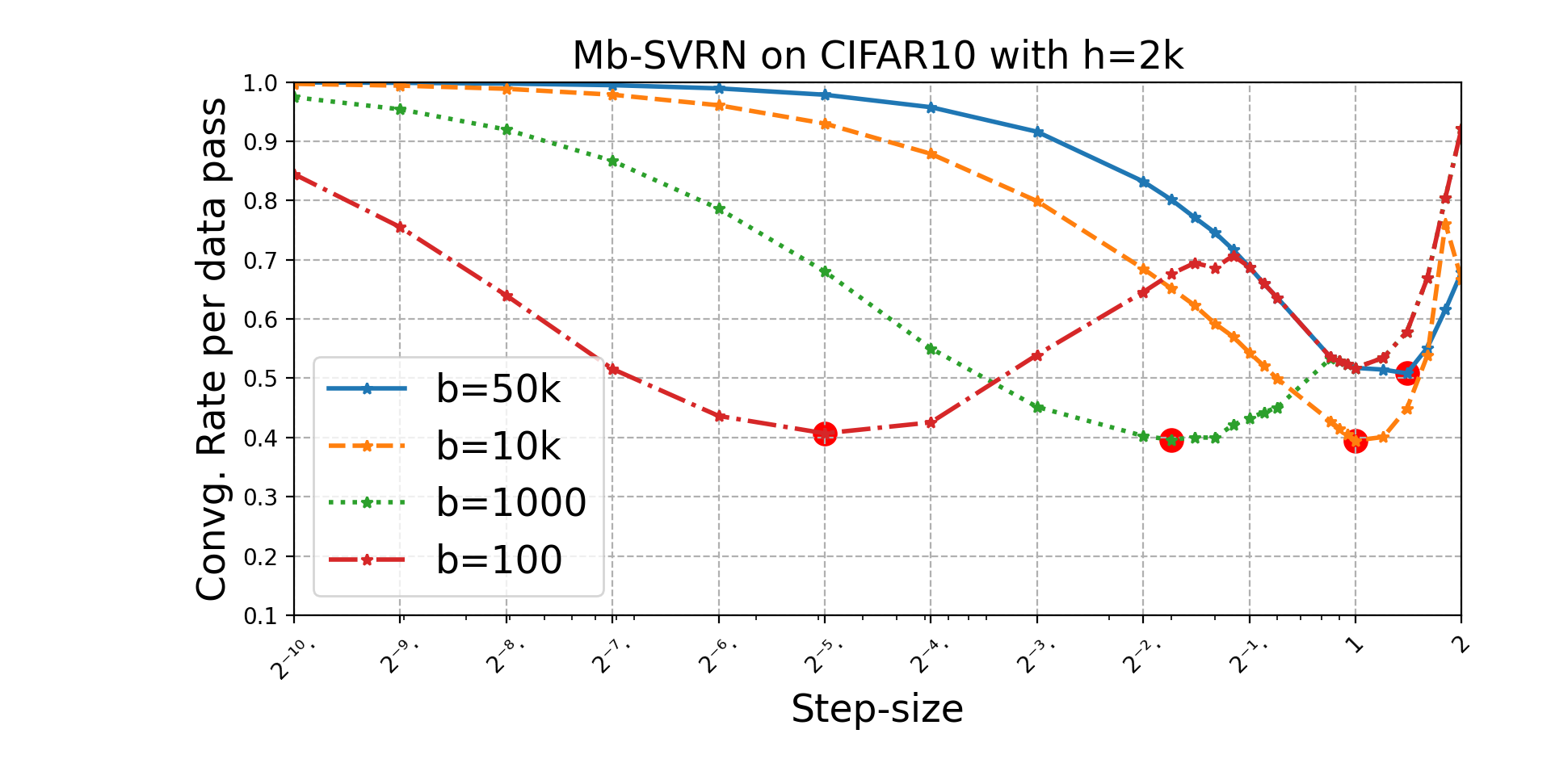}}
\end{tabular}
\caption{Experiments with logistic regression on \texttt{EMNIST} on \texttt{CIFAR10} datasets. The red dots on every curve mark the respective optimal convergence rate attained at the optimal step size. The bottom six plots demonstrate the performance of \texttt{Mb-SVRN} with different Hessian sample sizes $h$. The top two plots demonstrate \texttt{SVRG}'s performance.}
\label{robust_step-size}
\end{figure}

We first discuss Figure \ref{robust_mini_batch} showing the convergence rate of \texttt{Mb-SVRN} as we vary gradient mini-batch size and Hessian sample size for solving regularized logistic regression tasks on \texttt{EMNIST} and \texttt{CIFAR10} datasets. The convergence rates reported are obtained after tuning the convergence rate per data pass with respect to the step size $\eta$ and number of inner iterations, for any $(h,b)$ value pair. The theory suggests that the convergence rate of \texttt{Mb-SVRN} is independent of the gradient mini-batch size $b$ for a wide range of mini-batch sizes, and the plot recovers this phenomenon remarkably accurately. The plot highlights that \texttt{Mb-SVRN} is robust to gradient mini-batch size, since the fast convergence rate of \texttt{Mb-SVRN} is preserved for a very large range of gradient mini-batch sizes (represented by the curves in Figure \ref{robust_mini_batch} staying flat). This can also be observed by noticing the almost straight line alignment of red marker dots in the \texttt{Mb-SVRN} plots in Figure \ref{robust_step-size}. This phenomenon does not hold for first-order variance reduction methods. The plots demonstrate that the performance of \texttt{SVRG} suffers with increasing $b$, which is consistent with the existing convergence analysis of \texttt{SVRG}-type first-order methods \cite{konevcny2013semi,allen2017katyusha}.

We also note that as $b$ increases and enters into a very large gradient mini-batch size regime, the convergence rate of \texttt{Mb-SVRN} starts to deteriorate and effectively turns into Subsampled Newton (\texttt{SN}) when $b=n$. The empirical evidence showing deterioration in convergence rate for very large $b$ values agrees with our theoretical prediction of a phase transition into standard Newton's method after $b> \frac{n}{\alpha\log(n/\kappa)}$. In the extreme case of $b=n$, \texttt{Mb-SVRN} performs only one inner iteration and is the same as \texttt{SN}. As \texttt{SN} uses the exact gradient at every  iteration, its convergence rate per data pass is very sensitive to the Hessian approximation quality, which makes it substantially worse than \texttt{Mb-SVRN} for small-to-moderate Hessian sample sizes $h$.

\subsection{Robustness to step size}{\label{step_size_experiments}}
In addition to the demonstrated robustness to gradient mini-batch size, \texttt{Mb-SVRN} exhibits empirical resilience to step size variations. As depicted in Figure \ref{robust_step-size}, the convergence rate for small $b$ values closely aligns with the optimal rate (red dot) when the step size approaches the optimal range. In contrast, the convergence rate of Subsampled Newton $(b=n)$ sharply increases near its optimal step size. This suggests that the convergence rate of \texttt{Mb-SVRN} with small-to-moderate gradient mini-batch size is more robust to changes in step size as compared to using very large gradient mini-batches or the full gradients. Intuitively, with smaller $b$ values, the algorithm performs numerous inner iterations, relying more on variance reduction. This advantage offsets the impact of step size changes. Conversely, with very large $b$ values, \texttt{Mb-SVRN} selects larger step sizes, reducing the number of optimal inner iterations and limiting the variance reduction advantage. On the other hand, for \texttt{SVRG} with very large gradient mini-batches, deviating slightly from the optimal step size can significantly impact the convergence rate per data pass, as seen with the sharp changes in convergence in the top two plots of Figure \ref{robust_step-size}.

\section{Conclusions and Future Directions}

We have shown that incorporating second-order information into a stochastic variance-reduced method allows it to scale more effectively, and to scale to very large mini-batches. We have demonstrated this by analyzing the convergence of \texttt{Mb-SVRN}, a prototypical stochastic second-order method with variance reduction, and have shown that its associated convergence rate per data pass remains optimal for a very wide range of gradient mini-batch sizes (up to $n/\alpha\log(n)$).
Our main theoretical result provides a convergence guarantee robust to the gradient mini-batch size with high probability through a novel martingale concentration argument. Furthermore, empirically we have shown the robustness of \texttt{Mb-SVRN} not only to mini-batch size, but also to the step size and the Hessian approximation quality. Our algorithm, analysis, and implementation uses \texttt{SVRG}-type variance reduced gradients, and as such, a natural question pertains to whether the algorithm can be extended to use other (perhaps biased) variance reduction techniques. Another interesting future direction is to investigate the effect of using alternate sampling methods while selecting component gradients, as well as the effect of incorporating  acceleration into the method.


 \bibliographystyle{spmpsci}      
\bibliography{main}   

\appendix
\section{Proofs}{\label{appendix}}

\subsection{Proof of Lemma \ref{lma1}.}{\label{plma1}}
\begin{proof}
     As each $\psi_i$ is convex and $\lambda$-smooth, the following relation holds (see Theorem 2.1.5 in \cite{nesterov2018lectures}),
\begin{align*}
\nsq{\nabla{\psi}_i(\x_t) - \nabla{\psi}_i(\T{\x})} \leq 2\lambda\cdot(\psi_i(\T{\x}) - \psi_i(\x_t) - \langle \nabla{\psi}_i(\x_t) , \T{\x} - \x_t \rangle).
\end{align*}
Consider the variance of the stochastic gradient if we use just one sample, $\psi_i$ for calculating the stochastic gradient. The variance is given as $\E[\nsq{\nabla{\psi}_i(\x_t) - \nabla{\psi}_i(\T{\x}) +\g(\T{\x}) - \g(\x_t)} ]$.
\begin{align}
\E_t[\nsq{\nabla{\psi}_i(\x_t) - \nabla{\psi}_i(\T{\x}) +\g(\T{\x}) - \g(\x_t)}]& = \E_t[\nsq{\nabla{\psi}_i(\x_t) - \nabla{\psi}_i(\T{\x}) - (\g(\x_t)-\g(\T{\x}))}] \nonumber \\
& \leq \E_t[\nsq{\nabla{\psi}_i(\x_t) - \nabla{\psi}_i(\T{\x})}] \nonumber\\
& \leq 2\lambda\cdot\E_t[(\psi_i(\T{\x}) - \psi_i(\x_t) - \langle \nabla{\psi}_i(\x_t) , \T{\x} - \x_t \rangle)] \nonumber\\
& = 2\lambda \cdot (f(\T{\x}) - f(\x_t) - \langle \g(\x_t) , \T{\x} - \x_t \rangle). \label{var1}
\end{align}
Similarly, if we use $b$ samples and upper bound $\E \nsq{\bar{\g}(\x_t)- \g(\x_t)}$,
\begin{align*}
\E_t\nsq{\bar{\g}(\x_t)- \g(\x_t)} &= \E_t \nsq{\frac{1}{b}\sum_{t=1}^{b}{\nabla\psi_{i_t}(\x_t)} - \frac{1}{b}\sum_{t=1}^{b}{\nabla\psi_{i_t}(\T{\x})} + \g(\T{\x}) - \g(\x_t) }\\
& = \frac{1}{b^2}\E_t\nsq{\sum_{t=1}^{b}{\nabla\psi_{i_t}(\x_t)}-\sum_{t=1}^{b}{\nabla\psi_{i_t}(\T{\x})} + b\cdot\g(\T{\x}) - b\cdot\g(\x_t)}\\
& = \frac{1}{b^2} \E_t\nsq{\sum_{t=1}^{b}{(\nabla\psi_{i_t}(\x_t) -\nabla\psi_{i_t}(\T{\x}) + \g(\T{\x}) - \g(\x_t) )}}.
\end{align*}
Now as all the indices $i_t$ are chosen independently, we can write the variance of the sum as the sum of individual variances.
\begin{align}
\E_t \nsq{\bar{\g}_t- \g_t} &= \frac{1}{b^2}\sum_{t=1}^{b}{\E_t\nsq{(\nabla\psi_{i_t}(\x_t) -\nabla\psi_{i_t}(\T{\x}) + \g(\T{\x}) - \g(\x_t) )}} \nonumber\\
& = \frac{1}{b}\E_t\nsq{\nabla{\psi}_i(\x_t) - \nabla{\psi}_i(\T{\x}) +\g(\T{\x}) - \g(\x_t)} \nonumber \\
& \leq \frac{2\lambda}{b} (f(\T{\x}) - f(\x_t) - \langle \g(\x_t) , \T{\x} - \x_t \rangle), \label{var2}
\end{align}
where in the last inequality, we used (\ref{var1}). Since $f$ has continuous first and second-order derivatives, we can use the quadratic Taylor's expansion for $f$ around $\x_t$. For vectors $\ai$ and $\vi$, $\exists \ \theta \in [0,1]$ such that,
\begin{align*}
 f(\ai+\vi) = f(\ai) + \langle \nabla{f}(\ai), \vi \rangle + \frac{1}{2}\vi^\top\nabla^2{f}(\ai+\theta \vi)\vi.
\end{align*}
Let $\ai=\x_t$, $\vi=\T{\x} - \x_t$, we get,
\begin{align*}
f(\T{\x}) - f(\x_t) - \langle \nabla{f}(\x_t),(\T{\x}-\x_t) \rangle = \frac{1}{2}(\T{\x}-\x_t)^\top\nabla^2{f}(\x_t + \theta (\T{\x}-\x_t))(\T{\x}-\x_t).
\end{align*}
Using the assumption  $\T\x, \x_t \in \mathcal{U}_f(c\epsilon_0\eta)$, we have $\x_t + \theta (\T{\x}-\x_t) \in \mathcal{U}_f(c\epsilon_0\eta) $. Take $\z = \x_t + \theta (\T{\x}-\x_t)$, we have $\frac{1}{1+c\epsilon_0\eta}\cdot\Hi\preceq \nabla^2{f}(\z) \preceq (1+c\epsilon_0\eta)\cdot\Hi$,
\begin{align*}
\frac{1}{2}(\T{\x}-\x_t)^\top\nabla^2{f}(\x_t + \theta (\T{\x}-\x_t)) ^\top (\T{\x}-\x_t) &= \frac{1}{2}\nsq{\x_t-\T{\x}}_{\nabla^2{f}(\z)}\\
& \leq \frac{1}{2}(1+ c\epsilon_0\eta)\normHsq{\T{\x}-\x_t}.
\end{align*}
implying,
\begin{align*}
f(\T{\x}) - f(\x_t) - \langle \nabla{f}(\x_t),(\T{\x}-\x_t) \rangle \leq \frac{1}{2}(1+c\epsilon_0\eta)\normHsq{\T{\x}-\x_t}.
\end{align*}
Substitute the above relation in (\ref{var2}), we get,
\begin{align*}
\E_t\nsq{\bar{\g}_t-\g_t} \leq (1+c\epsilon_0\eta)\frac{\lambda}{b} \normHsq{\T{\x}-\x_t}.
\end{align*}
\ifdocenter
\else
\qed
\fi
\end{proof}

\subsection{Proof of Lemma \ref{lma2}.}{\label{plma2}}
\begin{proof}
For the sake of this proof we abuse the notation and write $\x_{\textbf{AN}}$ as $\x_{t+1}$. However we make clear that \texttt{Mb-SVRN} uses $\x_{t+1} = \x_t- \eta\hat{\Hi}^{-1}\bar{\g}_t$. In this proof we denote $\x_{t+1} = \x_{\textbf{AN}} = \x_t -\eta\hat{\Hi}^{-1}\g_t$,  $\Delta_{t+1} = \x_{\textbf{AN}} -\x^*$, and $\Delta_t = \x_t-\x^*$.
    \begin{align*}
        \Delta_{t+1} &= \Delta_t - \eta\hat{\Hi}^{-1}\g_t\\
        & = \Delta_t - \eta \hat{\Hi}^{-1}(\g_t-\g(\x^*))\\
        & = \Delta_t - \eta\hat{\Hi}^{-1}\int_{0}^{1}{\nabla^2{f}(\x^* + \theta\Delta_t)\Delta_td\theta}\\
        & = \Delta_t - \eta\hat{\Hi}^{-1}\bar{\Hi}\Delta_t\\
        \Rightarrow \norm{\Delta_{t+1}}_{\bar{\Hi}} &= \norm{(\I-\eta\hat{\Hi}^{-1}\bar{\Hi})\Delta_t}_{\bar{\Hi}},
    \end{align*}
where $\bar{\Hi} = \int_{0}^{1}{\nabla^2{f}(\x^* + \theta\Delta_t)d\theta} $. We upper bound $\norm{\Delta_{t+1}}_{\bar{\Hi}}$ as,
\begin{align*}
    \norm{\Delta_{t+1}}_{\bar{\Hi}} &= \norm{(\I-\eta\hat{\Hi}^{-1}\bar{\Hi})\Delta_t}_{\bar{\Hi}}\\
    & = \norm{\bar{\Hi}^{1/2}(\I-\eta\hat{\Hi}^{-1}\bar{\Hi})\Delta_t}\\
    & = \norm{(\I-\eta\bar{\Hi}^{1/2}\hat{\Hi}^{-1}\bar{\Hi}^{1/2})\bar{\Hi}^{1/2}\Delta_t}\\
    & \leq \norm{\I-\eta\bar{\Hi}^{1/2}\hat{\Hi}^{-1}\bar{\Hi}^{1/2}}\cdot\norm{\Delta_t}_{\bar{\Hi}}.
\end{align*}
As $\x_t \in \mathcal{U}_f(c\epsilon_0\eta)$, for any $0 \leq \theta \leq 1$, we have $\x^* + \theta\Delta_t \in \mathcal{U}_f(c\epsilon_0\eta) $, and therefore $\frac{1}{1+c\epsilon_0\eta}\Hi\preceq \bar{\Hi} \preceq (1+c\epsilon_0\eta)\Hi$. Also we have $\frac{1}{\sqrt{\alpha}}\Tilde\Hi \preceq \hat{\Hi} \preceq \sqrt{\alpha}\Tilde\Hi$ and $\frac{1}{1+\epsilon_0\eta}\Hi \preceq \Tilde\Hi \preceq (1+\epsilon_0\eta)\Hi$. Combining these positive semidefinite orderings for $\T\Hi$ and $\hat{\Hi}$ along with $1+\epsilon_0\eta <2$, we have $\frac{1}{2\sqrt{\alpha}}\Hi\preceq \hat{\Hi} \preceq 2\sqrt{\alpha}\Hi$. We get,
\begin{align*}
&\frac{1}{2\sqrt{\alpha}(1+c\epsilon_0\eta)}\bar{\Hi}^{-1} \preceq\hat{\Hi}^{-1} \preceq 2\sqrt{\alpha}(1+c\epsilon_0\eta)\bar{\Hi}^{-1},\\
&\frac{\eta}{2\sqrt{\alpha}(1+c\epsilon_0\eta)}\I \preceq \eta\bar{\Hi}^{1/2}\hat{\Hi}^{-1}\bar{\Hi}^{1/2} \preceq 2\eta\sqrt{\alpha}(1+c\epsilon_0\eta)\I,
\end{align*}
implying that,
\begin{align*}
    \norm{\I-\eta\bar{\Hi}^{1/2}\hat{\Hi}^{-1}\bar{\Hi}^{1/2}} \leq \max\{1-\frac{\eta}{2\sqrt{\alpha}(1+c\epsilon_0\eta)},2\eta\sqrt{\alpha}(1+c\epsilon_0\eta)-1\}.
\end{align*}
Since $\eta < \frac{1}{4\sqrt{\alpha}}$, and $\epsilon_0 < \frac{1}{8c\sqrt{\alpha}}$, maximum value would be $1-\frac{\eta}{2\sqrt{\alpha}(1+c\epsilon_0\eta)}$. We get,
\begin{align*}
     \norm{\I-\eta\bar{\Hi}^{1/2}\hat{\Hi}^{-1}\bar{\Hi}^{1/2}}
     \leq 1-\frac{\eta}{4\sqrt{\alpha}},
\end{align*}
and hence,
\begin{align*}
    \norm{\Delta_{t+1}}_{\bar{\Hi}} \leq \left(1-\frac{\eta}{4\sqrt{\alpha}}\right)\norm{\Delta_t}_{\bar{\Hi}}.
\end{align*}
Changing  $\norm{\cdot}_{\bar{\Hi}}$ to $\norm{\cdot}_{\Hi}$ we get,
\begin{align*}
    \normH{\Delta_{t+1}} &\leq \sqrt{1+c\epsilon_0\eta}\norm{\Delta_{t+1}}_{\bar{\Hi}} \leq \sqrt{1+c\epsilon_0\eta}\cdot \left(1-\frac{\eta}{4\sqrt{\alpha}}\right)\norm{\Delta_t}_{\bar{\Hi}} \leq (1+c\epsilon_0\eta)\cdot\left(1-\frac{\eta}{4\sqrt{\alpha}}\right)\norm{\Delta_t}_{\Hi}\\
    &\leq \left(1-\eta\left(\frac{1}{4\sqrt{\alpha}}-c\epsilon_0\right) - \frac{c\epsilon_0\eta^2}{4\sqrt{\alpha}}\right)\normH{\Delta_t}.
\end{align*}
Using $\epsilon_0 < \frac{1}{8c\sqrt{\alpha}}$, we conclude,
\begin{align*}
\normH{\Delta_{t+1}} \leq \left(1-\frac{\eta}{8\sqrt{\alpha}}\right)\normH{\Delta_{t}}.
\end{align*}

\end{proof}
\ifdocenter
\else
\qed
\fi

\subsection{Proof of Lemma \ref{lma3}.}{\label{plma3}}

In the proof, we will use a result from ~\cite{minsker2011some}, stated as the following,
\begin{lemma}[\textbf{Matrix Bernstein: Corollary 4.1 from~\cite{minsker2011some}}]
Let $\Y_1, \Y_2,..\Y_m \in \mathbb{C}^d$ be a sequence of random vectors such that $\E \Y_i=0$, $\norm{\Y_i} < U$ almost surely $\forall 1 \leq i \leq m$. Denote $\sigma^2 := \sum_{i=1}^{m}{\E \norm{\Y_i}^2}$. Then $\forall t^2 > \sigma^2 + \frac{tU}{3}$,
\begin{align*}
    \Pr \left( \norm{\sum_{i=1}^{m}\Y_i}_2 >t \right) \leq 28 \text{exp} \left[-\frac{t^2/2}{\sigma^2 + tU/3} \right].
\end{align*}
\end{lemma}
We now return to the proof of Lemma \ref{lma3}.
\begin{proof}
Define random vectors $\vi_i$, as $\vi_i =\nabla\psi_i(\x_t)-\nabla\psi_i(\T\x) + \g(\T\x)-\g(\x_t)$. Then $\E[\vi_i]=0$. Also,
\begin{align*}
        \nsq{\vi_i} &\leq 2 \nsq{\nabla\psi_i(\x_t)-\nabla\psi_i(\T\x)} + 2\nsq{\g(\x_t)-\g(\T\x)}\\
        & \leq 4\lambda \left(\psi_i(\T\x) - \psi_i(\x_t)-\langle\nabla\psi_i(\x_t),\T\x-\x_t \rangle\right) + 4\lambda \left(f(\T\x)-f(\x_t)-\langle \nabla{f}(\x_t),\T\x-\x_t\rangle\right)\\
        &\leq 2\lambda^2\nsq{\T\x-\x_t} + 2\lambda^2\nsq{\T\x-\x_t}\\
        & = 4\lambda^2\nsq{\T\x-\x_t}.
    \end{align*}
We also have the following upper bound on the variance of $\vi_i$:
\begin{align*}
\E[\nsq{\nabla{\psi}_i(\x_t) - \nabla{\psi}_i(\T{\x}) +\g(\T{\x}) - \g(\x_t)}]& = \E[\nsq{\nabla{\psi}_i(\x_t) - \nabla{\psi}_i(\T{\x}) - (\g(\x_t)-\g(\T{\x}))}] \nonumber \\
& \leq \E[\nsq{\nabla{\psi}_i(\x_t) - \nabla{\psi}_i(\T{\x})}] \nonumber\\
& \leq 2\lambda\cdot\E[(\psi_i(\T{\x}) - \psi_i(\x_t) - \langle \nabla{\psi}_i(\x_t) , \T{\x} - \x_t \rangle)] \nonumber\\
& = 2\lambda \cdot (f(\T{\x}) - f(\x_t) - \langle \g(\x_t) , \T{\x} - \x_t \rangle).
\end{align*}
Since $f$ is twice continuously differentiable, there exists a $\theta \in [0,1]$ such that,
\begin{align*}
f(\T{\x}) - f(\x_t) - \langle \g(\x_t),(\T{\x}-\x_t) \rangle = \frac{1}{2}(\T{\x}-\x_t)^\top\nabla^2{f}(\x_t + \theta (\T{\x}-\x_t)) ^\top (\T{\x}-\x_t).
\end{align*}
Using the assumption that $\x_t,\T\x \in \mathcal{U}_f(c\epsilon_0\eta)$, with $\z = \x_t + \theta (\T{\x}-\x_t)$ we have $\z_t \in \mathcal{U}_f(c\epsilon_0\eta)$. With $\epsilon_0 < \frac{1}{8c\sqrt{\alpha}}$, and $\eta < \frac{1}{4\sqrt{\alpha}}$,
\begin{align*}
\frac{1}{2}(\T{\x}-\x_t)^\top\nabla^2{f}(\x_t + \theta (\T{\x}-\x_t)) ^\top (\T{\x}-\x_t) &= \frac{1}{2}\nsq{\x_t-\T{\x}}_{\nabla^2{f}(\z)}\\
& \leq \frac{1}{2}(1+ c\epsilon_0\eta)\normHsq{\T{\x}-\x_t}\\
& \leq \normHsq{\T{\x}-\x_t}.
\end{align*}
We get,
\begin{align*}
    \E\nsq{\vi_i} \leq 2\lambda\normHsq{\T{\x}-\x_t}.
\end{align*}
So we take $U = 2\lambda \norm{\T\x-\x_t}$ and $\sigma^2 = 2b\lambda \normHsq{\T\x-\x_t}$.
Also note that, $U \leq 2\frac{\lambda}{\sqrt{\mu}}\normH{\T\x-\x_t}$. Look for $t$ such that,
\begin{align*}
    \exp\left(-\frac{t^2/2}{\sigma^2 + tU/3}\right) < \delta\frac{b^2}{n^2}
\end{align*}
\begin{align*}
    \iff&\frac{t^2/2}{\sigma^2 + tU/3} > 2\ln(n/b\delta)\\
    \iff &t^2 >4 \sigma^2\ln(n/b\delta) + \frac{4tU}{3}\ln(n/b\delta)\\
\iff &\frac{t^2}{2} + \frac{t^2}{2}>4 \sigma^2\ln(n/b\delta) + \frac{4tU}{3}\ln(n/b\delta).
\end{align*}
Now in the regime of $b< \frac{8}{9}\kappa$ we have $2\sigma\ln(n/b\delta) < \frac{4}{3}U\ln(n/b\delta)$. Consider $\Y_i=\vi_i$, $t=\frac{8}{3}U\ln(n/b\delta) = \frac{4}{3}.\frac{2\lambda}{\sqrt{\mu}}\normH{\T\x-\x_t} \ln(\frac{n}{b\delta})$. For this value of $t$, we get with probability $1-\delta\frac{b^2}{n^2}$,
\begin{align*}
     \norm{\sum_{i=1}^{b}{\Y_i}} \leq \frac{8}{3}.\frac{2\lambda}{\sqrt{\mu}}\normH{\T\x-\x_t} \ln(\frac{n}{b\delta}).
\end{align*}
This means with probability $1-\delta\frac{b^2}{n^2}$,
\begin{align*}
    \norm{\g(\x_t)-\bar{\g}(\x_t)} \leq \frac{16\lambda}{3b\sqrt{\mu}}\ln(n/b\delta)\normH{\x_t-\T\x}.
\end{align*}
Moreover, in the regime of $b \geq \frac{8}{9}\kappa$, we have $2\sigma\ln(n/b\delta) \geq \frac{4}{3}U\ln(n/b\delta)$. In this case set $t= 2\sqrt{2}\sigma\ln(n/b\delta)= 4\sqrt{b\lambda}\normH{\x_t-\T\x}\ln(n/b\delta)$. We conclude,
\begin{align*}
    \norm{\g(\x_t)-\bar{\g}(\x_t)} \leq \frac{4\sqrt{\lambda}}{\sqrt{b}}\ln(n/b\delta)\normH{\x_t-\T\x}.
\end{align*}
\end{proof}
\ifdocenter
\else
\qed
\fi

\subsection{Proof of Lemma \ref{MVT}.}{\label{proof_MVT}}
\begin{proof}
    Since $\g^*=0$, we have $\normH{\hat{\Hi}^{-1}\g_t} =\normH{\hat{\Hi}^{-1}(\g_t-\g^*)}$. Consider the following,
\begin{align*}
\normH{\hat{\Hi}^{-1}\g_t} =\normH{\hat{\Hi}^{-1}(\g_t-\g^*)}
    & = \normH{\hat{\Hi}^{-1}\int_{0}^{1}{\nabla^2{f}(\x^* + \theta(\x_t-\x^*))\Delta_t\cdot d\theta}}.
\end{align*}
Denoting $\bar{\Hi} =\int_{0}^{1}{\nabla^2{f}(\x^* + \theta(\x_t-\x^*))\cdot d\theta} $, we get,
\begin{align*}
     \normH{\hat{\Hi}^{-1}(\g_t-\g^*)}&=\normH{\hat{\Hi}^{-1}\bar{\Hi}\Delta_t}\\
    &=\norm{\Hi^{1/2}\hat{\Hi}^{-1}\bar{\Hi}\Delta_t}\\
    &=\norm{\Hi^{1/2}\hat{\Hi}^{-1}\bar{\Hi}\Hi^{-1/2}\Hi^{1/2}\Delta_t}\\
&\leq\norm{\Hi^{1/2}\hat{\Hi}^{-1}\bar{\Hi}\Hi^{-1/2}}\normH{\Delta_t}\\
    & =  \norm{\underbrace{\Hi^{1/2}\hat{\Hi}^{-1}\Hi^{1/2}}\underbrace{\Hi^{-1/2}\bar{\Hi}\Hi^{-1/2}}}\normH{\Delta_t}\\
    & \leq \norm{\Hi^{1/2}\hat{\Hi}^{-1}\Hi^{1/2}}\cdot\norm{\Hi^{-1/2}\bar{\Hi}\Hi^{-1/2}}\cdot\normH{\Delta_t}.
\end{align*}
Since we have $\hat{\Hi} \approx_{\sqrt{\alpha}} \T\Hi$ and $\T\x \in \mathcal{U}_f(\epsilon_0\eta)$, we have $\hat{\Hi} \approx_{\sqrt{\alpha}(1+\epsilon_0\eta)} \Hi$. Furthermore for  $\x_t \in \mathcal{U}_f(c\epsilon_0\eta)$, we have $\bar{\Hi}\approx_{1+c\epsilon_0\eta}\Hi$, because for all $\theta, \x^* + \theta(\x_t-\x^*) \in \mathcal{U}_f(c\epsilon_0\eta) $. So we use the results $\norm{\Hi^{1/2}\hat{\Hi}^{-1}\Hi^{1/2}} \leq \sqrt{\alpha}(1+\epsilon_0\eta)$ and $\norm{\Hi^{-1/2}\bar{\Hi}\Hi^{-1/2}} \leq (1+c\epsilon_0\eta)$ and get,
\begin{align*}
   \normH{\hat{\Hi}^{-1}(\g_t-\g^*)}&\leq\sqrt{\alpha}(1+\epsilon_0\eta)(1+c\epsilon_0\eta)\normH{\Delta_t}\\
   &< 2\sqrt{\alpha}\normH{\Delta_t}.
\end{align*}
\end{proof}
\ifdocenter
\else
\qed
\fi

\subsection{Proof of Theorem \ref{Freedman}.}{\label{Freedman_proof}}

    In our proof of Theorem \ref{Freedman}, we use a Master tail bound for adapted sequenced from \cite{tropp2011freedman}. Let $\mathcal{F}_k$ be a filtration and random process $(Y_k)_{k\geq 0}$ be $\mathcal{F}_k$ measurable. Also let another random process $V_k$ such that $V_k$ is $\mathcal{F}_{k-1}$ measurable. Consider the difference sequence for $k\geq 1$,
    \begin{align*}
        X_k = Y_{k}- Y_{k-1}.
    \end{align*}
    Also, assume the following relation holds for a function $g:(0,\infty) \rightarrow [0,\infty]$:
    \begin{align*}
        \E_{k-1}e^{\theta\X_k} \leq e^{g(\theta) V_k}.
    \end{align*}
    Then we have,
    \begin{theorem}[Master tail bound for adapted sequences \cite{tropp2011freedman}]{\label{master_tail_bound}}
        For all $\lambda,w \in \R$, we have,
        \begin{align*}
            \Pr\left(\exists k\geq 0: Y_k \geq Y_0 +\lambda \ \ and \ \ \sum_{i=1}^{k}{V_i} \leq \sigma^2\right)\leq \inf_{\theta>0}e^{-\lambda\theta +g(\theta)\sigma^2}.
        \end{align*}
    \end{theorem}
    In our proof, we also use the following Lemma from \cite{tropp2011freedman}. The original version of the result assumes that $\E[X]=0$, however, we show below that the proof also holds for the case when $\E[X] <0$.     \begin{lemma}[Freedman MGF,  Lemma 6.7 from \cite{tropp2012user}]{\label{Freedman_mgf}}
        Let $X$ be a random variable such that $\E X \leq 0$ and $X \leq R$ almost surely. Then for any $\theta >0$ and $h(R)=\frac{e^{\theta R} - \theta R-1}{R^2}$,
        \begin{align*}
            \E e^{\theta X} \leq e^{h(R)\E[X^2]}.
        \end{align*}
        
    \end{lemma}

    \begin{proof}
        Consider the Taylor series expansion of $e^{\theta x}$,
        \begin{align*}
            e^{\theta x} &= 1+x+\frac{x^2}{2!} +\cdots\\
            & = 1+x + x^2 h(x).
        \end{align*}
        Replace $x$ with the random variable $X$ and take expectation on both sides, we get,
        \begin{align*}
            \E e^{\theta X} \leq 1+ \E X + \E[X^2\cdot h(X)].
        \end{align*}
        On the second term use $\E X \leq 0$ and on the third term use $X\leq R$ to get $h(X) \leq h(R)$ almost surely, we get,
        \begin{align*}
            \E e^{\theta X} \leq 1+h(R)\cdot\E[X^2] \leq e^{h(R)\cdot\E[X^2]}.
        \end{align*}
        
    \end{proof}
    \ifdocenter
\else
\qed
\fi
    We now return to the proof of Theorem \ref{Freedman}.
\begin{proof}
    If $Y_k$ is a submartingale we have,
    \begin{align*}
        \E_{k-1}X_k \leq 0,
    \end{align*}
    and also we know that $X_k \leq R$. For $k \geq 1$, consider $V_k = \E_{k-1}(X_k^2)$. Clearly, $V_k$ is $\mathcal{F}_{k-1}$ measurable. We now establish the relation that $\E_{k-1}e^{\theta\X_k} \leq e^{g(\theta) V_k}$.
    For any $\theta>0$, consider a function $h(x):[0,\infty]\rightarrow [\frac{\theta^2}{2},\infty)$ defined  as follows,
        \begin{align*}
            h(x) = \frac{e^{\theta x}-\theta x-1}{x^2} \ ,and \ \ h(0) = \frac{\theta^2}{2}.
        \end{align*}
        It is easy to show that $h(x)$ is an increasing function of $x$.
    Using Lemma \ref{Freedman_mgf} we get $\E_{k-1}e^{\theta\X_k} \leq e^{g(\theta) V_k}$ for $g(\theta) =\frac{e^{\theta R} - \theta R-1}{R^2}$. Now we can use the Master tail bound Theorem \ref{master_tail_bound}. The only thing remaining to analyze is, $$\inf_{\theta>0}e^{\theta \lambda + g(\theta)\cdot\sigma^2}.$$
    Doing little calculus shows that, $\theta = \frac{1}{R}\ln\left(1+\frac{\lambda R}{\sigma^2}\right)$ minimizes $e^{\theta \lambda + g(\theta)\cdot\sigma^2}$ and the minimum value is $e^{-\frac{1}{2}\left(\frac{\lambda}{\sigma^2} + \frac{\lambda}{R}\right)}$. Finally observe that $e^{-\frac{1}{2}\left(\frac{\lambda}{\sigma^2} + \frac{\lambda}{R}\right)} \leq e^{-\frac{1}{4}\min\left(\frac{\lambda}{\sigma^2}, \frac{\lambda}{R}\right)}$. This completes the proof.
\end{proof}
\ifdocenter
\else
\qed
\fi

\subsection{Proof of Theorem \ref{global_convergence}.}{\label{p_global_convergence}}
\begin{proof}
    Using the $\lambda$-smoothness of $f$, we know that
    \begin{align*}
        f(\x_{t+1}) \leq f(\x_t) + \g_t^\top(\x_{t+1}-\x_t) + \frac{\lambda}{2}\nsq{\x_{t+1}-\x_t}.
    \end{align*}
    Substitute $\x_{t+1} = \x_t - \eta\hat{\Hi}^{-1}\bar{\g}_t$, we get,
    \begin{align*}
    f(\x_{t+1}) &\leq f(\x_t) -\eta\g_t^\top\hat{\Hi}^{-1}\bar{\g}_t + \eta^2\frac{\lambda}{2}\nsq{\hat{\Hi}^{-1}\bar{\g}_t}\\
    & = f(\x_t) -\eta\g_t^\top\hat{\Hi}^{-1}\bar{\g}_t + \eta^2\frac{\lambda}{2}\nsq{\hat{\Hi}^{-1}\left(\bar{\g}_t-\g_t +\g_t\right)}.
    \end{align*}
    Take total expectation $\E$ on both sides, by which we mean expectation conditioned only on known $\T\x$. We get, 
    \begin{align*}
        \E f(\x_{t+1}) &\leq \E \left( f(\x_t) -\eta\g_t^\top\hat{\Hi}^{-1}\bar{\g}_t + \eta^2\frac{\lambda}{2}\nsq{\hat{\Hi}^{-1}\left(\bar{\g}_t-\g_t +\g_t\right)}\right)\\
        & = \E \left(\E_t\left( f(\x_t) -\eta\g_t^\top\hat{\Hi}^{-1}\bar{\g}_t + \eta^2\frac{\lambda}{2}\nsq{\hat{\Hi}^{-1}\left(\bar{\g}_t-\g_t +\g_t\right)}\right)\right).
    \end{align*}
    Here in the last inequality, $\E_t$ means conditional expectation, conditioned on known $\x_t$. Analyzing the inner expectation $\E_t$, we get,
    \begin{align*}
        \E f(\x_{t+1}) &\leq \E\left(f(\x_t) -\eta\g_t^\top\hat{\Hi}^{-1}\g_t + \eta^2\frac{\lambda}{2}\E_t\nsq{\hat{\Hi}^{-1}\left(\bar{\g}_t-\g_t +\g_t\right)}\right),
    \end{align*}
    where we obtained that second term due to the unbiasedness of stochastic gradient i.e., $\E_t[\bar{\g_t}] = \g_t$. Furthermore using unbiasedness we have $\E_t\nsq{\hat{\Hi}^{-1}\left(\bar{\g}_t-\g_t +\g_t\right)} = \E_t\nsq{\hat{\Hi}^{-1}\left(\bar{\g}_t-\g_t\right)} +\nsq{\hat{\Hi}^{-1}\g_t}$. We get,
    \begin{align*}
         \E f(\x_{t+1}) &\leq \E \left(f(\x_t) -\eta\g_t^\top\hat{\Hi}^{-1}\g_t +  \eta^2\frac{\lambda}{2}\nsq{\hat{\Hi}^{-1}\g_t}+\eta^2\frac{\lambda}{2}\E_t\nsq{\hat{\Hi}^{-1}\left(\bar{\g}_t-\g_t\right)}\right).
    \end{align*}
     Since $\nsq{\hat{\Hi}^{-1}\g_t} \leq \norm{\hat{\Hi}^{-1}}\cdot\nsq{\hat{\Hi}^{-1/2}\g_t} \leq\frac {\sqrt{\alpha}}{\mu}\nsq{\hat{\Hi}^{-1/2}\g_t}=\frac {\sqrt{\alpha}}{\mu}\g_t^\top\hat{\Hi}^{-1}\g_t$, where we used $\hat{\Hi} \approx_{\sqrt{\alpha}} \T\Hi$ to write $\norm{\hat{\Hi}^{-1}} \leq \frac{\sqrt{\alpha}}{\mu} $. Substituting we get,
     \begin{align*}
         \E f(\x_{t+1}) &\leq \E\left(f(\x_t) -\eta \left(1-\frac{\eta\kappa\sqrt{\alpha}}{2}\right)\g_t^\top\hat{\Hi}^{-1}\g_t + \eta^2\frac{\lambda}{2}\E_t\nsq{\hat{\Hi}^{-1}\left(\bar{\g}_t-\g_t\right)} \right).
     \end{align*}
     Again using $\hat{\Hi} \approx_{\sqrt{\alpha}} \T\Hi$, we get $\g_t^\top\hat{\Hi}^{-1}\g_t \geq \frac{1}{\lambda\sqrt{\alpha}}\nsq{\g_t}$.
     Also note that $\eta < \frac{2}{\kappa\sqrt{\alpha}}$ and hence $\left(1-\frac{\sqrt{\alpha}\eta\kappa}{2}\right)>0$. We get,
     \begin{align*}
          \E f(\x_{t+1}) &\leq \E \left(f(\x_t) -\frac{\eta}{\lambda\sqrt{\alpha}} \left(1-\frac{\eta\kappa\sqrt{\alpha}}{2}\right)\nsq{\g_t} + \eta^2\frac{\lambda}{2}\E_t\nsq{\hat{\Hi}^{-1}\left(\bar{\g}_t-\g_t\right)} \right).
     \end{align*}
     Now use $\E_t\nsq{\hat{\Hi}^{-1}\left(\bar{\g}_t-\g_t\right)} \leq \frac{\alpha}{\mu^2}\E_t\nsq{\bar{\g}_t-\g_t} \leq \frac{8\lambda\alpha}{b\mu^2}\left(f(\x_t)-f(\x^*) + f(\T\x) - f(\x^*)\right) $. Also due to $\mu$-strong convexity of $f$ we have $\nsq{\g_t} \geq 2\mu\left(f(\x_t)-f(\x^*)\right)$. We get,
     \begin{align*}
         \E f(\x_{t+1}) &\leq \E\left(f(\x_t) -\frac{2\eta\mu}{\lambda\sqrt{\alpha}} \left(1-\frac{\eta\kappa\sqrt{\alpha}}{2}\right)\left(f(\x_t)-f(\x^*)\right) + \eta^2\frac{4\lambda^2\alpha}{b\mu^2}\left(f(\x_t)-f(\x^*) + f(\T\x) - f(\x^*)\right)\right)\\
         & = \E\left(f(\x_t) -\frac{2\eta}{\kappa\sqrt{\alpha}} \left(1-\frac{\eta\kappa\sqrt{\alpha}}{2}\right)\left(f(\x_t)-f(\x^*)\right) + \eta^2\frac{4\kappa^2\alpha}{b}\left(f(\x_t)-f(\x^*) + f(\T\x) - f(\x^*)\right)\right).
     \end{align*}
     Subtract $f(\x^*)$ from both sides we get,
     \begin{align}
         \E f(\x_{t+1})-f(\x^*) &\leq \E\left(\underbrace{\left[1-\frac{2\eta}{\kappa\sqrt{\alpha}} + \eta^2\left(1+ \frac{4\kappa^2\alpha}{b}\right)\right]}_{\xib}\cdot \left(f(\x_t)-f(\x^*)\right) + \eta^2\frac{4\kappa^2\alpha}{b}\cdot\left(f(\T\x)-f(\x^*)\right) \right) \nonumber\\
         & = \xib \cdot\E \left(f(\x_t)-f(\x^*)\right) + \eta^2\frac{4\kappa^2\alpha}{b}\left( f(\T\x) - f(\x^*)\right) \label{global_exp_1}
     \end{align}
     We denote $\xib:= 1-\frac{2\eta}{\kappa\sqrt{\alpha}} + \eta^2\left(1+ \frac{4\kappa^2\alpha}{b}\right)$. Since we perform $\frac{n}{b}$ inner iterations before updating $\T\x$, we recursively unfold the relation (\ref{global_exp_1}) for $\frac{n}{b}$ times. This provides the following,
     \begin{align*}
         \E f(\x_{n/b}) - f(\x^*) &\leq \xib^{\frac{n}{b}} \left(f(\T\x) - f(\x^*)\right) + \eta^2\frac{4\kappa^2\alpha}{b}\cdot\left(1+\xib +\xib^2+\cdots+\xib^{\frac{n}{b}-1}\right) \left(f(\T\x) - f(\x^*)\right)\\
         & = \xib^{\frac{n}{b}} \left(f(\T\x) - f(\x^*)\right) + \eta^2\frac{4\kappa^2\alpha}{b}\cdot\left(\frac{1-\xib^{\frac{n}{b}}}{1-\xib}\right)\cdot\left(f(\T\x) - f(\x^*)\right).
     \end{align*}
     For $\eta < \frac{b}{5\kappa^3\alpha^{3/2}}$, we have $\xib < 1-\frac{\eta}{\kappa\sqrt{\alpha}}$. Substituting upper bound for $\xib$ we get,
     \begin{align*}
         \E f(\x_{n/b})- f(\x^*) \leq \left[\left(1-\frac{\eta}{\kappa\sqrt{\alpha}}\right)^{n/b} + \frac{4\eta\kappa^3\alpha^{3/2}}{b}\cdot\left(1-\left(1-\frac{\eta}{\kappa\sqrt{\alpha}}\right)^{n/b}\right)\right]\cdot\left(f(\T\x) - f(\x^*)\right).
     \end{align*}
    Let $\eta < \frac{b}{8\kappa^3\alpha^{3/2}}$, we get,
    \begin{align*}
        \E f(\x_{n/b})- f(\x^*) \leq \left(\frac{1}{2}\left(1-\frac{\eta}{\kappa\sqrt{\alpha}}\right)^{\frac{n}{b}} + \frac{1}{2}\right)\cdot\left(f(\T\x) - f(\x^*)\right).
    \end{align*}
    Now if $\left(1-\frac{\eta}{\kappa\sqrt{\alpha}}\right)^{\frac{n}{b}} < \frac{1}{2}$ we get,
    \begin{align*}
        \E f(\x_{n/b})- f(\x^*) <\frac{3}{4}\cdot\left(f(\T\x) - f(\x^*)\right),
    \end{align*}
    otherwise if $\left(1-\frac{\eta}{\kappa\sqrt{\alpha}}\right)^{\frac{n}{b}} \geq \frac{1}{2}$, then we consider two cases based on the value of $\eta$. If $\eta = \frac{2}{\kappa\sqrt{\alpha}}$ then we get,
    \begin{align*}
        \E f(\x_{n/b})- f(\x^*) < \left(1-\frac{1}{\alpha\kappa^2}\right)\cdot\left(f(\T\x) - f(\x^*)\right),
    \end{align*}
    and if $\eta = \frac{b}{8\kappa^3\alpha^{3/2}}$ then,
    \begin{align*}
        \E f(\x_{n/b})- f(\x^*) < \left(1- \frac{b}{16\kappa^4\alpha^2}\right)\cdot\left(f(\T\x) - f(\x^*)\right).
    \end{align*}
    This completes the proof.   
\end{proof}
\ifdocenter
\else
\qed
\fi

\section{Additional Plots for Synthetic Data Experiments}{\label{synthetic_experiments}}
For designing synthetic data, we considered two random orthonormal matrices $\U \in \R^{n\times d}$ and $\V \in \R^{d \times d}$ and constructed a matrix $\A = \U\Sigmab\V^\top$ where $\Sigmab \in \R^{d \times d}$ is a diagonal matrix with large condition number $(\kappa \approx 6000)$. In $\Sigmab$, we set $(n-1)$ diagonal entries to be $n$ and only one diagonal entry as $n\cdot\kappa$, artificially enforcing a very large condition number. For constructing labels, we set $b_i = \text{sign}(\ai_i^\top\x)$ where $\ai_i$ is the $i^{th}$ row of $\A$ and $\x \in \R^d$ is a random vector drawn from an isotropic multivariate Gaussian distribution. Also, we fix the regularization parameter $\mu$ to $10^{-2}$.
 
Our empirical results for the synthetic data are presented in  Figure \ref{fig_4}.  
In plot \ref{SVRG}, we observe the sensitivity of the peformance of \texttt{SVRG} to the gradient mini-batch size, and step size. Unlike \texttt{SVRG}, we empirically observe the robustness of \texttt{Mb-SVRN} to gradient mini-batch size. This is demonstrated clearly in our plots by an almost flat-line alignment of red dots for $b=10k,40k,$ and $100k$. Note that the optimal step size for \texttt{SVRG} is close to $2^{-6}$ for all gradient mini-batch sizes, and the slight deviation from this optimal step size leads to rapid deterioration in the convergence rate. On the other hand, in Figures~\ref{mbSVRN1}-\ref{mbSVRN3}, we see the convergence rate curves for different gradient mini-batch sizes exhibit a flatter behavior near their respective optimal step sizes, demonstrating the robustness of \texttt{Mb-SVRN}  to the step size. In addition to robustness, the plots also demonstrate the faster overall convergence rate of \texttt{Mb-SVRN}.

\begin{figure}[H]
\ifdocenter
\centering
\fi
\hspace{-2cm}
\begin{tabular}{cc}
\subfloat[\texttt{SVRG\label{SVRG}}]{\includegraphics[width = \ww]{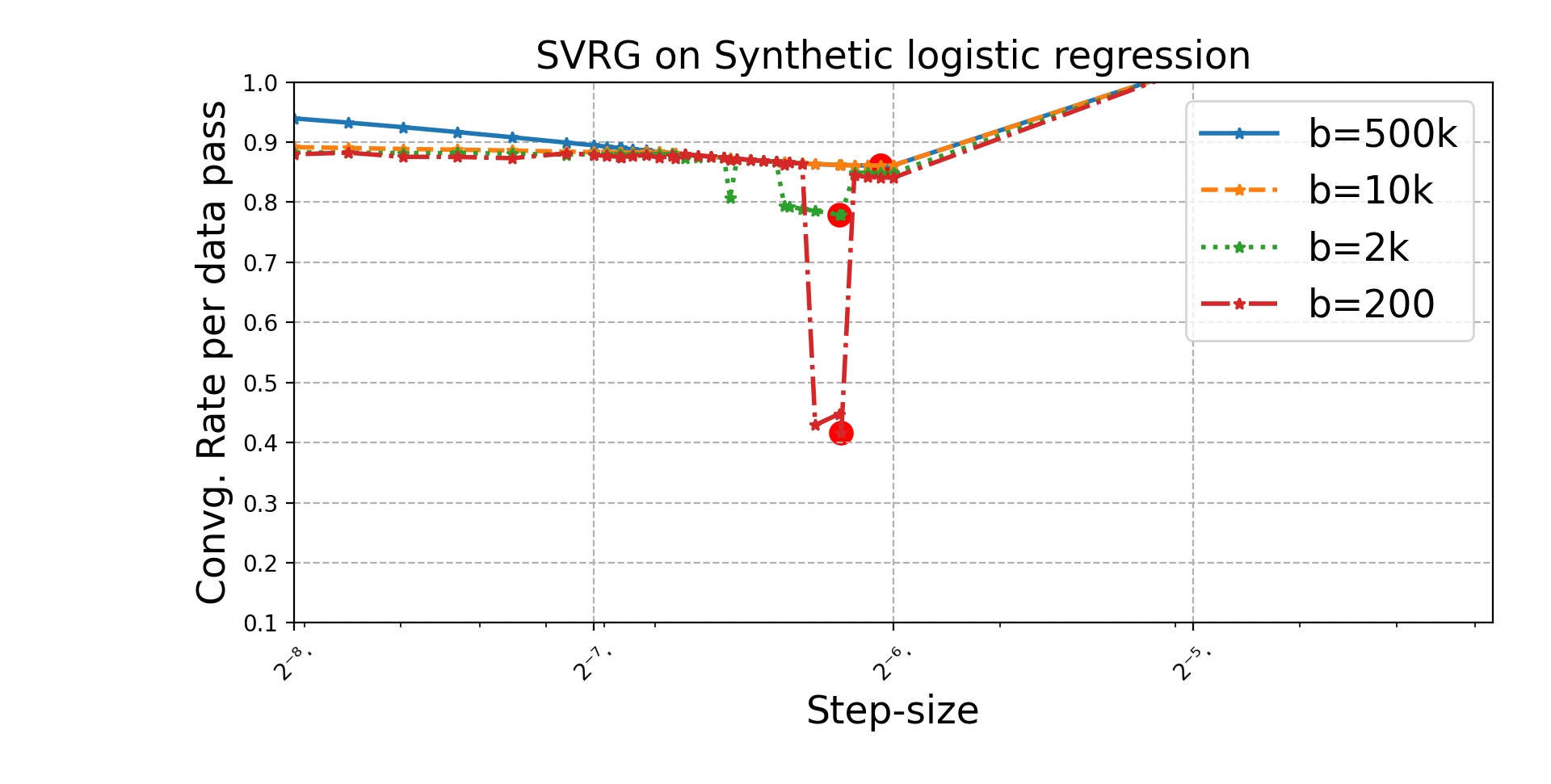}}&
\hspace{-8mm}\subfloat[\texttt{Mb-SVRN}, with $h=200$\label{mbSVRN1}]{\includegraphics[width = \ww]{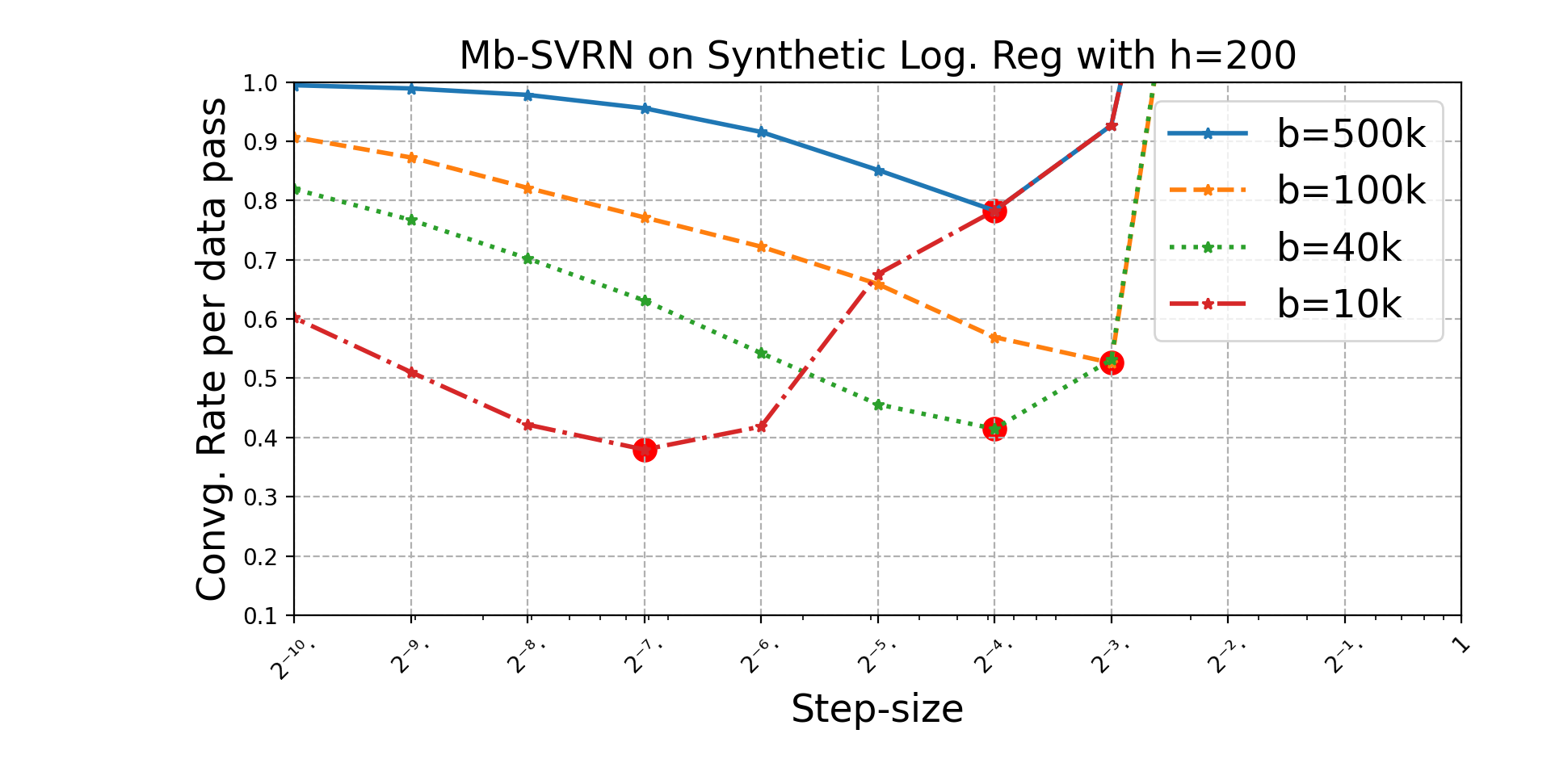}} \\
\subfloat[\texttt{Mb-SVRN}, with $h=1k$\label{mbSVRN2}]{\includegraphics[width = \ww]{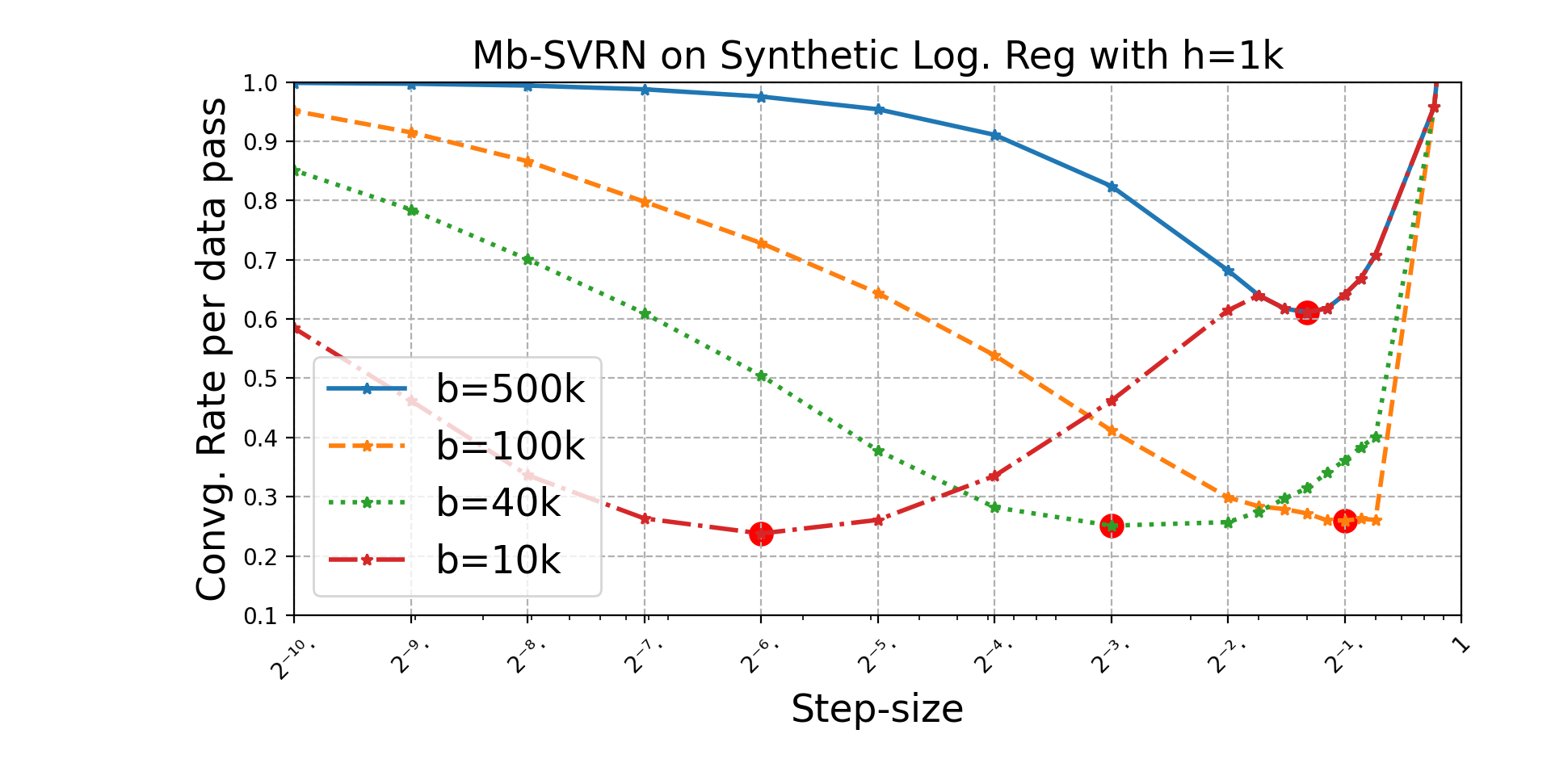}}&
\hspace{-8mm}\subfloat[\texttt{Mb-SVRN}, with $h=10k$\label{mbSVRN3}]{\includegraphics[width = \ww]{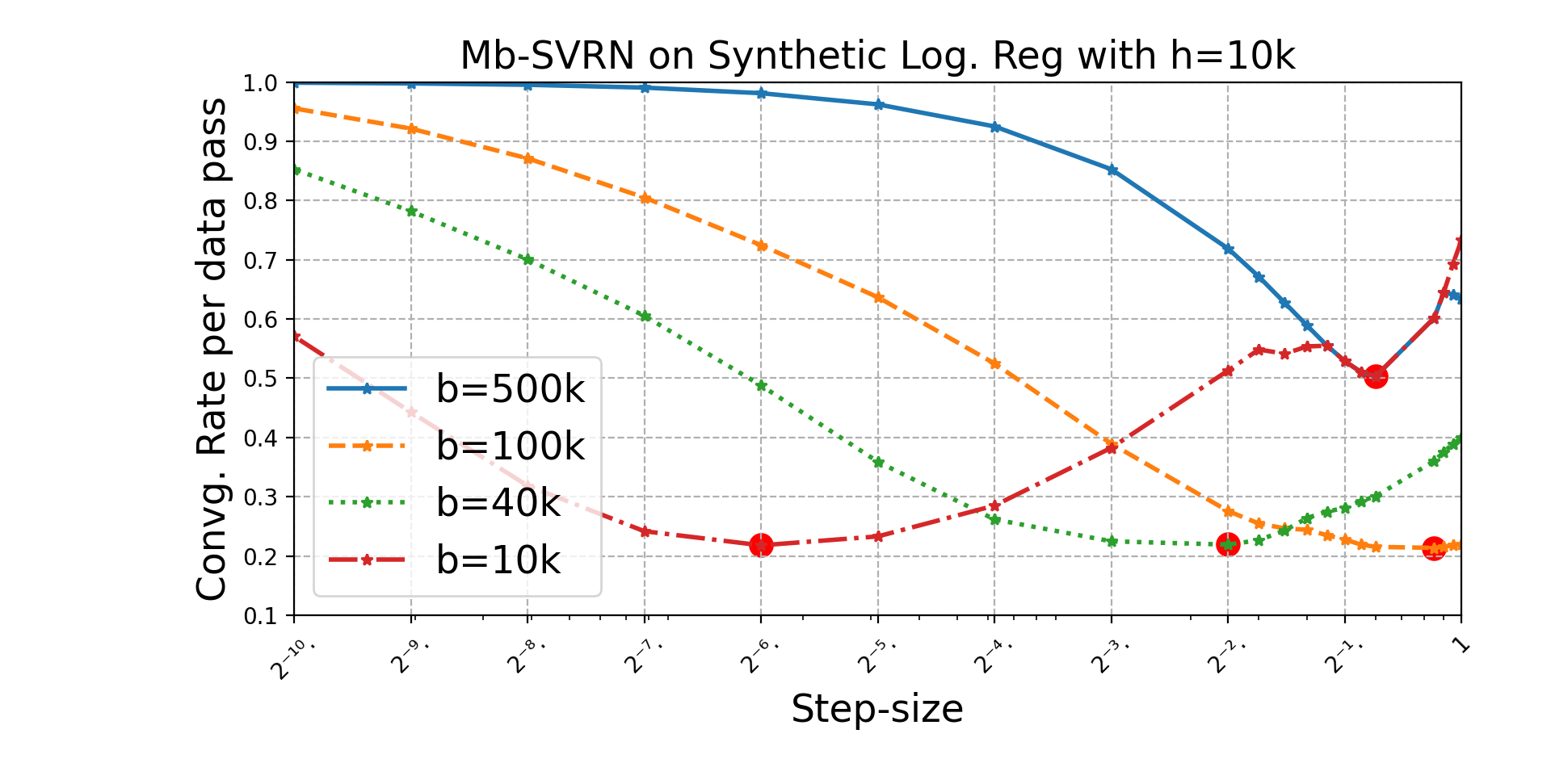}} 
\end{tabular}
\caption{Experiments on the synthetic logistic regression task. The red dots on every curve mark the respective optimal convergence rate attained at the optimal step size. The plot a) demonstrate the performance of \texttt{SVRG} for different gradient mini-batch sizes, whereas plots b) c) and d) demonstrate the performance of \texttt{Mb-SVRN} with different gradient and Hessian mini-batch sizes.}
\label{fig_4}
\end{figure}

\end{document}